\documentclass[
a4paper,reqno]{amsart}
\usepackage[T1]{fontenc}
\usepackage[english]{babel}
\usepackage{amssymb,bbm,enumerate}
\usepackage{color}
\usepackage[linktocpage=true,colorlinks=true, linkcolor=blue, citecolor=red, urlcolor=green]{hyperref}
\usepackage[all]{xy}

\usepackage{float}
\restylefloat{table}

\usepackage{tikz-cd}

\newcommand{\mc}[1]{\mathcal{#1}}
\newcommand{\mf}[1]{\mathfrak{#1}}
\newcommand{\mb}[1]{\mathbb{#1}}

\newcommand{\id}{\mathbbm{1}}
\newcommand{\quot}[2] {\ensuremath{\raisebox{.40ex}{\ensuremath{#1}}
\! \big / \! \raisebox{-.40ex}{\ensuremath{#2}}}}
\newcommand{\tint}{{\textstyle\int}}

\DeclareMathOperator{\Mat}{Mat}

\DeclareMathOperator{\End}{End}
\DeclareMathOperator{\Aut}{Aut}

\DeclareMathOperator{\Span}{Span}

\DeclareMathOperator{\mres}{mRes}

\DeclareMathOperator{\Cur}{Cur}
\DeclareMathOperator{\mgc}{mgc}

\theoremstyle{plain}
\newtheorem{theorem}{Theorem}[section]
\newtheorem{lemma}[theorem]{Lemma}
\newtheorem{proposition}[theorem]{Proposition}
\newtheorem{corollary}[theorem]{Corollary}

\theoremstyle{definition}
\newtheorem{definition}[theorem]{Definition}
\newtheorem{example}[theorem]{Example}

\theoremstyle{remark}
\newtheorem{remark}[theorem]{Remark}

\numberwithin{equation}{section}

\definecolor{light}{gray}{.9}

\begin{document}

\title[Non-local multiplicative PVA]{Local and non-local multiplicative Poisson vertex algebras
and differential-difference equations}

\author{Alberto De Sole}
\address{Dipartimento di Matematica, Universit\`a La Sapienza, 00185 Roma, Italy}
\email{desole@mat.uniroma1.it}
\urladdr{www1.mat.uniroma1.it/$\sim$desole}
\author{Victor G. Kac}
\address{Dept of Mathematics, MIT,
77 Massachusetts Avenue, Cambridge, MA 02139, USA}
\email{kac@math.mit.edu}
\author{Daniele Valeri}
\address{School of Mathematics and Statistics, University of Glasgow, G12 8QQ Glasgow, UK}
\email{daniele.valeri@glasgow.ac.uk}
\author{Minoru Wakimoto} 
\address{12-4 Karato-Rokkoudai, Kita-ku, Kobe 651-1334, Japan}
\email{wakimoto@r6.dion.ne.jp}



\begin{abstract}
We develop the notions of multiplicative Lie conformal and Poisson vertex algebras,
local and non-local, and their connections to the theory of integrable
differential-difference Hamiltonian equations.
We establish relations of these notions to $q$-deformed $W$-algebras
and lattice Poisson algebras.
We introduce the notion of Adler type pseudodifference
operators and apply them to integrability of differential-difference
Hamiltonian equations.
\end{abstract}

%

\maketitle

\tableofcontents

\section{Introduction}\label{sec:intro}

It has been demonstrated in a series of papers, 
\cite{BDSK09,DSK13,DSKVdirac,DSKV15,DSKV16,DSKV18}
to quote some of them,
that Poisson vertex algebras play as a fundamental role in the theory of Hamiltonian
integrable PDE,
as the Poisson algebras do in the theory of integrable Hamiltonian ODE.

Recall that a Poisson vertex algebra (PVA) is a unital commutative associative algebra $\mc V$
with a derivation $\partial$,
endowed with a Lie conformal algebra (LCA)  $\lambda$-bracket 
$$
\mc V\otimes\mc V\,\rightarrow\,\mc V[\lambda]
\,\,,\,\,\,\,
a\otimes b\mapsto\{a_\lambda b\}
\,,
$$
such that one has
\begin{enumerate}
\item[L]
(left Leibniz rule)
$\{a_\lambda bc\}
=
\{a_\lambda b\}c+b\{a_\lambda c\}$.
\end{enumerate}
Recall also the axioms of a LCA:
\begin{enumerate}
\item[A1]
(sesquilinearity)
$\{\partial a_\lambda b\}
=
-\lambda\{a_\lambda b\}$,
$\{a_\lambda\partial b\}
=
(\partial+\lambda)\{a_\lambda b\}$;
\item[A2]
(skewsymmetry)
$\{b_\lambda a\}
=
-_{\leftarrow}\{a_{-\partial-\lambda}b\}$;
\item[A3]
(Jacobi identity)
$\{a_\lambda\{b_\mu c\}\}-\{b_\mu\{a_\lambda c\}\}
=
\{\{a_\lambda b\}_{\lambda+\mu}c\}$.
\end{enumerate}

Note that PVA appears naturally as a quasiclassical limit of a vertex algebra,
hence the name.

For a non-local PVA the $\lambda$-brackets are allowed to take values in $\mc V((\lambda^{-1}))$,
the space of Laurent series in $\lambda^{-1}$,
and they are not quasiclassical limits of vertex algebras.
However they are indispensable for the theory of integrable Hamiltonian PDE \cite{DSK13}.
Note that one of the main sources of non-locality is the Dirac reduction,
which makes non-local even a local PVA \cite{DSKVdirac}.

Now, according to Kupershmidt's philosophy \cite{Kup85},
many ideas of the theory of integrable PDE should be extended to the theory
of integrable differential-difference equations.
In our recent paper \cite{DSKVW18}
we observed that, in order to extend the ideas of the PVA theory to the theory 
of integrable Hamiltonian differential-difference equations,
one is led to a ``multiplicative'' version of LCA and PVA.
This notion was derived in \cite{DSKVW18} from the notion 
of a $\Gamma$-conformal algebra \cite{GKK98} for the group $\Gamma=\mb Z$.

Note that, while the vertex algebras encode the operator product expansion 
of local fields along the diagonal, and the Lie conformal algebras encode its singular part,
the $\Gamma$-conformal algebras encode the singular part of the operator product expansion
off the diagonal when only simple poles are allowed.

Recall \cite{DSKVW18} that a multiplicative PVA is a unital commutative associative
algebra $\mc V$ with an automorphism $S$,
endowed with a multiplicative LCA $\lambda$-bracket
$$
\mc V\otimes\mc V\,\rightarrow\,\mc V[\lambda,\lambda^{-1}]
\,\,,\,\,\,\,
a\otimes b\mapsto\{a_\lambda b\}
\,,
$$
such that the same left Leibniz rule L holds as in the ``additive'' case.
The axioms of a multiplicative LCA are multiplicative analogues of A1--A3:
\begin{enumerate}
\item[M1]
(sesquilinearity)
$\{S(a)_\lambda b\}
=
\lambda^{-1}\{a_\lambda b\}$,
$\{a_\lambda S(b)\}
=
\lambda S\{a_\lambda b\}$;
\item[M2]
(skewsymmetry)
$\{b_\lambda a\}
=
-_{\leftarrow}\{a_{\lambda^{-1}S^{-1}}b\}$;
\item[M3]
(Jacobi identity)
$\{a_\lambda\{b_\mu c\}\}-\{b_\mu\{a_\lambda c\}\}
=
\{\{a_\lambda b\}_{\lambda\mu}c\}$.
\end{enumerate}
Note that axioms L and M2 imply
\begin{enumerate}
\item[rL]
(right Leibniz rule)
$\{ab_\lambda c\}
=
\{a_{\lambda S}c\}_\to b+\{b_{\lambda S}c\}_\to a$.
\end{enumerate}
(As usual, the arrow indicates where $S$ should be moved.)

The non-local multiplicative PVA are indispensable for the theory of integrable
Hamiltonian differential-difference equations as well.
But, while in the ``additive'' PVA case
the $\lambda$-brackets could be allowed to take values only in the Laurent series,
the ``multiplicative'' $\lambda$-brackets can be any bilateral series in $\lambda$.
However, for the ``multiplicative'' Dirac reduction one needs the $\lambda$-brackets to be rational,
i.e. symbols of rational difference operators (see Theorem \ref{20180310:thm1}).

In \cite{GKK98} a correspondence between multiplicative LCAs and multiplicative $q$-local
formal distribution Lie algebras was established
(see also Theorem \ref{thm:mformal} of the present paper),
which is similar to that in the ``additive'' case \cite{Kac96}.
However, in the ``multiplicative'' case this is just one side of a medal.
The other side is a correspondence between multiplicative LCAs and local lattice Lie algebras
(see \cite{GKK98} and Proposition \ref{20180310:rem} of the present paper). The latter
is a Lie algebra $\mf g$ with an automorphism $S$ such that $[S^n(a),b]=0$
for all but finitely many $n\in\mb Z$ ($a,b\in\mf g$).

In the same spirit, the non-local $q$-deformations of $W$-algebras attached to $\mf{sl}_N$
of Frenkel and Reshetikhin \cite{FR96}
can be encoded by the non-local multiplicative PVA $\mc W_N$ (see Example \ref{ex:qdef} for $N=2$ and
\ref{eq:wak2}, \ref{eq:wak1} for general $N$),
and what is called the ``lattice analogue'' \cite{FR96,HI97}
is encoded by the same multiplicative PVA (see Example \ref{ex:lattice_sl2} for $N=2$)
as the corresponding non-local lattice Poisson algebra.

Note that, as in the additive case \cite{Kac96},
an important ingredient of the theory is the multiplicative calculus of formal distributions,
in particular the multiplicative formal Fourier transform,
which we naturally call the formal Mellin transform (see Section \ref{sec:1.3a}).

In our paper \cite{DSKVW18} we classified all (local) multiplicative PVA in one variable $u$
up to order 5, which provides a rather large list of examples.
In particular, applying the Lenard-Magri scheme to the simplest compatible pair from this classification,
we proved the integrability of the Volterra lattice:
$$
\frac{du}{dt}
=
u(S^{-1}-S)u
\,.
$$

The simplest example of a non-local multiplicative PVA in $u$ is $\mc W_2$, given by
$$
\{u_\lambda u\}
=
u\frac{\lambda S-1}{\lambda S+1}u
\,.
$$
This $\lambda$-bracket is compatible with $\{u_\lambda u\}=\lambda-\lambda^{-1}$.
Applying the Lenard-Magri scheme to this pair,
we prove integrability of the modified Volterra lattice (see Section \ref{sec:wak})
$$
\frac{dv}{dt}
=
v^2(S^{-1}-S)v
\,.
$$
More generally, in Section \ref{sec:4.3}, using the non-local multiplicative PVA $\mc W_N$ with $N>2$ we construct a bi-Hamiltonian equation
\eqref{20180807:eq2} in $n=N-1$ variables, and conjecture that it is integrable.

After developing the foundations of the theory in Sections \ref{sec:1}--\ref{sec:05},
we turn to the notion of an Adler type pseudodifference operator,
which is a ``multiplicative'' version of that introduced in \cite{DSKV15,DSKV16,DSKV18}.
Given a unital commutative associative algebra $\mc V$ with an automorphism $S$,
the algebra of pseudodifference operators $\mc V((S^{-1}))$
is defined by the relation
$$
S^n\circ f=S^n(f)S^n
\,\,,\,\,\,\,
n\in\mb Z,\,f\in\mc V
\,.
$$
An operator $L(S)\in\mc V((S^{-1}))$ is called of Adler type if the following identity holds
with respect to a multiplicative $\lambda$-bracket on $\mc V$
(i.e. satisfying axioms M1, L and rL):
\begin{equation}\label{intro:1.1}
\begin{array}{l}
\displaystyle{
\vphantom{\Big(}
\{L(z)_\lambda L(w)\}_2^L
=
L(w\lambda S)
\delta_+\big(\frac{w\lambda S}{z}\big)
L^*\big(\frac{\lambda}{z}\big)
-
L(z)
\delta_+\big(\frac{w\lambda S}{z}\big)
L(w)
} \\
\displaystyle{
\vphantom{\Big(}
-\frac12
\big(L(w\lambda S)+ L(w)\big)
\big(L^*\big(\frac{\lambda}{z}\big)- L(z)\big)
\,.}
\end{array}
\end{equation}
Here $\delta_+(z)=\sum_{n\geq0}z^n$,  $L(z)$ is the symbol of $L(S)$, and $L^*(S)$ stands for the adjoint operator,
defined by $(fS^n)^*=S^{-n}\circ f$.

We show that, as in the additive case,
identity \eqref{intro:1.1} implies that the subalgebra of $\mc V$ generated by the coefficients of $L(S)$
is a multiplicative PVA (Corollary \ref{20180312:cor}) and, moreover,
the hierarchy of difference equations of Lax type
\begin{equation}\label{intro:1.2}
\frac{dL(S)}{dt_n}
=
[(L(S)^n)_+,L(S)]
\,\,,\,\,\,\,
n=1,2,\dots,
\end{equation}
is compatible, and has conserved densities
\begin{equation}\label{intro:1.3}
h_p=-\frac1p \mres L(S)^p
\,\,,\,\,\,\,
p=1,2,\dots
\,,\,\,
h_0=0
\,.
\end{equation}
Hereafter $\mres$ stands for the coefficient of $S^0$.

In fact, following ideas from Oevel-Ragnisco \cite{OR89}
we introduce the notions of a $3$-Adler type pseudodifference operator
(see Definition \ref{def:3adler}),
from which identity \eqref{intro:1.1} is obtained by a degeneration.
We show that we again obtain a multiplicative PVA for which \eqref{intro:1.2}
is a hierarchy of compatible difference equations and \eqref{intro:1.3} are integrals of motion.
As a result, we obtain in Section \ref{sec:4x} a tri-Hamiltonian hierarchy of difference equations
$$
\frac{dL(z)}{dt_n}
=
\{\int h_{n-1},L(z)\}^L_3
=
\{\int h_{n},L(z)\}^L_2
=
\{\int h_{n+1},L(z)\}^L_1
\,\,,\,\,\,\,
n\in\mb Z
\,.
$$

One of the most important operators of Adler type is the ``generic'' pseudodifference operator of order $N$:
$$
\widetilde{L}_N(S)
=
\sum_{j\leq N}u_jS^j
\,,
$$
where $u_N,u_{N-1},\dots$
is the (infinite) set of generators of the difference polynomial algebra $\widetilde{\mc V}_{\infty,N}$.
Identity \eqref{intro:1.1} endows $\widetilde{\mc V}_{\infty,N}$ with a structure of a (local) multiplicative PVA,
and an integrable hierarchy of Hamiltonian differential-difference equations \eqref{intro:1.2}
on $L(S)=\widetilde{L}_N(S)$.
Applying the Dirac reduction (provided by Theorem \ref{20180310:thm1})
by the constraint $u_N=1$
to the multiplicative PVA
$\widetilde{\mc V}_{\infty,N}$,
we obtain the algebra $\mc V_{\infty,N}$ of difference polynomials in the variables $u_{N-1},\,u_{N-2},\dots$,
so that
$$
L_N(S)
=
S^N+u_{N-1}S^{N-1}+\dots+u_0+\dots
\,,
$$
satisfies the Dirac reduced identity  \eqref{eq:adler-dirac} of \eqref{intro:1.1}.
As a result, we obtain the following rational multiplicative PVA structure on $\mc V_{\infty,N}$:
\begin{equation}
\label{intro:1.4}
\begin{split}
\{{u_i}_\lambda u_j\}
&=\sum_{n=0}^{N-i}
\big(u_{j-n}(\lambda S)^{j-i-n}u_{i+n}
-u_{i+n}(\lambda S)^{n}u_{j-n}\big) 
\\
&+u_j\left((\lambda S)^N-(\lambda S)^j\right)
\left((\lambda S)^N-1\right)^{-1}\left(1-(\lambda S)^{-i}\right)u_i
\,,
\end{split}
\end{equation}
subject to
\begin{equation}
\label{intro:1.5}
u_N=1
\,\,,\,\,\,\,
u_j=0
\,\,\text{ for }\,\,
j>N
\,.
\end{equation}
Note that for $N=1$ the $\lambda$-bracket \eqref{intro:1.4} is local, 
hence we obtain an integrable hierarchy \eqref{intro:1.2} on $L(S)=L_N(S)$.
This is the discrete KP hierarchy,
studied in detail by Adler and van Moerbeke in \cite{AvM99}.

Next, assuming that $N\geq2$, consider the difference subalgebra $\mc V_N$ of $\mc V_{\infty,N}$
generated by $u_{N-1},\dots,u_1,u_0$.
It is clear from \eqref{intro:1.4} that the element $u_0$ is central,
hence we can further reduce by the difference ideal generated by $u_0-c$,
where $c$ is a constant.
As a result, we get the multiplicative $W$-algebra $\mc W_N$,
which is the algebra of difference polynomials in $u_1,\dots,u_{N-1}$,
with a family of multiplicative rational Poisson $\lambda$-brackets
$\{{u_i}_{\lambda}{u_j}\}
=
c\{{u_i}_{\lambda}{u_j}\}_1+\{{u_i}_{\lambda}{u_j}\}_2$,
where
$$
\{{u_i}_\lambda u_j\}_1
=
\big((\lambda S)^{-i}-\lambda^j\big) u_{i+j}
\,,
$$
and
$$
\{{u_i}_\lambda u_j\}_2
=
\Big(\sum_{n=1}^j-\sum_{n=i}^{i+j-1}\Big)
u_n(\lambda S)^{n-i}u_{i+j-n}
+u_j
\frac{\big((\lambda S)^N-(\lambda S)^j\big)\big(1-(\lambda S)^{-i}\big)}{(\lambda S)^N-1}
u_i
\,,
$$
subject to \eqref{intro:1.5} (see Section \ref{sec:5.2c}).
This Poisson structure (for $c=1$) corresponds to the $q$-deformed $W$-algebras of \cite{FR96}
mentioned above.

We find it remarkable that, though the multiplicative PVA $\mc W_N$ is non-local,
it contains a beautiful local multiplicative PVA, which we denote by $\mc A_N$,
in the same number of difference variables
(see Example \ref{ex:5.6} and Theorem \ref{thm:wakimoto}).
The corresponding local lattice Poisson algebras for $N=2$ and $3$ are the well-known 
Faddeev-Takhtajan-Volkov \cite{FT86}
and Belov-Chaltikian \cite{BC93} algebras, and for $N>3$ they are the more recently discovered Mari-Beffa-Wang algebras \cite{MBW13}. 

As we have mentioned above, the pair of compatible multiplicative Poisson structures for $\mc W_2$ leads to integrability of
the modified Volterra lattice via the Lenard-Magri scheme, while the pair for $\mc A_2$ leads to the integrability of the Volterra lattice
\cite{DSKVW18}.
Likewise, for any $N>2$ we get a bi-Hamiltonian differential-difference equation on $N-1$ functions, which is a multicomponent
generalization of the modified Volterra and Volterra lattices, see \eqref{20180807:eq2} and \eqref{20180814:eq1} respectively.
These equations have been discovered by Mari-Beffa and Wang in \cite{MBW13}.
We conjecture that in both cases the Lenard-Magri scheme can be infinitely extended, proving thereby integrability
of these lattices. Note that in both cases certain master symmetries are constructed in \cite{MBW13}.

We are planning to develop in the subsequent publications a theory of multiplicative $W$-algebras, 
attached to any simple Lie algebra, which will include these examples.

In the last Section \ref{sec:last} we discuss various reductions of the discrete KP hierarchy,
reproving thereby integrability of various Hamiltonian differential-difference equations,
like the Volterra lattice, $1$-dimensional Toda lattice,
Bogoyavlensky lattice.
In conclusion, we present the $2$-dimensional Toda lattice by Ueno and Takasaki \cite{UT84}
and the corresponding two compatible multiplicative PVA structures.
The corresponding local lattice Poisson algebra structures have been computed by Carlet \cite{Carlet2}.

We are grateful to Sylvain Carpentier, who pointed out to us that
the pair of Poisson structures from Example \ref{ex:lattice_sl2}
can be used to prove integrability of the modified Volterra lattice via the Lenard-Magri scheme.
The research was partially conducted during the authors' visits 
to MIT and the University of Rome La Sapienza.
We are grateful to these institutions for their kind hospitality.
The first author was partially supported by the national PRIN fund n. 2015ZWST2C$\_$001
and the University funds n. RM116154CB35DFD3 and RM11715C7FB74D63,
and the third author was supported by a Tshinghua University startup research grant when working in the Yau Mathematical Sciences Center.

Throughout the paper the base field $\mb F$ is a field of characteristic zero.

\section{Multiplicative Lie conformal algebras 
and multiplicative Poisson vertex algebras}\label{sec:1}

\subsection{Multiplicative Lie conformal algebras (mLCA)}

\begin{definition}\label{def:mLCA}
A \emph{multiplicative Lie conformal algebra} (mLCA) 
is a vector space $R$
endowed with an invertible endomorphism $S:\,R\to R$
and a linear (over $\mb F$) multiplicative $\lambda$-bracket
$$
\{\cdot\,_\lambda\,\cdot\}
\,:\,\,
R\otimes R\,\to\,R[\lambda,\lambda^{-1}]
\,\,,\,\,\,\,
a\otimes b\mapsto\{a_\lambda b\}
\,,
$$
satisfying the following axioms ($a,b,c\in R$):
\begin{enumerate}[(i)]
\item sesquilinearity: 
$\{Sa_\lambda b\}=\lambda^{-1}\{a_\lambda b\}$, $\{a_\lambda Sb\}=\lambda S\{a_\lambda b\}$,
\item 
skewsymmetry: $\{a_\lambda b\}=-\{b_{\lambda^{-1} S^{-1}}a\}$,
\item 
Jacobi identity: 
$\{a_{\lambda}\{b_{\mu}c\}\}-\{b_{\mu}\{a_{\lambda}c\}\}=\{\{a_{\lambda}b\}_{\lambda\mu}c\}$.
\end{enumerate}
In the RHS of skew-symmetry $S$ is moved to the left to act on coefficients.
Namely, if $\{b_\lambda a\}=\sum_nc_n\lambda^n$, then
$\{b_{\lambda^{-1}S^{-1}}a\}=\sum_nS^{-n}(c_n)\lambda^{-n}$.
\end{definition}
Note that, as a consequence of the sesquilinearity axioms, $S$ is an automorphism
of the $\lambda$-bracket.
The reader should not fail to notice that a multiplicative Lie conformal algebra 
is a multiplicative analogue of a Lie conformal algebra \cite{Kac96}.
This notion was derived in \cite{DSKVW18}
from the notion of a $\Gamma$-conformal algebra \cite{GKK98}
for $\Gamma=\mb Z$.
\begin{example}\label{ex:mLCA1}
Let $\mf g$ be a Lie algebra.
The \emph{current} mLCA is defined as
$$
\Cur\mf g=\mb F[S,S^{-1}]\otimes\mf g
\,,
$$
with $S$ acting by left multiplication on the first factor,
and with the multiplicative $\lambda$-bracket given by
$$
\{S^m\otimes a_\lambda S^n\otimes b\}=\lambda^{n-m}S^n\otimes [a,b]
\,\,,\,\,\,\,
a,b\in\mf g,\, m,n\in\mb Z
\,.
$$
In other words we extend by the sesquilinearity axioms
the Lie bracket of $\mf g$.
\end{example}
\begin{example}\label{ex:mLCA2}
The \emph{general} mLCA $\mgc_1$
is defined as the free module over the algebra $\mb F[S,S^{-1}]$
with generators $u_m$, $m\in\mb Z$,
and the $\lambda$-bracket on generators defined by
$$
\{{u_m}_\lambda{u_n}\}
=
(\lambda^{-m}S^{-m}-\lambda^n)u_{m+n}
\,,\,\,m,n\in\mb Z
\,,
$$
and extended to $\mgc_1$ by the sesquilinearity axioms.
It is shown in \cite{GKK98} that representations of an mLCA $R$ in the free
$\mb F[S,S^{-1}]$-module of rank $1$
correspond to homomorphisms $R\to\mgc_1$.
\end{example}
\begin{example}\label{ex:mLCA3}
Let $V$ be a vector space.
The \emph{general} mLCA over $V$ is defined as
$\mgc(V)=\mgc_1\otimes\End V$. For $A\in\End(V)$ and $n\in\mb Z$ denote $A_n=u_n\otimes A$.
The multiplicative $\lambda$-bracket on $\mgc(V)$ is given by
($A,B\in\End V$)
$$
\{{A_m}_\lambda B_n\}
=
\lambda^{-m}S^{-m}(AB)_{m+n}-\lambda^n (BA)_{m+n}
\,,\,\,m,n\in\mb Z
\,.
$$
For $N\in\mb Z_{\geq0}$, we denote $\mgc_N=\mgc(V)$, where $V$ is an $N$-dimensional vector space.
\end{example}
The following is the ``multiplicative analogue'' of the Key Lemma in \cite{Kac96}.
\begin{lemma}\label{ex:mkey}
Let $R$ be an mLCA,
and let $\int:\,R\to \bar R:=R/(S-1)R$ be the quotient map.
Then we have a well-defined Lie algebra bracket $\{\cdot\,,\,\cdot\}$ on $\bar R$ given by
\begin{equation}\label{eq:key1}
\{\tint a,\tint b\}=\tint \{a_\lambda b\}\big|_{\lambda=1}
\,.
\end{equation}
We also have a well-defined representation of this Lie algebra on $R$,
with the action
\begin{equation}\label{eq:key2}
\{\cdot\,,\,\cdot\}:\,\bar R\times R\to R
\,\,\,\,,\,\,\,\,\,\,\,\,
\{\tint a,b\}=\{a_\lambda b\}\big|_{\lambda=1}
\,,
\end{equation}
by derivations of the $\lambda$-bracket,
commuting with the action of $S$ on $R$.
\end{lemma}
\begin{proof}
Obvious.
\end{proof}

\subsection{Affinization of an mLCA}

Let $R$ be an mLCA
and let $A$ be a unital commutative associative algebra
with an automorphism $S_A$.
Then, in analogy with the LCA case \cite{Kac96}, 
we can construct a new mLCA,
called the \emph{affinization} of $R$,
as 
$$
\widetilde{R}=R\otimes A
\,,
$$
with the automorphism $\widetilde{S}=S\otimes S_A$
and the following multiplicative $\lambda$-bracket:
\begin{equation}\label{eq:affin}
\{a\otimes f\,_\lambda\,b\otimes g\}^{\sim}
=
\{a_{\lambda S_A^{\,1}}b\}\otimes f\cdot g
\,,
\end{equation}
where $S_A^{\,1}$ denotes $S_A$ acting on the first factor.
Explicitly, if $\{a_\lambda b\}=\sum_nc_n\lambda^n$,
then the RHS of \eqref{eq:affin} is
$$
\sum_nc_n\otimes S_A^n(f)g\,\lambda^n
\,.
$$
\begin{proposition}\label{prop:affin}
The triple $(\widetilde{R},\widetilde{S},\{\cdot\,_\lambda\,\cdot\}^{\sim})$
defined above is an mLCA.
\end{proposition}
\begin{proof}
Straightforward.
\end{proof}

\subsection{Multiplicative calculus of formal distributions}
\label{sec:1.3a}

Recall that a \emph{formal distribution} in the variable $z$ with values in the vector space $\mf g$
is a formal bilateral series
\begin{equation}\label{eq:fdistr}
a(z)=\sum_{n\in\mb Z}a_nz^{-n}
\,\in\mf g[[z,z^{-1}]],\, a_n\in \mf g
\,.
\end{equation}
Similarly, a $\mf g$-valued formal distribution in two variables $z$ and $w$ is an element of 
$\mf g[[z,z^{-1},w,w^{-1}]]$.
An example of an $\mb F$-valued formal distribution is the \emph{multiplicative} $\delta$-function
\begin{equation}\label{eq:delta}
\delta(z)=\sum_{n\in\mb Z}z^n
\,.
\end{equation}
It has the following property:
\begin{equation}\label{eq:delta2}
a(z)\delta(z/w)=a(w)\delta(z/w)
\,\in \mf g[[z,z^{-1},w,w^{-1}]]\,,
\end{equation}
for every formal distribution $a(z)\in \mf g[[z,z^{-1}]]$.

The \emph{multiplicative residue} of the formal distribution \eqref{eq:fdistr}
is defined as
\begin{equation}\label{eq:mres}
\mres_z(a(z))=a_0
\,\in \mf g\,.
\end{equation}
Hence, the \emph{Fourier modes} $a_n$, $n\in\mb Z$, of the formal distribution \eqref{eq:fdistr}
can be obtained as 
$$
a_n=\mres(a(z)z^n)
\,.
$$
Likewise, for a pseudodifference operator $a(S)=\sum_n a_n S^n\in\mc V((S^{\pm1}))$, we define its \emph{multiplicative residue} by
\begin{equation}\label{20180203:eq3}
  \mres_S a(S) \,\,\Big(\!=  \mres_z a(z)\Big)\,\, =a_0 \,.
\end{equation}
It immediately follows from \eqref{eq:delta2} that
\begin{equation}\label{eq:delta3}
\mres_z a(z)\delta(z/w)=a(w)
\,\,\text{ for every }\,\,
a(z)\in \mf g[[z,z^{-1}]]
\,.
\end{equation}
By taking non-negative (resp. negative) powers of $w$ in both sides of equation
\eqref{eq:delta3} we also have
\begin{equation}\label{20180203:eq2}
  \mres_z a(z)\left(1-\frac w z \right)^{-1}=a(w)_+ \,\,\Big( \text{resp. } \mres_z a(z) \frac zw \left(1-\frac zw \right)^{-1}=a(w)_- \Big) \,.
\end{equation}

The multiplicative $\delta$-function splits as sum of its
positive and negative parts:
\begin{equation}\label{20180312:eq4}
\delta(z)=\delta_+(z)+\delta_-(z)
\,,
\end{equation}
where
\begin{equation}\label{20180312:eq2}
\delta_+(z)
=
\sum_{n\geq0}z^n
\,\,\text{ and }\,\,
\delta_-(z)
=
\sum_{n\leq-1}z^n
=
\delta_+(\frac1z)-1
\,.
\end{equation}
Later we will use the following properties of the positive and negative $\delta$-functions.
\begin{lemma}\label{20180404:lem}
The following identities hold:
\begin{equation}\label{20180316:eq1}
\delta_+(x_1x_2)=\delta_+(x_1)\delta_+(x_1x_2)+\delta_+(x_1x_2)\delta_+(x_2)-\delta_+(x_1)\delta_+(x_2)
\,,
\end{equation}
in $\mb F[[x_1,x_2]]$, and
\begin{equation}\label{20180316:eq2}
\delta(x_1)\delta_+(x_1x_2)
=\delta_+(x_1)\delta_+(x_2)-\delta_+(x_1x_2)\delta_+(x_2)+\delta_+(\frac{1}{x_1})\delta_+(x_1x_2)
\,,
\end{equation}
in $(\mb F[[x_1,x_1^{-1}]])[[x_2]]$.
\end{lemma}
\begin{proof}
Equation \eqref{20180316:eq1} turns into an obvious identity of rational functions,
by substituting $\delta_+(x)=\frac1{1-x}$.
Equation \eqref{20180316:eq2} is obtained
from equation \eqref{20180316:eq1} using
 \eqref{20180312:eq4} and \eqref{20180312:eq2}.
\end{proof}

%
%

Keeping in mind the mLCA,
we fix a non-zero element $q\in\mb F$
which is not a root of unity.
We say that a formal distribution in two variables $a(z,w)$ is $q$-local if
\begin{equation}\label{eq:qlocal}
\prod_{n\in T}(z-q^nw)a(z,w)=0
\,\,\text{ for some finite subset }\,\, T\subset\mb Z
\,.
\end{equation}
A formal distribution in three (or more) variables $a(z,w,x)$
is called $q$-local if it is $q$-local for each pair of variables.
Examples of $q$-local formal distributions in two variables are the $q$-\emph{shifted} 
$\delta$-functions:
$\delta(z/q^nw)$, $n\in\mb Z$,
where $\delta(\cdot)$ is as in \eqref{eq:delta}.
Indeed, it follows from \eqref{eq:delta2} that 
$$
(z-q^nw)\delta(z/q^nw)=0
\,.
$$
If $\mf g$ is a Lie algebra,
we say that a pair $(a(z),b(z))$ of $\mf g$-valued formal distributions in one variable
is $q$-\emph{local} if the Lie bracket
$[a(z),b(w)]$ is a $q$-local formal distribution in two variables.
\begin{lemma}[Multiplicative Decomposition Theorem, {\cite[Prop.1.1]{GKK98}}]\label{lem:decomp}
A formal distribution in two variables $a(z,w)$
is $q$-local if and only if it admits a decomposition into a finite sum
\begin{equation}\label{eq:qlocal2}
a(z,w)
=
\sum_{n\in\mb Z}
c_n(w)\delta(z/q^nw)
\,,
\end{equation}
for some $c_n(w)\in \mf g[[w,w^{-1}]]$.
In this case, the decomposition \eqref{eq:qlocal2} is unique and 
\begin{equation}\label{eq:qlocal3}
c_n(w)
=
\mres_z
\Big(\!\!
\prod_{i\in T\backslash\{n\}}\!\!\frac{z-q^iw}{q^nw-q^iw}\,\cdot\,a(z,w)
\Big)
\,,
\end{equation}
where $T\subset\mb Z$ is a finite subset for which \eqref{eq:qlocal} holds.
\end{lemma}
Note that formula \eqref{eq:qlocal3} follows by the following obvious identity:
$$
\mres_z
\Big(\!\!
\prod_{i\in T\backslash\{n\}}\!\!\frac{z-q^iw}{q^nw-q^iw}\,\cdot\,\delta(z/q^mw)\Big)
=\delta_{m,n}
\,\,\text{ for every }\,\,
m\in T,\,n\in\mb Z
\,.
$$
\begin{definition}\label{def:fourier}
The \emph{formal Mellin transform} 
is the linear map 
$$
\mc M^{n}_{z,w}:\,\mf g[[z,z^{-1},w,w^{-1}]]\to\mf g[[w,w^{-1}]]
\,\,,\,\,\,\,n\in\mb Z\,,
$$
defined by the following formula:
\begin{equation}\label{eq:mellin}
\mc M^{n}_{z,w}(a(z,w))
=
\mres_z\big(\frac zw\big)^n a(z,w)
\,.
\end{equation}
\end{definition}
\begin{proposition}\label{prop:mellin}
The formal Mellin transform satisfies the following properties:
\begin{enumerate}[(i)]
\item
$\mc M^n_{z,w}(\delta(z/q^jw))=q^{nj}$;
\item
$\mc M^n_{z,w} S_z=q^{-n}\mc M^n_{z,w}$,
and
$\mc M^n_{z,w} S_w=q^{n}S_w^n \mc M^n_{z,w}$,
where the operators $S_z$ and $S_w$ are given by
$S_z(a(z,w))=a(qz,w)$ and $S_w(a(z,w))=a(z,qw)$;
\item
if $a(z,w)=\sum_jc_j(w)\delta(z/q^jw)$ is a $q$-local formal distribution,
then 
$$
\mc M^n_{z,w}(a(z,w))=\sum_jc_j(w)q^{nj}
\,\,\text{ and }\,\,
\mc M^n_{z,w}(a(w,z))=\sum_jc_j(q^{-j}w)q^{-nj}
\,;
$$
\item
for every formal distribution in three variables $a(z,w,x)$,
we have
$$
\mc M^m_{z,x}\mc M^n_{w,x}(a(z,w,x))
=
\mc M^n_{w,x}\mc M^m_{z,x}(a(z,w,x))
=
\mc M^{m+n}_{w,x}\mc M^m_{z,w}(a(z,w,x))
\,.
$$
\end{enumerate}
\end{proposition}
\begin{proof}
Straightforward verification.
\end{proof}
By Lemma \ref{lem:decomp} and Proposition \ref{prop:mellin}(i),
we can define the $\lambda$-\emph{Mellin transform}
of a local formal distribution in two variables $a(z,w)$ as
\begin{equation}\label{eq:mellin2}
\mc M^{\lambda}_{z,w}(a(z,w))
=
\mc M^{n}_{z,w}(a(z,w))\big|_{q^n=\lambda}
\,\,\in
\mf g[[w,w^{-1}]][\lambda,\lambda^{-1}]\,.
\end{equation}
\begin{corollary}\label{cor:mellin}
For local formal distributions $a(z,w)$ and $a(z,w,x)$, we have
\begin{enumerate}[(i)]
\item
if $a(z,w)=\sum_jc_j(w)\delta(z/q^jw)$, then
$$
\mc M^{\lambda}_{z,w}(a(z,w))=\sum_jc_j(w)\lambda^j\,;
$$
\item
$\mc M^{\lambda}_{z,w} (S_za(z,w))=\lambda^{-1}\mc M^{\lambda}_{z,w} (a(z,w))$,
and
$\mc M^{\lambda}_{z,w} (S_w a(z,w))=\lambda S_w \mc M^{\lambda}_{z,w} (a(z,w))$,
where $S_z$ and $S_w$ are as in Proposition \ref{prop:mellin}(ii);
\item
$\mc M^{\lambda}_{z,w}(a(w,z))=\mc M^{\lambda^{-1}S_w^{-1}}_{z,w}(a(w,z))$,
(where $S_w$ is moved to the left);
\item
$\mc M^{\lambda}_{z,x}\mc M^{\mu}_{w,x}(a(z,w,x))
=
\mc M^{\mu}_{w,x}\mc M^{\lambda}_{z,x}(a(z,w,x))
=
\mc M^{\lambda\mu}_{w,x}\mc M^{\lambda}_{z,w}(a(z,w,x))$.
\end{enumerate}
\end{corollary}

\subsection{Multiplicative formal distribution Lie algebras and correspondence to mLCA}
\label{sec:1.3b}

The following notion is the ``multiplicative analogue''
of a regular formal distribution Lie algebra \cite{Kac96}.
\begin{definition}\label{def:mformal}
A \emph{multiplicative} $q$-\emph{local formal distribution Lie algebra}
is a pair $(\mf g,\mc R)$,
where $\mf g$ is a Lie algebra,
$\mc R\subset\mf g[[z,z^{-1}]]$ is a subspace 
such that:
\begin{enumerate}[(i)]
\item
$\mf g$ is the space of the Fourier modes of the formal distributions in $\mc R$;
\item
for $a(z)\in\mc R$ and $n\in\mb Z$, we have $a(q^nz)\in\mc R$;
\item
the formal distributions in $\mc R$ are pairwise $q$-local and,
in the decomposition of the commutator of $a(w),b(w)\in\mc R$  
in the finite sum (cf. \eqref{eq:qlocal2})
\begin{equation}\label{eq:qlocal2b}
[a(z),b(w)]
=
\sum_{n\in\mb Z}
c_n(w)\delta(z/q^nw)
\,,
\end{equation}
all the coefficients $c_n(w)$ lie in $\mc R$.
\end{enumerate}
An ideal $J\subset\mf g$ is called \emph{irregular}
if $\mc R\cap J[[z,z^{-1}]]=0$.
\end{definition}
\begin{remark}
The ``multiplicative analogue'' of Dong's Lemma does not seem to hold in general.
In fact, it is not hard to prove, by arguments similar to the ``additive'' case,
that if $a(w),b(w),c(w)$ are pairwise local and \eqref{eq:qlocal2b} holds,
then $[c_n(w),c(x)]$ decomposes as a finite combination of $q$-shifted $\delta$-functions
\emph{and their derivatives}.
\end{remark}
\begin{theorem}\label{thm:mformal}
\begin{enumerate}[(a)]
\item
If $(\mf g,\mc R)$ is a multiplicative $q$-local formal distribution Lie algebra,
then $\mc R$ has the structure of an mLCA, 
with $S:\,\mc R\to\mc R$ given by
$$
S(a(z))=a(qz)
\,,
$$
and multiplicative $\lambda$-bracket
$$
\{a(w)_\lambda b(w)\}
=
\sum_{n\in\mb Z}c_n(w)\lambda^n
\,,
$$
for $a(w),b(w)\in\mc R$, where the elements $c_n(w)$
are given by the decomposition \eqref{eq:qlocal2b}.
\item
Conversely, let $R$ be an mLCA, with automorphism $S$ 
and multiplicative $\lambda$-bracket $\{\cdot\,_\lambda\,\cdot\}$.
We obtain a multiplicative $q$-local formal distribution Lie algebra $(\mf g_R,\mc R_R)$ as follows:
$$
\mf g_R=R[t,t^{-1}]/\langle\mu_qS-1\rangle
\,\,,\,\,\,\,
\mc R_R=\big\{a(z)=\sum_{n\in\mb Z}at^n z^{-n-1}\,\big|\,a\in R\big\}
\,,
$$
where, the automorphism $S$ is extended to $R[t,t^{-1}]$ by commuting with the multiplication by $t$,
and $\mu_q:\,R[t,t^{-1}]\to R[t,t^{-1}]$ is defined by $\mu_q(at^n)=q^nat^n$,
for every $a\in R$, $n\in\mb Z$.
The Lie algebra structure of $\mf g$ is constructed, using the multiplicative $\lambda$-bracket of $R$,
as follows:
\begin{equation}\label{eq:formal-bracket}
[af(t),bg(t)]
=
\{a_{\mu_q^1}b\}
f(t)\cdot g(t)
\,,
\end{equation}
for every $a,b\in R$ and $f,g\in\mb F[t,t^{-1}]$.
In \eqref{eq:formal-bracket} $\mu_q^1$ is the map $\mu_q$ acting only on the first factor $f(t)$.
In other words, if $\{a_\lambda b\}=\sum_{n}c_n\lambda^n$, 
then the RHS of \eqref{eq:formal-bracket} is
$$
\sum_{n} c_n
f(q^nt) g(t)
\,.
$$
\item
Let $R$ be an mLCA.
Consider the corresponding multiplicative $q$-local 
formal distribution Lie algebra $(\mf g_R,\mc R_R)$
from part (b),
and the corresponding mLCA structure on $\mc R_R$ given by (a).
We have a canonical mLCA isomorphism $\mc R_R\simeq R$.
\item
Conversely, let $(\mf g,\mc R)$ be a multiplicative $q$-local formal distribution Lie algebra.
Consider the corresponding mLCA structure on $\mc R$ given by (a),
and then the corresponding multiplicative formal distribution Lie algebra
$(\mf g_{\mc R},\mc R_{\mc R})$ given by (b).
There is a canonical surjective Lie algebra homomorphism $\mf g_{\mc R}\twoheadrightarrow\mf g$,
whose kernel is an irregular ideal.
\end{enumerate}
\end{theorem}
\begin{proof}[{Proof of Theorem \ref{thm:mformal}}]
The proof of (a) follows by using Corollary \ref{cor:mellin} on the $\lambda$-Mellin transform.
For (b), the Lie algebra $\mf g_R$ is obtained, via Lemma \ref{ex:mkey},
from the affinization for $A=\mb F[t,t^{-1}]$ and $S_A=\mu_q$.
The proof of (c) and (d) is the same as in \cite{Kac96}.
\end{proof}
\begin{example}\label{ex:mLCA1b}
Consider the current mLCA $\Cur(\mf g)$ defined in Example \ref{ex:mLCA1}.
It is not hard to check that the corresponding multiplicative $q$-local
formal distribution Lie algebra given by Theorem \ref{thm:mformal}(b)
is the loop algebra
$\mf g[t,t^{-1}]$, with the Lie bracket
$$
[at^m,bt^n]=[a,b]t^{m+n}
\,\,,\,\,\,\,
m,n\in\mb Z\,,
$$
and the collection of pairwise $q$-local formal distributions
$$
\mc R=\Span\big\{
a(q^iz)=\sum_{n\in\mb Z}q^{-ni}at^nz^{-n}
\,\big|\,
a\in\mf g,\,i\in\mb Z
\big\}
\,.
$$
\end{example}
\begin{example}\label{ex:mLCA2b}
Consider the general mLCA $\mgc_1$ 
from Example \ref{ex:mLCA2}.
It is not hard to check that the corresponding multiplicative $q$-local
formal distribution Lie algebra given by Theorem \ref{thm:mformal}(b)
is the space
$$
\mf g=\bigoplus_{n\in\mb Z}\mb F[t,t^{-1}] u_n
\,,
$$
with the Lie algebra bracket
$$
[u_it^m,u_jt^n]
=
(q^{in}-q^{jm})u_{i+j}t^{m+n}
\,,\,\,
i,j,m,n\in\mb Z
\,,
$$
and the collection of pairwise $q$-local formal distributions
$$
\mc R=\Span\big\{
u_i(q^s z)=\sum_{n\in\mb Z}q^{-sn}u_it^nz^{-n}
\,\big|\,
i,s\in\mb Z
\big\}
\,.
$$
Note that the Lie algebra $\mf g$ constructed above is isomorphic
to the Lie algebra associated to the associative algebra 
of $q$-difference operators on the circle:
$$
\mb F[x,x^{-1}][\mu_q,\mu_q^{-1}]
\,,
$$ 
with the associative product defined by
the relation $\mu_q^nx^i=q^{in}x^i\mu_q^n$, $i,n\in\mb Z$.
The isomorphism is obtained by identifying $u_it^n\mapsto x^i\mu_q^n$, $i,n\in\mb Z$.
\end{example}
\begin{example}\label{ex:mLCA3b}
Consider the general mLCA $\mgc(V)$ 
from Example \ref{ex:mLCA3}. 
It is not hard to generalize the results in Example \ref{ex:mLCA2b} and check that the corresponding multiplicative $q$-local
formal distribution Lie algebra given by Theorem \ref{thm:mformal}(b)
is the space
$$
\mf g=\big(\bigoplus_{n\in\mb Z}\mb F[t,t^{-1}] u_n\big)\otimes\End(V)
\,.
$$
with the Lie algebra bracket ($A,B\in\mf g$)
$$
[A_it^m,B_jt^n]
=
q^{in}(AB)_{i+j}t^{m+n}-q^{jm}(BA)_{i+j}t^{m+n}
\,,\,\,
i,j,m,n\in\mb Z
\,,
$$
and the collection of pairwise $q$-local formal distributions is
$$
\mc R=\Span\big\{
A_i(q^s z)=\sum_{n\in\mb Z}q^{-sn}A_it^nz^{-n}
\,\big|\,
A\in\End(V)\,,i,s\in\mb Z
\big\}
\,.
$$
Note that the Lie algebra $\mf g$ constructed above is isomorphic
to the Lie algebra associated to the associative algebra 
of $\End(V)$-valued $q$-difference operators on the circle
$(\End V)[x,x^{-1}][\mu_q,\mu_q^{-1}]$.
\end{example}

\subsection{Local lattice Lie algebras and correspondence to mLCA}

We introduce here the notion of a local lattice Lie algebra,
which is equivalent to that of an mLCA.
\begin{definition}\label{def:lattice}
A \emph{lattice Lie algebra} 
is a Lie algebra $\mf g$ with an automorphism $S\in\Aut(\mf g)$.
It is called \emph{local} if,
for every $a,b\in\mf g$, we have
\begin{equation}\label{eq:local-lattice}
\{S^n(a),b\}=0
\,\,\text{ for all but finitely many values of }\,\, n\in\mb Z\,.
\end{equation}
\end{definition}
\begin{proposition}[{\cite{GKK98}}]\label{20180310:rem}
If $(R,S,\{\cdot\,_\lambda\,\cdot\})$ is an mLCA,
then $R$ is a local lattice Lie algebra 
with the automorphism $S$
and Lie bracket 
\begin{equation}\label{eq:2.15}
\{a,b\}=\mres_\lambda\{a_\lambda b\}
\,\,,\,\,\,\,
a,b\in\ R
\,,
\end{equation}
where the multiplicative residue is defined by \eqref{eq:mres}
Conversely,
if $(R,S)$ is a local lattice Lie algebra,
then we can endow it with a structure of an mLCA with the multiplicative $\lambda$-bracket
\begin{equation}\label{eq:2.16}
\{a_\lambda b\}=\sum_{n\in\mb Z}\{S^n(a),b\}\lambda^n
\,\,,\,\,\,\,
a,b\in\ R
\,.
\end{equation}
\end{proposition}
\begin{proof}
Straightforward.
\end{proof}
\begin{example}\label{ex:mLCA1c}
Consider the current mLCA $\Cur(\mf g)$ defined in Example \ref{ex:mLCA1}.
The corresponding local lattice Lie algebra is the space
$\Cur\mf g=\mb F[S,S^{-1}]\otimes\mf g$,
with the automorphism $S$ and the Lie bracket
$$
\{S^m\otimes a,S^n\otimes b\}=\delta_{m,n}S^n\otimes [a,b]
\,\,,\,\,\,\,
a,b\in\mf g,\,m,n\in\mb Z
\,.
$$
In other words, it is isomorphic to the direct sum of infinitely many copies 
of the Lie algebra $\mf g$,
and the automorphism $S$ is the ``shift'' operator.
\end{example}
\begin{example}\label{ex:mLCA2c}
Consider the general mLCA $\mgc_1$  from Example \ref{ex:mLCA2}.
The corresponding local lattice Lie algebra is the space
$\mgc_1=\bigoplus_{n\in\mb Z}\mb F[S,S^{-1}]u_n$
with the automorphism $S$ and the Lie bracket
$$
\{S^iu_m,S^ju_n\}
=
\delta_{j,i+m}S^iu_{m+n}
-
\delta_{i,j+n}S^ju_{m+n}
\,\,,\,\,\,\,
i,j,m,n\in\mb Z
\,.
$$
This lattice Lie algebra is isomorphic to the Lie algebra $\mf{gl}_\infty$
with the automorphism $S(E_{i,j})=E_{i+1,j+1}$,
via the isomorphism $S^iu_n\mapsto-E_{n+i,i}$.
\end{example}
\begin{example}\label{ex:mLCA3c}
Consider the general mLCA $\mgc(V)$  from Example \ref{ex:mLCA3}.
The corresponding local lattice Lie algebra is the space
$\mgc_V=\mgc_1\otimes\End(V)$
with the automorphism $S$ and the Lie bracket ($A,B\in\End(V)$)
$$
\{S^iA_m,S^jB_n\}
=
\delta_{j,i+m}S^i(AB)_{m+n}
-
\delta_{i,j+n}S^j(BA)_{m+n}
\,\,,\,\,\,\,
i,j,m,n\in\mb Z
\,.
$$
\end{example}
\begin{remark}\label{rem:exponent}
If $S$ is an automorphism of order $e\geq1$,
it seems natural to introduce the notion of an mLCA $R$
with $\lambda$-bracket 
$\{\cdot\,_\lambda\,\cdot\}:\,R\otimes R\to R[\lambda]/\langle\lambda^e-1\rangle$,
satisfying axioms (i)--(iii) of Definition \ref{def:mLCA}.
Then we still have Examples \ref{ex:mLCA1} and \ref{ex:mLCA2},
where $\mb F[S,S^{-1}]$ should be replaced by $\mb F[S]/\langle S^e-1\rangle$.
Furthermore, all results of this and the next section extend to this framework
with little changes.
For example, $q$ in Section \ref{sec:1.3b} should be a primitive $e$-th root of $1$.
These ``periodic'' mLCA should be useful in the study of periodic lattice equations.
\end{remark}

\section{Multiplicative Poisson vertex algebras and Hamiltonian differential-difference equations}
\label{sec:2}

\subsection{Multiplicative Poisson vertex algebras (mPVA)}

\begin{definition}\label{def:mPVA}
Let $\mc V$ be a unital commutative associative
algebra with an automorphism $S:\,\mc V\to\mc V$.
A \emph{multiplicative} $\lambda$-\emph{bracket} on $\mc V$
is a linear map 
$\{\cdot\,_\lambda\,\cdot\}:\,\mc V\otimes\mc V\to\mc V[\lambda,\lambda^{-1}]$
satisfying the sesquilinearity axioms (i) of Definition \ref{def:mLCA}
and the left and right Leibniz rules ($a,b,c\in\mc V$):
\begin{align}
& \{a_\lambda bc\}
=
\{a_\lambda b\}c+\{a_\lambda c\}b\, , 
\label{eq:l-leibniz}\\
& \{ab_\lambda c\}
=
\{a_{\lambda x}c\}\big(\big|_{x=S} b\big)+\{b_{\lambda x}c\}\big(\big|_{x=S} a\big)
\,.
\label{eq:r-leibniz}
\end{align}
Here and further we use the following notation:
for a polynomial (or a bilateral series) $a(z)=\sum_na_nz^n$
and $b,c\in\mc V$,
we let
\begin{equation}\label{eq:notation}
a(zx)\big(\big|_{x=S}b\big)c
=
\sum_na_nS^n(b)cz^n
\,.
\end{equation}
For example, the RHS of the skewsymmetry axiom in Definition \ref{def:mLCA}
can be written, using this notation, as $-\big(\big|_{x=S}\{b_{\lambda^{-1}x^{-1}}a\}\big)$.
A \emph{multiplicative Poisson vertex algera} (mPVA) is a 
unital commutative associative algebra with an automorphism $S:\,\mc V\to\mc V$
and a multiplicative $\lambda$-bracket
satisfying also the skewsymmetry and Jacobi identity axioms 
from Definition \ref{def:mLCA} of an mLCA.
\end{definition}
Note that the left and right Leibniz rules \eqref{eq:l-leibniz}-\eqref{eq:r-leibniz}
are equivalent, provided that the skewsymmetry axiom (ii) of Definition \ref{def:mLCA} holds.

A lattice Poisson algebra is defined as a Poisson algebra with an automorphism $S$,
and it is called local if condition \eqref{eq:local-lattice} holds.
In the same way as in the mLCA case,
there is a canonical bijective correspondence between mPVA
and local lattice Poisson algebras
(cf. Proposition \ref{20180310:rem}).


\subsection{Algebras of difference functions and multiplicative Poisson structures}
\label{sec:3.2}

In order to construct examples of mPVA,
consider the algebra of difference polynomials in $\ell$ variables
\begin{equation}\label{eq:rell}
\mc V_\ell=\mb F[u_{i,n}\,|\,i\in I,n\in\mb Z]
\,,
\end{equation}
where $I=\{1,\dots,\ell\}$,
with the automorphism $S$ defined by $S (u_{i,n})=u_{i,n+1}$.
Note that on $\mc V_\ell$ we have
\begin{equation}\label{20180202:eq1}
S\circ\frac{\partial}{\partial u_{i,n}}=\frac{\partial}{\partial u_{i,n+1}}\circ S\,.
\end{equation}
A multiplicative $\lambda$-bracket on $\mc V_\ell$ is introduced by letting 
(denote $u_i=u_{i,0}$)
\begin{equation}\label{20180310:eq3}
\{{u_i}_\lambda{u_j}\}
=
H_{ji}(\lambda)
=\sum_kh_{ji;k}\lambda^k
\,\in\mc V_\ell[\lambda,\lambda^{-1}]
\,\,,\,\,\,\,
i,j\in I\,,
\end{equation}
and extending (uniquely) to the whole space $\mc V_\ell$ by the sesquilinearity and Leibniz rules.
Then we have, for arbitrary $a,b\in\mc V_\ell$, the following Master Formula (cf. \cite{BDSK09}):
\begin{equation}\label{masterformula}
\{a_\lambda b\}=
\sum_{\substack{i,j\in I\\m,n,k\in\mb Z}}
\frac{\partial b}{\partial u_{j,n}}
S^n\Big(
h_{ji;k}S^{k-m} \big(\frac{\partial a}{\partial u_{i,m}}\big)
\Big)
\lambda^{n-m+k}
\,.
\end{equation}
\begin{definition}\label{def:alg-diff}
An \emph{algebra of difference functions} in the variables $u_i,\,i\in I$,
is a commutative associative algebra extension of $\mc V_\ell$,
with an automorphism extending $S$
and commuting derivations extending $\frac{\partial}{\partial u_{i,n}}$,
such that, for every $f\in\mc V$, 
\begin{equation}\label{eq:partial}
\frac{\partial f}{\partial u_{i,n}}=0
\,\,\text{ for all but finitely many }\,\, i,n\in\mb Z
\,,
\end{equation}
and satisfying the commutation relation \eqref{20180202:eq1}.
An element $c\in\mc V$ 
is called a \emph{constant} if it is fixed by $S$,
and it is called a \emph{quasiconstant}
if it is annihilated by all partial derivatives 
$\frac{\partial}{\partial u_{i,n}}$, $i,n\in\mb Z$.
Note that, as a consequence of \eqref{20180202:eq1} and \eqref{eq:partial},
the algebra of quasiconstants is $S$-invariant
and it contains the algebra of constants.
\end{definition}
One can construct an algebra of difference functions
by adding to $\mc V_\ell$
any smooth function $f=f(u_{i,n}\,|\,i\in I,n\in\mb Z)$ in finitely many of the variables $u_{i,n}$,
the shifted functions $S^k(f)=f(u_{i,n+k}\,|\,i\in I,n\in\mb Z)$
and all their partial derivatives of arbitrary order.
\begin{example}\label{rem:s+1}
An algebra of difference functions in one variable $u$
cannot contain a solution $f$ to the difference equation
\begin{equation}\label{eq:s+1}
(S-a)(f)=p(u)\,,
\end{equation}
where $a$ is a non-zero quasiconstant and $p(u)$ is a function of $u$ such that $p'(u)\neq0$.
Indeed, 
obviously $f$ cannot be a quasiconstant.
Let then $N$ and $M$ be respectively the largest and smallest integers
such that $\frac{\partial f}{\partial u_N}\neq0$ and $\frac{\partial f}{\partial u_M}\neq0$,
where $u_i$ stands for $S^i(u)$.
Then, applying $\frac{\partial}{\partial u_{N+1}}$ to both sides of \eqref{eq:s+1},
we get that $N+1=0$,
while applying $\frac{\partial}{\partial u_{M}}$ to both sides of \eqref{eq:s+1},
we get that $M=0$,
a contradiction since $N\geq M$.
\end{example}

In the same way as in \cite{BDSK09} for the case of PVA, 
one proves the following:
\begin{proposition}\label{prop:master}
Given an algebra of difference functions $\mc V$ 
and an $\ell\times\ell$ matrix 
$H(\lambda)=\big(H_{ij}(\lambda)\big)_{i,j=1}^\ell\in\Mat_{\ell\times\ell}\mc V[\lambda,\lambda^{-1}]$, where $H_{ij}(\lambda)=\sum_kh_{ij;k}\lambda^k$,
the multiplicative $\lambda$-bracket \eqref{masterformula}
defines a structure of an mPVA on $\mc V$ 
if and only if
skew-symmetry and the Jacobi identity hold on the generators $u_i$:
\begin{enumerate}[(i)]
\item
$\{{u_i}_\lambda{u_j}\}
=
-\{{u_j}_{\lambda^{-1}S^{-1}}{u_i}\}$,
\item
$\{{u_i}_\lambda\{{u_j}_\mu{u_k}\}\}
-\{{u_j}_\mu\{{u_i}_\lambda{u_k}\}\}
=\{\{{u_i}_\lambda{u_j}\}_{\lambda\mu}{u_k}\}$.
\end{enumerate}
In this case we call the matrix $H$
a \emph{multiplicative Poisson structure} on $\mc V$.
\end{proposition}
%

\begin{example}\label{20180309:ex1}
Let $p(\lambda)\in\mb F[\lambda,\lambda^{-1}]$ be a Laurent polynomial satisfying 
\begin{equation}\label{20180309:eq1}
p(\lambda^{-1})=-p(\lambda)
\,.
\end{equation}
Then, we have an mPVA structure on any algebra of difference functions $\mc V$
in one variable $u$, defined by
\begin{equation}\label{20180202:eq2}
\{u_\lambda u\}=p(\lambda)
\,.
\end{equation}
Indeed, skewsymmetry of the $\lambda$-bracket follows from the assumption \eqref{20180309:eq1},
while the Jacobi identity holds trivially, since $\{u_\lambda u\}$ is central.
\end{example}
\begin{example}\label{20180310:ex1}
As an application of Proposition \ref{prop:master},
if $R$ is an mLCA, the symmetric algebra over $R$
has a canonical structure of an mPVA.
\end{example}
\begin{example}\label{20180310:ex2}
Let $(\mc P,\{\cdot\,,\,\cdot\})$ be a Poisson algebra
and let $\mc V=\otimes_{n\in\mb Z}\mc P$
be the tensor product of $\mb Z$ copies of $\mc P$, 
where it is understood that a monomial in $\mc V$ 
has only finitely many factors different from $1$.
For $u\in\mc P$ we denote $u_n$ the monomial which has the factor $u$ in $n$-th place
and $1$ everywhere else.
Clearly, $\mc V$ is a Poisson algebra, being tensor product of Poisson algebras,
i.e. the commutative associative product is defined componentwise,
and the Poisson bracket is such that
($u,v\in\mc P,\,m,n\in\mb Z$)
$$
\{u_m,v_n\}=\delta_{m,n}\{u,v\}_n
\,,
$$
which defines a local lattice Poisson algebra,
with the automorphism 
$S:\,\mc V\to\mc V$ given by $S(u_n)=u_{n+1}$.
It is clearly local in the sense of Definition \ref{def:lattice}.
Hence, we have the corresponding multiplicative Poisson $\lambda$-bracket
on $\mc V$, defined by
($u,v\in\mc P,\,m,n\in\mb Z$):
$$
\{{u_m}_\lambda{v_n}\}
=
\{u,v\}_n\lambda^{n-m}
\,.
$$
\end{example}
\begin{example}\label{ex:vic}
Let $\mc V$ be an algebra of difference functions in one variable $u$,
and fix $f(u)\in\mc V$
(i.e. an element $f\in\mc V$ such that $\frac{\partial f}{\partial u_{n}}=0$
for $n\neq0$).
Then the formula 
\begin{equation}\label{eq:ff}
\{u_\lambda u\}=\sum_{j=1}^Nc_jf(u)(f(u_j)\lambda^j-f(u_{-j})\lambda^{-j})
\,\,,\,\,\,\,
\text{ where }\,\,
c_j 
\text{ are constants },
\end{equation}
defines a structure of an mPVA on $\mc V$, called in \cite{DSKVW18}
the multiplicative $\lambda$-bracket of general type. Hereafter $u_n=u_{1,n}$
in the case $\ell =1$.
\end{example}
\begin{example}\label{ex:wak-claim}
Let $\mc V$ be the field of fractions of difference polynomials in one variable $u$. Consider the following mPVA on $\mc V$
$$
\{u_\lambda u\}=\lambda ^n-\lambda^{-n}
\,,
$$
where $n=2m+1$ is an odd positive integer.
Let $v=(uS(u))^{-1}$. By a straightforward $\lambda$-bracket computation we get
\begin{equation}
\begin{split}\label{20180711:eq2}
\{v_{\lambda}v\}
&=v\big(1+\lambda S\big)v w \big(1+(\lambda S)^{-1}\big) (\lambda S)^n v
\\
&-v\big(1+\lambda S\big)(\lambda S)^{-n}v w \big(1+(\lambda S)^{-1}\big)v
\,,
\end{split}
\end{equation}
where
$$
w=\prod_{k=1}^m\frac{S^{2k}(v)}{S^{2k-1}(v)}
\,.
$$
Note that the RHS of \eqref{20180711:eq2} is equal to the 
multiplicative $\lambda$-bracket
denoted in \cite[Eq.s (1.11)-(1.12)]{DSKVW18} by $\{v_\lambda v\}_{n+1,v,-1}$,
which is a special case of the complementary type $\lambda$-bracket
for $\epsilon=-1$.
\end{example}

\subsection{Integrable hierarchies of Hamiltonian differential-difference equations}

Let $\mc V$ be an mPVA.
We call $\bar{\mc V}:=\quot{\mc V}{(S-1)\mc V}$ the space of \emph{Hamiltonian functionals},
and we denote by $\tint:\,\mc V\to\bar{\mc V}$ the canonical quotient map.
Recall from Lemma \ref{ex:mkey} that we have a Lie algebra bracket on $\bar{\mc V}$
given by \eqref{eq:key1},
and a representation of $\bar{\mc V}$ on $\mc V$ with the action given by \eqref{eq:key2}.
This action is by derivations of both the $\lambda$-bracket 
and the commutative associative product, 
and it commutes with the action of $S$.

\begin{definition}\label{def:hameq}
The \emph{Hamiltonian equation} associated 
to a Hamiltonian functional $\tint h\in\bar{\mc V}$
is, by definition,
\begin{equation}\label{20180202:eq3}
\frac{du}{dt}=\{\tint h,u\}
\,\,,\,\,\,\,
u\in\mc V\,.
\end{equation}
An \emph{integral of motion} for the Hamiltonian equation \eqref{20180202:eq3} is 
a Hamiltonian functional $\tint g\in\bar{\mc V}$
such that 
$$
\{\tint h,\tint g\}=0
\,.
$$
Equation \eqref{20180202:eq3} is called \emph{integrable}
if there are infinitely many linearly independent integrals of motion 
$\tint h_n,\,n\in\mb Z_{\geq0}$, with $h_0=h$,
which are in involution, i.e. such that 
$$
\{\tint h_m,\tint h_n\}=0
\,\,,\,\,\,\,
\text{ for all }\,\,
m,n\,.
$$
In this case, we have an integrable hierarchy of Hamiltonian equations
$$
\frac{du}{dt_n}=\{\tint h_n,u\}
\,\,,\,\,\,\,
u\in\mc V,\,n\in\mb Z_{\geq0}\,.
$$
\end{definition}
In the particular case of a multiplicative Poisson structure $H$ 
on an algebra of difference functions $\mc V$,
equation \eqref{20180202:eq3} becomes 
\begin{equation}\label{3.15}
\frac{d u}{dt}
=
H(S)\frac{\delta h}{\delta u}
\,,
\end{equation}
where $u=(u_i)_{i\in I}$,
and $\frac{\delta h}{\delta u}=(\frac{\delta h}{\delta u_i})_{i\in I}\in\mc V^{\oplus\ell}$ 
is the vector of variational derivatives
$$
\frac{\delta h}{\delta u_i}
=\sum_{n\in\mb Z}S^{-n}\Big(\frac{\partial h}{\partial u_{i,n}}\Big)
\,,\,\,
i=1,\dots,\ell
\,.
$$
Consequently, since the map, associating to $\int h \in \bar{\mc V}$ the derivation
$\{\int h, .\}$ of $\mc V$, is a Lie algebra homomorphism, if the operator $H(S)$ has finite-dimensional kernel and equation (\ref{3.15})
is integrable, then it has infinitely many linearly independent commuting symmetries.
\section{Non-local multiplicative Poisson vertex algebras}
\label{sec:4}

\subsection{Non-local mLCA and non-local mPVA}

\begin{definition}\label{def:nonlocal}
A \emph{non-local mLCA} is a vector space $\mc V$
with an invertible endomorphism $S:\,\mc V\to\mc V$
endowed with a non-local multiplicative $\lambda$-bracket,
$\{\cdot\,_\lambda\,\cdot\}:\,\mc V\otimes\mc V\to\mc V[[\lambda,\lambda^{-1}]]$
satisfying axioms (i)-(iii) of Definition \ref{def:mLCA}.
A \emph{non-local mPVA} is a unital commutative associative algebra $\mc V$
endowed with an automorphism $S:\,\mc V\to\mc V$
and a non-local mLCA  $\lambda$-bracket,
$\{\cdot\,_\lambda\,\cdot\}:\,\mc V\otimes\mc V\to\mc V[[\lambda,\lambda^{-1}]]$
satisfying the left Leibniz rule \eqref{eq:l-leibniz}
(or, equivalently, the right Leibniz rule \eqref{eq:r-leibniz}).
\end{definition}
Recall that $\mc V[[\lambda,\lambda^{-1}]]$
denotes the space of bilateral series $\sum_{n\in\mb Z}a_n\lambda^n$, 
where $a_n\in\mc V$ for all $n\in\mb Z$.
Thus, non-local mLCA or mPVA differ from local ones just in replacing
$\mc V[\lambda,\lambda^{-1}]$ by $\mc V[[\lambda,\lambda^{-1}]]$.
Note that in the non-local case all axioms still make perfect sense.
\begin{remark}\label{rem:additive}
Recall that in the ``additive'' case of non-local PVA's,
the $\lambda$-bracket cannot be a bilateral series in $\lambda,\lambda^{-1}$,
otherwise the skewsymmetry and Jacobi identity axioms would give divergent series.
As we have seen, this issue does not arise in the ``multiplicative'' case,
which, in this respect, seems to be much easier to deal with.
\end{remark}

Let $\mc V$ be an algebra of difference equations in the variables $u_i$, $i\in I$,
and let $H(\lambda)=(H_{ij}(\lambda))_{i,j\in I}\in\Mat_{\ell\times\ell}\mc V[[\lambda,\lambda^{-1}]]$
be a matrix valued bilateral series in $\lambda$ and $\lambda^{-1}$.
As in Section \ref{sec:3.2},
we can define a structure of a non-local multiplicative $\lambda$-bracket on $\mc V$
by letting the $\lambda$-bracket of $a,b\in\mc V$ 
be given by the Master Formula \eqref{masterformula},
which makes sense also for bilateral series.
One can check that Proposition \ref{prop:master}
still holds in the non-local case:
\begin{proposition}\label{prop:master-nl}
Given an algebra of difference functions $\mc V$ in $\ell$ variables $u_i$, $i\in I$,
and an $\ell\times\ell$ matrix 
$H(\lambda)=\big(H_{ij}(\lambda)\big)_{i,j=1}^\ell
\in\Mat_{\ell\times\ell}\mc V[[\lambda,\lambda^{-1}]]$,
the multiplicative $\lambda$-bracket \eqref{masterformula}
defines a structure of an mPVA on $\mc V$ 
if and only if
skew-symmetry and the Jacobi identity hold on the generators $u_i$.
In this case we call the matrix $H$
a \emph{non-local multiplicative Poisson structure} on $\mc V$.
\end{proposition}
%
%
\begin{example}\label{ex:nl-mPVA1}
If we replace in Example \ref{20180309:ex1} the Laurent polynomial $p(\lambda)$
by an arbitrary element of $\mb F[[\lambda,\lambda^{-1}]]$
satisfying condition \eqref{20180309:eq1},
formula \eqref{20180202:eq2}
gives a non-local mPVA structure on any algebra
of difference functions in one variable $u$.
\end{example}
\begin{example}\label{20180309:ex2}
We can generalize Example \ref{ex:vic} to the non-local setting as follows.
Let $\mc V$ be an algebra of difference functions in one variable $u$.
Let $f(u)\in\mc V$ be a function of the variable $u$ only
(i.e. $\frac{\partial f}{\partial u_n}=0$ for all $n\neq0$).
Let $r(\lambda)\in\mb F[[\lambda,\lambda^{-1}]]$ be a bilateral series
satisfying the condition:
\begin{equation}\label{eq:rational0}
r(\lambda^{-1})=-r(\lambda)
\,.
\end{equation}
For example, 
$r(\lambda)=\sum_{n\geq1}(\lambda^n-\lambda^{-n})$.
Define a multiplicative $\lambda$-bracket on $\mc V$
by letting
\begin{equation}\label{20180309:eq4}
\{u_\lambda u\}
=
f(u)\,r(\lambda S)f(u)
\,,
\end{equation}
and extending to $\mc V$ by the Master Formula \eqref{masterformula}.
The RHS of \eqref{20180309:eq4} has the obvious meaning:
if $r(\lambda)=\sum_nc_n\lambda^n$, then (cf. \eqref{eq:ff})
$$
f(u)r(\lambda S)f(u)=\sum_nc_nf(u)f(u_n)\lambda^n
\,.
$$
We claim that this defines a structure of non-local multiplicative PVA on $\mc V$.

First, it is immediate to check that the assumption \eqref{eq:rational0} 
implies the skewsymmetry condition $\{u_\lambda u\}=-\{u_{\lambda^{-1}S^{-1}}u\}$.
Let us check the Jacobi identity.
We have, by the sesquilinearity axioms and the left Leibniz rule
$$
\begin{array}{l}
\displaystyle{
\vphantom{\Big(}
\{u_\lambda{\{u_\mu u\}}\}
=
\big\{u_\lambda f(u)\,r(\mu S) f(u)\big\}
} \\
\displaystyle{
\vphantom{\Big(}
=
\{u_\lambda f(u)\}\,r(\mu S)f(u)
+
f(u)\,r(\lambda\mu S)\{u_\lambda f(u)\}
} \\
\displaystyle{
\vphantom{\Big(}
=
\frac{\partial f}{\partial u}f(u)
\big(r(\lambda S)f(u)\big)
\big(r(\mu S)f(u)\big)
+
f(u)\,r(\lambda\mu S)
\Big(
\frac{\partial f}{\partial u}f(u)\,r(\lambda S)f(u)
\Big)
\,.
}
\end{array}
$$
Hence,
$$
\{u_\lambda{\{u_\mu u\}}\}
-
\{u_\mu{\{u_\lambda u\}}\}
=
f(u)\,r(\lambda\mu S)
\big(
\frac{\partial f}{\partial u}f(u)\,r(\lambda S)f(u)
-
\frac{\partial f}{\partial u}f(u)\,r(\mu S)f(u)
\big)
\,.
$$
On the other hand,
by the sesquilinearity axioms and the right Leibniz rule,
we have, using the notation \eqref{eq:notation},
$$
\begin{array}{l}
\displaystyle{
\vphantom{\Big(}
\{{\{u_\lambda u\}}_{\lambda\mu} u\}
=
\big\{f(u)\,r(\lambda S)f(u)_{\lambda\mu}u\big\}
} \\
\displaystyle{
\vphantom{\Big(}
=
\{f(u)_{\lambda\mu x}u\}
\big(
\big|_{x=S}
r(\lambda S)f(u)
+
r(\mu^{-1}S^{-1})f(u)
\big)
} \\
\displaystyle{
\vphantom{\Big(}
=
f(u)
r(\lambda\mu S)
\Big(
\frac{\partial f}{\partial u}f(u)\,r(\lambda S)f(u)
+
\frac{\partial f}{\partial u}f(u)\,r(\mu^{-1}S^{-1})f(u)
\Big)
\,.}
\end{array}
$$
Hence, by \eqref{eq:rational0}, the Jacobi identity holds, and \eqref{20180309:eq4}
defines a structure of non-local mPVA on $\mc V$.

Note that the bilateral series $r(\lambda)$ satisfying \eqref{eq:rational0}
form a vector subspace of $\mb F[[\lambda,\lambda^{-1}]]$.
Hence, for a fixed function $f(u)$, 
all the multiplicative $\lambda$-brackets \eqref{20180309:eq4}
are compatible mPVA $\lambda$-brackets.
\end{example}
\begin{example}\label{ex:lattice_sl2}
Let $\mc V$ be an algebra of difference functions in one variable $u$
and consider the following two non-local mPVA $\lambda$-brackets on $\mc V$,
special cases of Examples \ref{ex:nl-mPVA1} and \ref{20180309:ex2} respectively:
\begin{equation}\label{eq:compatible}
\{u_\lambda u\}_1=p(\lambda)
\,\,\text{ and }\,\,
\{u_\lambda u\}_2=ur(\lambda S)u
\,,
\end{equation}
where $p(\lambda)=\lambda-\lambda^{-1}$ 
and $r(\lambda)\in\mb F[[\lambda,\lambda^{-1}]]$
satisfies the condition \eqref{eq:rational0}.
We can ask when these two structure are compatible,
in the sense that their sum is still a non-local mPVA $\lambda$-bracket on $\mc V$.
The compatibility condition reads, in this case,
$$
\{u_\lambda \{u_\mu u\}_2\}_1
-
\{u_\mu \{u_\lambda u\}_2\}_1
=
\{{\{u_\lambda u\}_2}_{\lambda\mu}u\}_1
\,.
$$
Expanding all three terms via the sesquilinearity axioms and the Leibniz rules,
we get the following equation on the bilateral series $r(\lambda)$:
\begin{equation}\label{eq:japan}
p(\lambda)r(\mu S)-p(\mu)r(\lambda S)
+(p(\lambda)-p(\mu))r(\lambda\mu)
=p(\lambda\mu S)(r(\lambda S)-r(\mu S))
\,.
\end{equation}
It is not hard to prove that, for $p(\lambda)=\lambda-\lambda^{-1}$,
there is a unique (up to a constant factor) solution of equation \eqref{eq:japan}:
\begin{equation}\label{eq:japan2}
r(\lambda)
=
\sum_{n\geq1}(-1)^n(\lambda^n-\lambda^{-n})
\,.
\end{equation}
This mPVA, denoted by $\mc W_2$, corresponds to the
\emph{classical lattice $W$-algebra} of $\mf{sl}_2$ \cite{HI97} 
via \eqref{eq:2.15}
(the classical lattice $W$-algebras for $\mf{gl}_N$ and
$\mf{sl}_N$, $N\geq2$, will be considered in Section \ref{sec:5}).
\end{example}

\subsection{The Lenard-Magri scheme for the modified Volterra equation (cf. \cite{Mik})}
  \label{sec:wak}

Recall that,
given two difference operators $K(S)$ and $H(S)$ in $\mc V[S,S^{-1}]$,
a Lenard-Magri sequence of length $n$ is a sequence of elements $\xi_j\in\mc V$,
where $j=0,1,2,\dots,n-1$, such that
\begin{equation}\label{eq:LM}
K(S)\xi_j=H(S)\xi_{j-1}
\,\,,\,\,\,\,
j=1,2,\dots,n-1
\,.
\end{equation}
\begin{proposition}
\label{prop:wak}
Let $\widetilde{K}(S)=S-S^{-1}$ and let
$\widetilde{H}(S)$ be a skewadjoint difference operator of the form 
$\widetilde H(S)=(S+1)\circ D(S)$,
for some difference operator $D(S)$.
Then any Lenard-Magri sequence of length $n\geq1$,
with $\xi_0=\frac12$ can be extended to a Lenard-Magri sequence of length $n+1$.
\end{proposition}
\begin{proof}
By induction on $n$.
First, we claim that $\widetilde H(S)(\xi_{n-1})\in(S-1)\mc V$.
Indeed, we have
\begin{align*}
\frac12&
\int \widetilde H(S)(\xi_{n-1})
=
\int \xi_0\widetilde H(S)(\xi_{n-1})
=
\int \xi_0\widetilde H(S)(\xi_{n-1})
\\
& +
\int \xi_1\widetilde H(S)(\xi_{n-2})
+\dots+
\int \xi_{n-1}\widetilde H(S)(\xi_{0})
\\
& -
\int \xi_1\widetilde K(S)(\xi_{n-1})
-\dots-
\int \xi_{n-1}\widetilde K(S)(\xi_{1})
\\
& =
\sum_{i=0}^{n-1}\int\xi_i\widetilde H(S)(\xi_{n-i})
-
\sum_{i=1}^{n-1}\int\xi_i\widetilde K(S)(\xi_{n-i})
=
0-0=0\,,
\end{align*}
since $\widetilde H(S)$ and $\widetilde K(S)$ are skewadjoint.
Hence, by the assumption on $\widetilde H(S)$, there exist $a,b\in\mc V$
such that
$$
\widetilde H(S)(\xi_{n-1})=(S-1)(a)=(S+1)(b)
\,.
$$
A solution for the Lenard-Magri recurrence relation \eqref{eq:LM} is then
$$
\xi_n
=
\frac12S(a-b)
\,.
$$
\end{proof}
\begin{remark}
  Proposition \ref{prop:wak} can be generalized as follows (proof is the same).
Let $K(S)=g(u)(S-S^{-1})\circ g(u)$ and let $H(S)=g(u)(S+1)\circ D(S)$ be a skewadjoint operator for some difference operator $D(S)$.
Then there exists an infinite Lenard-Magri sequence with $\xi_0=\frac{1}{2g(u)}$. Proposition 7.3 from \cite{DSKVW18} is a special case of this.
\end{remark}

Now consider the compatible pair of Poisson structures from Example \ref{ex:lattice_sl2}:
$$
K(S)=u\frac{S-1}{S+1}\circ u
\,\,,\,\,\,\,
H(S)=S-S^{-1}
\,.
$$
\begin{proposition}\label{prop:wak2}
There exists an infinite Lenard-Magri sequence
$\xi_0=\frac1u,\xi_1,\xi_2,\dots$ 
for the operators $K(S)$ and $H(S)$.
\end{proposition}
\begin{proof}
Relation \eqref{eq:LM} can be rewritten as
\begin{equation}\label{eq:wak3}
(S-1)u\xi_j
=
(S+1)\frac1u (S-S^{-1})\xi_{j-1}
\,.
\end{equation}
Letting $u\xi_j=(1+S^{-1})\omega_j$, equation \eqref{eq:wak3}
can be written as
\begin{equation}\label{eq:wak4}
\widetilde{K}(S)\omega_j
=
\widetilde{H}(S)\omega_{j-1}
\,,\,\,
\omega_0=\frac12
\,,
\end{equation}
where $\widetilde{K}(S)=S-S^{-1}$
and 
$$
\widetilde{H}(S)
=
(S+1)\circ \frac1u(S-S^{-1})\circ \frac1u(1+S^{-1})
\,.
$$
The claim follows from Proposition \ref{prop:wak}.
\end{proof}
Obviously $\xi_0=\frac1u$.
Next, it is easy to see that
$$
\xi_1
=
\frac1{u^2}\big(\frac1{u_1}+\frac1{u_{-1}}\big)
=
-\frac{\delta}{\delta u}\frac1{uu_1}
\,.
$$
Due to the general theorem 
(see e.g. \cite[Thm.6.20]{DSK13}, \cite[Thm.5.5]{DSKVW18}),
all the $\xi_i$ are variational derivatives: $\xi_i=\frac{\delta}{\delta u}\int h_i$.
It is easy to see that the first three conserved densities are
\begin{equation}\label{eq:wak5}
h_0=\log u
\,\,,\,\,\,\,
h_1=-\frac1{uu_1}
\,\,,\,\,\,\,
h_2=-\frac1{2u^2u_1^2}-\frac1{uu_1^2u_2}
\,.
\end{equation}
So we get an integrable hierarchy of Hamiltonian equations
$\frac{du}{dt_j}=H(S)(\xi_j)$, $j=0,1,2,\dots$.
(It is easy to show that they are linearly independent.)
The first two equations of this hierarchy are:
\begin{equation}\label{eq:wak6}
\frac{du}{dt_0}
=
\frac1{u_1}-\frac1{u_{-1}}
\,\,,\,\,\,\,
\frac{du}{dt_1}
=
\frac1{u_1^2}\big(\frac1u+\frac1{u_2}\big)
-
\frac1{u_{-1}^2}\big(\frac1u+\frac1{u_{-2}}\big)
\,.
\end{equation}
Note that, after the substitution $u=\frac1v$,
the first of these equations turns into the modified Volterra lattice
$$
\frac{dv}{dt_0}
=
v^2(v_{-1}-v_{1})
\,.
$$

Introduce the following Lax operator:
$L=S+\frac1u-\frac1{u_1}-\frac1{u^2}S^{-1}$.
Then the first equation in \eqref{eq:wak6} can be written 
in the Lax form
$\frac{dL}{dt_0}=[(L^2)_+,L]$
and the integrals of motion from \eqref{eq:wak5}
can be written as
$\tint h_1=\frac12\tint\mres(L^2)$, $\tint h_2=-\frac14\tint\mres(L^4)$.
We conjecture that the whole hierarchy has the Lax form 
$$
\frac{dL}{dt_j}=[(L^{2j+2})_+,L]
\,\,,\,\,\,\,
j=0,1,2,\dots
\,,
$$
and the integrals of motion are 
$$
\tint h_j=\frac{(-1)^{j+1}}{2^j}\tint\mres(L^{2j})
\,\,,\,\,\,\,
j=1,2,\dots
\,.
$$

\subsection{A bi-Hamiltonian equation in $n\geq2$ difference variables $u_1,\dots,u_n$}\label{sec:4.3}

Here we generalize Section \ref{sec:wak}, using a compatible pair of Poisson $\lambda$-brackets for the multiplicative $W$-algebra
$\mc W_N$ with $N=n+1\geq3$, constructed in Section \ref{sec:5.2c} below.
We obtain a bi-Hamiltonian differential-difference equation on $n$-variables as follows.
Let $K(S)$ and $H(S)$ be the $n\times n$ matrix difference operators, corresponding to the Poisson $\lambda$-brackets
\eqref{eq:wak2} and \eqref{eq:wak1}.
Let
$$
h_0=\log u_1\,,
\qquad
h_1=\frac{u_2}{u_1 S(u_1)}\,,
\qquad
\xi_i=\frac{\delta h_i}{\delta u}\,,
\quad i=0,1
\,.
$$
Then we have
\begin{equation}\label{20180807:eq1}
K(S)\xi_0=0\,,
\qquad
K(S)\xi_1=H(S)\xi_0
\,,
\end{equation}
hence we obtain the bi-Hamiltonian differential-difference equation $\frac{du}{dt_0}=H(S)\xi_0$.
Explicitly:
\begin{equation}\label{20180807:eq2}
\frac{du_j}{dt_0}=\frac{S^{-1}(u_{j+1})}{S^{-1}(u_1)}-\frac{u_{j+1}}{S^j(u_1)}\,,
\qquad
j=1,\dots,n\,,
\end{equation}
where $u_{n+1}=1$. Its first two conserved densities are $h_0$ and $h_1$.

We conjecture that the Lenard-Magri sequence \eqref{20180807:eq1} can be infinitely extended, hence,
by a general theorem as above, the equation \eqref{20180807:eq2} is integrable.
This equation have appeared earlier in \cite{MBW13}.
\section{Rational multiplicative Poisson vertex algebras}
\label{sec:05}

\subsection{Pseudodifference operators}\label{sec:1.1}

Let $\mc V$ be a unital commutative associative algebra with an automorphism $S$.
%
%
The algebra of \emph{scalar difference operators} over $\mc V$
is the space of Laurent polynomials $\mc V[S,S^{-1}]$,
with the associative product $\circ$ defined by the relation
$$
S\circ f=S(f) S
\,,
\qquad
f\in\mc V
\,.
$$
Hence, for $a(S)=\sum_ma_mS^m$ and $b(S)=\sum_nb_nS^n$ in $\mc V[S,S^{-1}]$,
their product is
$$
a(S) \circ b(S)
=
\sum_{m,n}
a_mS^m(b_n)S^{m+n}
\,.
$$

The algebra $\mc V[S,S^{-1}]$ naturally acts on $\mc V$:
the action of $a(S)=\sum_na_nS^n\in\mc V[S,S^{-1}]$ (finite sum)
on $f\in\mc V$ is 
\begin{equation}\label{20180308:eq1}
a(S)f=\sum_na_nS^n(f)
\,\in\mc V\,.
\end{equation}
(It should not be confused with the associative product
$a(S)\circ f=\sum_na_nS^n(f)S^n\in\mc V[S,S^{-1}]$.)

The algebra $\mc V[S,S^{-1}]$ is $\mb Z$-graded by the powers of $S$,
and it can be completed either in the positive or in the negative directions,
giving rise to two algebras of pseudodifference operators:
$\mc V((S))=\mc V[[S]][S^{-1}]$ and $\mc V((S^{-1}))=\mc V[[S^{-1}]][S]$.
Given a pseudodifference operator
$a(S)=\sum_na_nS^n\in\mc V((S^{\pm1}))$,
we define
its \emph{formal adjoint} as
$$
a^*(S)=\sum_nS^{-n}\circ a_n
\,\in\mc V((S^{\mp1}))
\,,
$$
its positive part as
$$
a(S)_+=\sum_{n\geq0}a_nS^n
\,,
$$
its \emph{negative part} as 
$$
a(S)_-=\sum_{n\leq-1}a_nS^n
\,,
$$
and its \emph{symbol} as 
$$
a(z)=\sum_na_nz^{n}\in\mc V((z^{\pm1}))
\,.
$$
(Note: here and further $\mc V((S^{\pm1}))$ stands for $\mc V((S))$ or $\mc V((S^{-1}))$
respectively, NOT for $\mc V((S,S^{-1}))$.)

The action \eqref{20180308:eq1} of $\mc V[S,S^{-1}]$ on $\mc V$
does not extend to an action of $\mc V((S^{\pm1}))$ on $\mc V$.
On the other hand, we have a $z$-\emph{action}
$$
\mc V((S^{\pm1}))\times\mc V\to\mc V((z^{\pm1}))
\,,
$$
mapping $a(S)\in\mc V((S^{\pm1}))$ and $f\in\mc V$ to
\begin{equation}\label{20180308:eq2}
a(zS)f
=
\sum_na_nS^n(f)z^n
\,\in\mc V((z^{\pm1}))
\,.
\end{equation}
For example, the symbol of $a(S)\in\mc V((S^{\pm1}))$
is given, in terms of this action, by 
\begin{equation}\label{20180308:eq4}
a(z)=a(zS)1
\,\in\mc V((z^{\pm1}))
\,.
\end{equation}


Given pseudodifference operators $a(S),b(S)\in\mc V((S^{\pm1}))$,
it is not difficult to write a formula for the symbol of the product $a(S)\circ b(S)$,
and its formal adjoint $(a\circ b)^*(S)$.
We have (cf. \cite[Lem.2.1]{DSKV18}):
\begin{equation}\label{20180307:eq1}
(a\circ b)(z)=a(zS)b(z)
\,,
\end{equation}
and
\begin{equation}\label{20180307:eq2}
(a\circ b)^*(z)=b^*(zS)a^*(z)
\,.
\end{equation}

\subsection{Rational difference operators}
\label{sec:1.6}

Let $\mc V$ be a field with an automorphism $S$,
and consider the algebra of rational difference operators:
$\mc V(S)\,\big(=\mc V(S^{-1})\big)$,
defined as the skewfield of fractions of the algebra of polynomial 
difference operators $\mc V[S]$
(or, equivalently, $\mc V[S^{-1}]$).
Since $\mc V[S]$ is a Euclidean (non-commutative) domain,
it satisfies the Ore condition,
and therefore
$$
\mc V(S)=\big\{a(S)\circ b(S)^{-1}\,\big|\, a(S),b(S)\in\mc V[S],\,b(S)\neq0\big\}
\,.
$$
It can be embedded in both algebras of pseudodifference operators $\mc V((S))$
and $\mc V((S^{-1}))$.
Indeed, if $b(S)=\sum_{n=M}^Nb_nS^n\in\mc V[S]$ ($M\leq N$), 
we can factor it as
$$
b(S)=b_MS^M\circ \Big(1+\sum_{n>M}S^{-M}\big(\frac{b_n}{b_M}\big)S^{n-M}\Big)
\,,
$$
and expand $b(S)^{-1}$, via geometric series expansion,
as an element of $\mc V((S))$,
or we can factor $b(S)$ as
$$
b(S)=b_NS^N\circ \Big(1+\sum_{n<N}S^{-N}\big(\frac{b_n}{b_N}\big) S^{-(N-n)} \Big)
\,,
$$
and expand $b(S)^{-1}$, via geometric series expansion,
as an element of $\mc V((S^{-1}))$.
We denote by $\iota^{\pm}$ the resulting embeddings
of the algebra of rational difference operators
\begin{equation}\label{eq:iota}
\iota^{\pm}:\,\mc V(S)\hookrightarrow\mc V((S^{\pm1}))
\,.
\end{equation}
If $\mc V$ is not a field, but only a domain, the above construction applies over the field of fractions of $\mc V$.

\subsection{The symbol of a rational difference operator as a bilateral series}
\label{sec:1.6b}

By composing the embeddings 
$\mc V(S)\hookrightarrow\mc V((S^{\pm1}))$ defined in \eqref{eq:iota} 
with the symbol maps $\mc V((S^{\pm1}))\stackrel{\sim}{\longrightarrow}\mc V((z^{\pm1}))$
defined in \eqref{20180308:eq4},
we get the \emph{positive and negative symbol maps} 
\begin{equation}\label{eq:pmsymbol}
\mc V(S)\hookrightarrow\mc V((z^{\pm1}))
\,\,\,\,,\,\,\,\,\,\,\,\,
r(S)
\mapsto 
r^{\pm}(z)
\,
\big(:=
(\iota_{\pm}r)(z)
\big)
\,.
\end{equation}
\begin{definition}\label{def:symbol}
The \emph{symbol} $r(z)$ of a rational difference operator $r(S)\in\mc V(S)$
is defined as the bilateral series 
\begin{equation}\label{eq:symbol}
r(z)
=
\frac12
r^+(z)
+
\frac12
r^-(z)
\,\in\mc V[[z,z^{-1}]]
\,.
\end{equation}
\end{definition}
Note that, for a difference operator $a(S)\in\mc V[S,S^{-1}]$,
the symbol coincides with the positive and negative ones.
\begin{proposition}\label{prop:symbol}
The symbol map $\mc V(S)\to\mc V[[z,z^{-1}]]$
is an injective linear map.
Let the bilateral series $R(z)\in\mc V[[z,z^{-1}]]$ be
the symbol of a rational difference operator $r(S)\in\mc V(S)$.
We can reconstruct the rational difference operator $r(S)$ as follows.
Decompose (uniquely) the bilateral series $R(z)$ as 
\begin{equation}\label{eq:symbol0}
R(z)=R(z)_++R(z)_-
\,\,,\text{ where }\,\,
R(z)_+\in\mc V[[z]]
\,\,\text{ and }\,\,
R(z)_-\in\mc V[[z^{-1}]]z^{-1}
\,.
\end{equation}
Then $R(z)_+$ and $R(z)_-$ are the positive and negative symbols,
respectively, of two (uniquely defined) rational difference operators:
\begin{equation}\label{eq:symbol1}
R(z)_+=(r_+)^+(z)
\,\,,\,\,\,\,
R(z)_-=(r_-)^-(z)
\,\,\text{ for some }\,\,
r_{\pm}(S)\in\mc V(S)
\,,
\end{equation}
and we have
\begin{equation}\label{eq:symbol2}
r(S)=r_+(S)+r_-(S)
\,.
\end{equation}
\end{proposition}
\begin{proof}
Obviously the symbol map $r(S)\mapsto r(z)\in \mc V[[z,z^{-1}]]$ 
defined by \eqref{eq:symbol}
is a linear map, since it is a linear combination
of compositions of linear maps.
If $r\in\mc V(S)$ is such that $r(z)=0$,
then we have
\begin{equation}\label{20180704:eq1}
r^+(z)
=
-
r^-(z)
\,\in
\mc V((z))\cap\mc V((z^{-1}))
=
\mc V[z,z^{-1}]
\,.
\end{equation}
On the other hand, 
the positive and negative symbol maps $\mc V(S)\to\mc V((z^{\pm1}))$ 
are injective,
and they both restrict to the ``symbol map'' bijection
$\mc V[S,S^{-1}]\stackrel{\sim}{\longrightarrow}\mc V[z,z^{-1}]$.
Since, by \eqref{20180704:eq1}, $r^+(z)\in\mc V[z,z^{-1}]$,
it then follows that $r(S)\in\mc V[S,S^{-1}]$,
and since, again by \eqref{20180704:eq1}, $r^+(z)=-r^-(z)$, we get $r=0$.
This proves that the symbol map \eqref{eq:symbol} is injective.

Let us prove the reconstruction claim.
By assumption, we have
\begin{equation}\label{eq:symbol3}
R(z)
=
\frac12r^+(z)+\frac12r^-(z)
\,\,\text{ for some }\,\, 
r(S)\in\mc V(S)
\,.
\end{equation}
Combining \eqref{eq:symbol0} and \eqref{eq:symbol3}, we get
$$
q(z)
:=
R(z)_+-\frac12 r^+(z)
=
\frac12 r^-(z)-R(z)_-
\,\in\mc V((z))\cap\mc V((z^{-1}))
=\mc V[z,z^{-1}]
\,.
$$
Hence, $q(z)$ is the symbol of a difference operator $q(S)\in\mc V[S,S^{-1}]$.
In particular, 
$q^+(z)=q^-(z)=q(z)$.
we thus get
\begin{equation}\label{eq:symbol4}
R(z)_+
=
\frac12r^+(z)+q(z)
=
(\frac12r+q)^+(z)
\,\,,\,\,\,\,
R(z)_-
=
\frac12r^-(z)-q(z)
=
(\frac12r-q)^-(z)
\,.
\end{equation}
Since the positive and negative symbol maps \eqref{eq:pmsymbol}
are injective,
we get from \eqref{eq:symbol1} and \eqref{eq:symbol4}
that $r_{\pm}(S)=\frac12r(S)\pm q(S)$.
Equation \eqref{eq:symbol2} follows.
\end{proof}
\begin{example}\label{ex:japan}
The bilateral series 
$r(\lambda)=\sum_{n\geq1}(-1)^n(\lambda^n-\lambda^{-n})$
(cf. equation \eqref{eq:japan2})
is the symbol
of the rational function
$\frac{1-S}{1+S}$.
\end{example}

One has to be careful when using the notation \eqref{eq:symbol}.
Indeed, for $a(S),b(S)\in\mc V((S^{\pm1}))$ we have
$$
(a\circ b)(zS)1=a(zS)b(zS)1
\,.
$$
This formula for rational difference operators $a(S),b(S)\in\mc V(S)$
makes no sense, 
since the RHS, being product of bilateral series,
may have divergent series.
Instead, the correct version for rational difference operators
is given by the following:
\begin{lemma}\label{lem:product-rational}
The symbol of the composition of two rational difference operators
$f(S),g(S)\in\mc V(S)$ is given by
\begin{equation}\label{eq:product-rational}
(f\circ g)(z)
=
\frac12
f^+(zS)g^+(z)
+
\frac12
f^-(zS)g^-(z)
\,.
\end{equation}
\end{lemma}
\begin{proof}
Obvious.
\end{proof}

\subsection{Rules for computing $\lambda$-brackets with rational operators}
\label{sec:1.6c}

Note that the embeddings $\iota^{\pm}$ in \eqref{eq:iota}
are algebra homomorphisms,
while their halfsum $\frac12(\iota^++\iota^-)$ is not
(and it has values in $\mc V[[S,S^{-1}]]$, which is not an algebra).
The following proposition provides useful rules
for computing $\lambda$-brackets
of symbols of rational difference operators.
\begin{proposition}\label{prop:rules}
  Let $\mc V$ be a domain
  with an automorphism $S$,
endowed with a (possibly non-local) multiplicative $\lambda$-bracket $\{\cdot\,_\lambda\,\cdot\}$.
Let $u\in\mc V$ and $f(S),g(S)\in\mc V(S)$.
We have, recalling the notation \eqref{eq:notation},
\begin{align}
\{u_\lambda(f\circ g)^{\pm}(z)\}
&=
\{u_\lambda f^{\pm}(zx)\}\big(\big|_{x=S}g^{\pm}(z)\big)
+
f^{\pm}(z\lambda x)
\big(\big|_{x=S}
\{u_\lambda g^{\pm}(z)\}
\big)
\label{eq:rules1}\\
\{(f\circ g)^{\pm}(z)_\lambda u\}
&=
\{f^{\pm}(zx)_{\lambda x} u\}
\big(\big|_{x=S} g^{\pm}(z)\big)
+
\{g^{\pm}(z)_{\lambda x} u\} \big(\big|_{x=S}f^{\pm}(z\lambda^{-1}x^{-1})\big)
\label{eq:rules2}\\
\{u_\lambda(f^{-1})^{\pm}(z)\}
&=
-(f^{-1})^{\pm}(z\lambda S)_\to
\{u_\lambda f^{\pm}(zx)\}
\big(\big|_{x=S}(f^{-1})^{\pm}(z)\big)
\label{eq:rules3}\\
\{(f^{-1})^{\pm}(z)_\lambda u\}
&=
-
\{f^{\pm}(zx)_{\lambda xy} u\}
\big(\big|_{x=S}(f^{-1})^{\pm}(z)\big)
\big(\big|_{y=S}(f^{-1})^{\pm}(z\lambda^{-1}y^{-1})\big)
\label{eq:rules4}
\end{align}
\end{proposition}
\begin{proof}
The proof is straightforward.
It uses formula \eqref{20180307:eq1}, the Leibniz rules and the sesquilinearity conditions.
For example:
\begin{align*}
\{& u_\lambda(f\circ g)^+(z)\}
=
\{u_\lambda f^+(zS)g^+(z)\}
\\
& =
\{u_\lambda f^{\pm}(zx)\}\big(\big|_{x=S}g^{\pm}(z)\big)
+
f^{\pm}(z\lambda x)
\big(\big|_{x=S}
\{u_\lambda g^{\pm}(z)\}
\big)
\,,
\end{align*}
proving \eqref{eq:rules1} with $+$.
\end{proof}
\begin{remark}\label{rem:rules}
Note that equations \eqref{eq:rules1}-\eqref{eq:rules4} fail
if in place of the positive and negative symbols we have the symbols.
However, we can compute the analogous $\lambda$-brackets with the symbols 
$(f\circ g)(z)$ or $(f^{-1})(z)$
using \eqref{eq:rules1}-\eqref{eq:rules4}
and the definition \eqref{eq:symbol} of the symbol:
$(f\circ g)(z)=\frac12(f\circ g)^+(z)+\frac12(f\circ g)^-(z)$,
and $(f^{-1})(z)=\frac12(f^{-1})^+(z)+\frac12(f^{-1})^-(z)$.
\end{remark}

\subsection{Rational mPVA}
\label{sec:1.6d}

\begin{definition}\label{def:rational}
A non-local mPVA $\mc V$ is called \emph{rational} if, for every $a,b\in\mc V$,
the multiplicative $\lambda$-bracket $\{a_\lambda b\}\in\mc V[[\lambda,\lambda^{-1}]]$
is the symbol \eqref{eq:symbol} of a rational difference operator.
For $a,b\in\mc V$, if $\{a_\lambda b\}=f_{a,b}(\lambda)$ 
is the symbol of the rational operator $f_{a,b}(S)\in\mc V(S)$,
we denote by $\{a_\lambda b\}^{\pm}$ the corresponding positive and negative symbols:
\begin{equation}\label{eq:pm-lambda}
\{a_\lambda b\}^{\pm}
=
{f_{a,b}}^{\pm}(\lambda)
\,\in\mc V((\lambda^{\pm1}))
\,,
\end{equation}
so that $\{a_\lambda b\}=\frac12\{a_\lambda b\}^++\frac12\{a_\lambda b\}^-$.
\end{definition}
\begin{remark}\label{rem:def-rational}
An alternative definition of a rational mPVA $\mc V$
is obtained by letting the $\lambda$-bracket $\{\cdot\,_\lambda\,\cdot\}$
have values in $\mc V((\lambda))$,
and requiring that, for $a,b\in\mc V$, the $\lambda$-bracket $\{a_\lambda b\}$
is the expansion in $\mc V((\lambda))$ of the symbol of a rational 
pseudodifference operator.
The skew-symmetry axiom (ii) and the Jacobi identity (iii) then would require
some explanation.
First, by assumption $\{b_\lambda a\}$ is the symbol 
of a rational difference operator 
$r(S)=f(S)\circ g(S)^{-1}$.
Then, the RHS of the skewsymmetry axiom, 
$\big(\big|_{x=S}\{b_{\lambda^{-1}x^{-1}}a\}\big)$,
is the symbol of  $r^*(S)=g^*(S)^{-1}\circ f^*(S)$,
which is also a rational difference operator.
Hence, the skewsymmetry axiom can be rewritten as the skewadjointness $r^*(S)=-r(S)$
in the space $\mc V(S)$.
As for the Jacobi identity, 
it is not hard to check that all three terms of the identity 
are linear combinations of expressions of the form
$$
r(\lambda\mu S)\big(p(\lambda)q(\mu)\big)
\,,
$$
for rational pseudodifference operators $p(S),q(S),r(S)\in\mc V(S)$.
Hence, the Jacobi identity should be interpreted as an identity 
between expressions of this form
(i.e., can be rewritten as an identity in $\mc V(S)^{\otimes 3}$).
\end{remark}
\begin{example}
  \label{ex:rational}
Let $\mc V$ be an algebra of difference functions in one variable $u$.
The non-local mPVA structure on $\mc V$ defined in Example \ref{20180309:ex2}
is rational provided that the bilateral series $r(\lambda)\in\mb F[[\lambda,\lambda^{-1}]]$
is the symbol 
of a constant coefficients rational difference operator $r(S)\in\mb F(S)$
satisfying $r(S^{-1})=-r(S)$.
(The same is true for the non-local mPVA structure of Example \ref{ex:nl-mPVA1}).
\end{example}
\begin{example}\label{ex:5.6}
  Let $\mc V$ be as in Example \eqref{ex:rational}.
%
Consider the two compatible non-local mPVA structures on $\mc V$
constructed in Example \ref{ex:lattice_sl2}.
Note that the first mPVA $\lambda$-bracket $\{\cdot\,_\lambda\,\cdot\}_1$ 
in \eqref{eq:compatible} is local,
while the second mPVA $\lambda$-bracket $\{\cdot\,_\lambda\,\cdot\}_2$,
with $r(\lambda)$ as in \eqref{eq:japan2},
is non-local, though rational by Example \ref{ex:japan}.
In order to construct a local mPVA subalgebra with respect to both $\lambda$-brackets,
assume that $u$ is an invertible element of $\mc V$, and consider the Miura transformation $v=(uS(u))^{-1}$
(cf. \cite{HI97}).
A straightforward $\lambda$-bracket computation using 
equations \eqref{eq:rules3}-\eqref{eq:rules4}
yields
\begin{equation}
\begin{split}\label{20180616:eq1}
\{v_\lambda v\}_{1}
=&v\big(1+\lambda S\big)v\big(1+\lambda S\big)v
-v\big(1+(\lambda S)^{-1}\big)v\big(1+(\lambda S)^{-1}\big)v
\,,
\\
\{v_\lambda v\}_{2}
=&v\big(\lambda S-(\lambda S)^{-1}\big)v
\,.
\end{split}
\end{equation}
Let $\mc A_2\subset\mc V$ be the subalgebra of $\mc V$ of difference polynomials in $v$.
Thus we get a pair of compatible local mPVA $\lambda$-brackets on $\mc A_2$.
It is proved in \cite{DSKVW18} that any mPVA $\lambda$-bracket
of order less than or equal to $2$ on $\mc A_2$
is either a linear combination of those from
\eqref{20180616:eq1}
or is a $\lambda$-bracket \eqref{eq:ff} of order $\leq 2$.
We show in \cite{DSKVW18} that, applying the Lenard-Magri scheme to the compatible
$\lambda$-brackets from \eqref{20180616:eq1}, gives integrability of the Volterra lattice.
Also, we point out there that 
the local lattice Poisson algebra corresponding to the difference of the
structures \eqref{20180616:eq1}
is the \emph{Faddeev-Takhtajan-Volkov} algebra \cite{FT86}.
\end{example}
In Section \ref{sec:5.2c} we will consider a generalization of this example for arbitrary $\mc W_N$, $N\geq 3$. In the next example we construct $\mc W_3$.
\begin{example}\label{ex:lattice_sl3}
  Let $\mc V$ be
an algebra of difference functions in two variables $u,v$.
Consider the constant coefficients rational difference operator
$r(S)=\frac{(S-1)^2}{S^3-1}\in\mb F(S)$.
Define 
the following two multiplicative $\lambda$-brackets on $\mc V$:
$$
\{u_\lambda u\}_1=0
\,,\qquad
\{u_\lambda v\}_1=\lambda^{-2}-\lambda
\,,
\qquad
\{v_\lambda v\}_1=((\lambda S)^{-1}-\lambda)u
$$
and
\begin{align*}
\{u_\lambda u\}_2&=(\lambda^{-1}-\lambda S)v+ur(\lambda S)(\lambda S+1)u\,,
\\
\{u_\lambda v\}_2&=vr(\lambda S)u\,,%
\\
\{v_\lambda v\}_2&=vr(\lambda S)(\lambda S+1)v
\,.
\end{align*}
One can check that they are compatible rational mPVA $\lambda$-brackets.
The corresponding commutators
of formal distributions
define the $q$-deformed $W$-algebra of $\mf{sl}_3$ \cite{FR96}, 
see also \cite{HI97}. In Section \ref{sec:5.2c} we shall construct a local mPVA subalgebra
$\mc A_3$ as well.
%
%
The local lattice Poisson algebra corresponding to the mPVA $\mc A_3$
is the \emph{Belov-Chaltikian} algebra \cite{BC93}.
\end{example}

\subsection{Dirac reduction}\label{sec:1.7}

Let $\mc V$ be a rational mPVA with multiplicative $\lambda$-bracket $\{\cdot_\lambda\cdot\}$.
Given elements $\theta_1,\dots,\theta_m\in\mc V$, we consider the matrix
\begin{equation}\label{20180310:eq2}
C(\lambda)
=
\big(
\{{\theta_{\beta}}_{\lambda}\theta_{\alpha}\}
\big)_{\alpha,\beta=1}^m
\,\in\Mat_{m\times m}\mc V[[\lambda,\lambda^{-1}]]
\,.
\end{equation}
By the rationality assumption on $\mc V$,
this is an $m\times m$-matrix,
symbol of a rational matrix difference operator:
$$
C=C(S)\in\Mat_{m\times m}\mc V(S)
\,,
$$
which we assume to be invertible.
The \emph{Dirac modified $\lambda$-bracket} $\{\cdot_\lambda\cdot\}^D$ by the constraints $\theta_1,\dots,\theta_m$ is defined as follows
\begin{equation}\label{eq:dirac1}
\{a_\lambda b\}^D
=
\{a_\lambda b\}
-\frac12\sum_{\pm}\sum_{\alpha,\beta=1}^m\{{\theta_\alpha}_{\lambda x}b\}^{\pm}
\big(\big|_{x=S}
(C^{-1})^{\pm}_{\alpha\beta}(\lambda S)
\{a_\lambda\theta_\beta\}^{\pm}
\big)
\,,
\quad
a,b\in\mc V
\,.
\end{equation}
In the RHS of \eqref{eq:dirac1}
we are using the notation \eqref{eq:pm-lambda}.

The following result is the ``multiplicative analogue''
of \cite[Thm.2.2]{DSKVdirac} 
\begin{theorem}\label{20180310:thm1}
Let $\mc V$ be a rational mPVA with automorphism $S$ 
and $\lambda$-bracket $\{\cdot\,_\lambda\,\cdot\}$.
Let $\theta_1,\dots,\theta_m\in\mc V$ be elements such that the rational matrix 
pseudodifference operator $C(S)\in\Mat_{m\times m}\mc V(S)$ 
with symbol \eqref{20180310:eq2} is invertible.
\begin{enumerate}[(a)]
\item
The Dirac modification $\{\cdot\,_\lambda\,\cdot\}^D$ defined by \eqref{eq:dirac1}
is a rational multiplicative Poisson $\lambda$-bracket on $\mc V$.
\item
All elements $\theta_1,\dots,\theta_m$ are central with respect to 
the Dirac modified $\lambda$-bracket:
$\{a_\lambda\theta_i\}^D=\{{\theta_i}_\lambda a\}^D=0$
for all $i=1,\dots,m$ and $a\in\mc V$.
\item
The associative algebra ideal $\mc I=\langle S^n(\theta_i)\,|\,i=1,\dots,m;\,n\in\mb Z\rangle$
is an mPVA ideal of $\mc V$ 
with respect to the Dirac modified $\lambda$-bracket $\{\cdot\,_\lambda\,\cdot\}^D$.
\end{enumerate}
Hence, the quotient space $\mc V/\mc I$
is a rational mPVA with respect to the multiplicative $\lambda$-bracket
induced by the Dirac modified $\lambda$-bracket $\{\cdot\,_\lambda\,\cdot\}^D$ on $\mc V$,
provided that it is defined.
We call this mPVA
the Dirac reduction of of $\mc V$ by the constraints $\theta_1,\dots,\theta_m$.
\end{theorem}
\begin{proof}
Straightforward.
\end{proof}
As a special case, assume that $\mc V$ is an algebra of difference functions in
$u_1,...,u_\ell$, and that the multiplicative Poisson $\lambda$-bracket
on $\mc V$ is given by the Poisson structure 
$H(S)=\big(H_{ij}(S)\big)_{i,j=1}^\ell\in\Mat_{\ell\times\ell}\mc V(S)$,
with symbols of the entries given by
$$
H_{ij}(\lambda)
=
\{{u_j}_\lambda{u_i}\}
\,\in\mc V(\lambda).
$$
Then, by the Master Formula \eqref{masterformula},
the matrix $C(S)\in\Mat_{m\times m}\mc V(S)$ with symbol \eqref{20180310:eq2}
is given by
\begin{equation}\label{20180310:eq4}
C(S)
=
D_\Theta(S)\circ H(S)\circ D^*_\Theta(S)
\,,
\end{equation}
where 
$$
D_\Theta(S)
=
\big(
\sum_{n\in\mb Z}
\frac{\partial\theta_\alpha}{\partial u_{i,n}}S^n
\big)_{\alpha=1,\dots,m,\,i=1,\dots,\ell}
\,\in\Mat_{m\times\ell}\mc V[S,S^{-1}]
\,,
$$
is the Frechet derivative of $\Theta=(\theta_\alpha)_{\alpha=1,\dots,n}$, 
and $D^*_\Theta(S)$ is the transposed adjoint matrix:
$$
D^*_\Theta(S)
=
\big(
\sum_{n\in\mb Z}
S^{-n}\circ \frac{\partial\theta_\beta}{\partial u_{j,n}}
\big)_{j=1,\dots,\ell,\,\beta=1,\dots,m}
\,\in\Mat_{\ell\times m}\mc V[S,S^{-1}]
\,.
$$
Moreover, it is not hard to check that
the Dirac modified Poisson structure $H^D(S)$,
corresponding to the Dirac modified multiplicative $\lambda$-bracket \eqref{eq:dirac1},
is as follows:
\begin{equation}\label{20180310:eq5}
H^D(S)
=
H(S)
+
B(S)\circ C(S)^{-1}\circ B^*(S)
\,\in\Mat_{\ell\times\ell}\mc V(S)
\,,
\end{equation}
where
\begin{align*}
& B(S):=H(S)\circ D^*_\Theta(S)
\,\in\Mat_{\ell\times m}\mc V(S)
\,, \\
& B^*(S)=-D_\Theta(S)\circ H(S)
\,\in\Mat_{m\times\ell}\mc V(S)
\,,
\end{align*}
and $C(S)$ is as in \eqref{20180310:eq4}.

As a further special case, 
assume that the constraints have the form
$$
\theta_i=u_i+c_i\,,\,\,i=1,\dots,m
\,,
$$
where $c_i,\,i=1,\dots,m$, are constants.
In this case, let us write the multiplicative Poisson structure $H(S)$ in block form
$$
H(S)
=
\left(\begin{array}{ll}
H_1(S)&H_2(S) \\
H_3(S)&H_4(S)
\end{array}\right)
\,,
$$
where the blocks are of sizes
$H_1(S)\in\Mat_{m\times m}\mc V(S)$,
$H_2(S)\in\Mat_{m\times(\ell-m)}\mc V(S)$,
$H_3(S)\in\Mat_{(\ell-m)\times m}\mc V(S)$
and $H_4(S)\in\Mat_{(\ell-m)\times(\ell-m)}\mc V(S)$.
Then, the above matrices 
$D_\Theta(S)\in\Mat_{m\times\ell}\mc V(S)$,
$D_\Theta(S)^*\in\Mat_{\ell\times m}\mc V(S)$,
$C(S)\in\Mat_{m\times m}\mc V(S)$,
$B(S)\in\Mat_{\ell\times m}\mc V(S)$,
and $B(S)^*\in\Mat_{m\times\ell}\mc V(S)$,
are as follows:
$$
\begin{array}{l}
\displaystyle{
D_\Theta(S)
=
\left(\begin{array}{ll}
\id_{m\times m} & 0_{m\times(\ell-m)} 
\end{array}\right)
\,,\,\,
D^*_\Theta(S)
=
\left(\begin{array}{l}
\id_{m\times m} \\ 0_{(\ell-m)\times m}
\end{array}\right)
\,,} \\
\displaystyle{
B(S)
=
\left(\begin{array}{l}
H_1(S) \\ H_3(S)
\end{array}\right)
\,,\,\,
B^*(S)
=
-
\left(\begin{array}{ll}
H_1(S) & H_2(S) 
\end{array}\right)
\,,\,\,
C(S)=H_1(S)
\,.}
\end{array}
$$
Hence,
$$
H^D(S)
=
\left(\begin{array}{ll}
0&0 \\
0& H_4(S)-H_3(S)\circ H_1(S)^{-1}\circ H_2(S)
\end{array}\right)
\,.
$$
In other words, 
the multiplicative Poisson structure for the Dirac reduced mPVA $\mc V/\mc I$
is the quasideterminant of the matrix $H(S)$ with respect to the block $H_4(S)$
(cf. formula \eqref{20180310:eq5} and \cite{OR89}):
$$
H_4(S)-H_3(S)\circ H_1(S)^{-1}\circ H_2(S)
\,.
$$

\subsection{Relation with $q$-deformations of Poisson algebras}\label{sec:q-def}

Let $\mc V$ be a vector space
over the field $\mb F(q)$ of rational functions in the variable $q$.
Consider the space 
$\mc V[[z,z^{-1},w,w^{-1}]]$
of $\mc V$-valued formal distributions in two variables. 
An element $a(z,w)\in\mc V[[z,z^{-1},w,w^{-1}]]$ is called \emph{quasi-local}
if it has an expansion of the following form:
\begin{equation}\label{eq:quasi-local}
a(z,w)
=
\sum_{j\in\mb Z}c_j(w)\delta(\frac{z}{q^jw})
\,\,\text{ for some }\,\,
c_j(w)\in\mc V[[w,w^{-1}]]
\,.
\end{equation}
Due to Lemma \ref{lem:decomp}
this is a generalization of the notion of a $q$-local formal distribution in two variables.

\begin{example}\label{ex:qdef}
Examples of quasi-local formal distributions
are provided by the $q$-deformed $W$-algebras $W_N$
of Frenkel and Reshetikhin \cite{FR96}.
Let us consider the simplest example when $N=2$.
Then $\mc V$ is a completed algebra of polynomials 
in the indeterminates $t_n$, $n\in\mb Z$,
with the following
two compatible $q$-deformed Poisson brackets 
$$
\{t(z),t(w)\}_1=\delta(\frac{w q}{z})-\delta(\frac{w}{zq})
\quad\text{and}\quad
\{t(z),t(w)\}_2=\sum_{m\in\mb Z}\frac{1-q^m}{1+q^m}\left(\frac{w}{z}\right)^mt(z)t(w)
\,,
$$
where $t(z)=\sum_{n\in\mb Z}t_nz^n$.
The first bracket is obviously local,
while the second can be written using property \eqref{eq:delta2}
in the form \eqref{eq:quasi-local},
where $c_j(w)=a_jt(q^jw)t(w)$,
and $a_j$ are coefficients of the expansion of the function $\frac{1-q}{1+q}$.
Replacing, as in Section \ref{sec:1.3b},
$\delta(z/q^jw)$ by $\lambda^j$,
letting $S(t(w))=t(qw)$ and identifying $t(w)$ with $u$,
these brackets correspond to the compatible multiplicative $\lambda$-brackets 
of Example \ref{ex:lattice_sl2}.
\end{example}

\section{The multiplicative Adler identity and Poisson vertex algebras}\label{sec:2b}

\subsection{The multiplicative Adler identity}\label{sec:2.1}

Let $\mc V$ be a unital commutative associative algebra 
endowed with an automorphism $S$
and a multiplicative $\lambda$-bracket $\{\cdot\,_\lambda\,\cdot\}$.
By analyzing the notion of an Adler type pseudodifferential operator
from \cite{DSKV15,DSKV16},
we arrive at the following multiplicative analogue of it.
\begin{definition}\label{20180312:def}
The \emph{multiplicative Adler identity} 
(or, simply, \emph{Adler identity})
on a pseudodifference operator 
$L(S)\in\mc V((S^{\pm1}))$ 
with respect to a multiplicative $\lambda$-bracket $\{\cdot\,_\lambda\,\cdot\}$
reads
\begin{equation}\label{eq:adler}
\begin{array}{l}
\displaystyle{
\vphantom{\Big(}
\{L(z)_\lambda L(w)\}
=
L(w\lambda S)
\delta_+\big(\frac{w\lambda S}{z}\big)
L^*\big(\frac{\lambda}{z}\big)
-
L(z)
\delta_+\big(\frac{w\lambda S}{z}\big)
L(w)
} \\
\displaystyle{
\vphantom{\Big(}
-\frac12
\big(L(w\lambda S)+ L(w)\big)
\big(L^*\big(\frac{\lambda}{z}\big)- L(z)\big)
\,,}
\end{array}
\end{equation}
where $\delta_+(z)$ is the positive part of the $\delta$-function, defined in \eqref{20180312:eq2}.
\end{definition}
Recalling \eqref{eq:delta2}, we have
$$
L(w\lambda S)
\delta\big(\frac{w\lambda S}{z}\big)
L^*\big(\frac{\lambda}{z}\big)
=
L(z)
\delta\big(\frac{w\lambda S}{z}\big)
L(w)
\,.
$$
Hence, using \eqref{20180312:eq4}, we can rewrite the Adler identity \eqref{eq:adler}
in the equivalent form involving the negative $\delta$-function $\delta_-(z)$:
\begin{equation}\label{eq:adler2}
\begin{array}{l}
\displaystyle{
\vphantom{\Big(}
\{L(z)_\lambda L(w)\}
=
-L(w\lambda S)
\delta_-\big(\frac{w\lambda S}{z}\big)
L^*\big(\frac{\lambda}{z}\big)
+
L(z)
\delta_-\big(\frac{w\lambda S}{z}\big)
L(w)
} \\
\displaystyle{
\vphantom{\Big(}
-\frac12
\big(L(w\lambda S)+ L(w)\big)
\big(L^*\big(\frac{\lambda}{z}\big)- L(z)\big)
\,.}
\end{array}
\end{equation}

Next, we observe that the Adler identity can be rewritten
equivalently in terms of \emph{local} (i.e. polynomial) $\lambda$-brackets
among the coefficients of the pseudodifference operator $L(S)$.
Indeed, let $L(z)=\sum_{i\leq N}u_iz^i\in\mc V((z^{-1}))$
(the same argument works for $L(z)\in\mc V((z))$).
Clearly, the RHS of \eqref{eq:adler} is a Laurent series in $z^{-1}$
with powers of $z$ bounded above by the positive integer $N$,
while the RHS of the equivalent equation \eqref{eq:adler2}
is a Laurent series in $w^{-1}$ with powers of $w$ bounded above by $N$.
Hence, the Adler identity \eqref{eq:adler} is consistent in the degrees of $z$ and $w$,
and, comparing the coefficient of $z^iw^j$ ($i,j\leq N$)
in both sides of \eqref{eq:adler2}, we get the following $\lambda$-bracket
relations:
\begin{equation}\label{eq:adler-coeff1}
\begin{array}{l}
\displaystyle{
\vphantom{\Big(}
\{{u_i}_\lambda{u_j}\}
=
\sum_{n=0}^{N-i}
\big(
u_{j-n}(\lambda S)^{j-i-n}u_{i+n}
-
u_{i+n}(\lambda S)^{n}u_{j-n}
\big)
} \\
\displaystyle{
\vphantom{\Big(}
\qquad\qquad
-
\frac12
u_j((\lambda S)^j+1)((\lambda S)^{-i}-1)u_i
\,.}
\end{array}
\end{equation}
Likewise, for $L(z)=\sum_{i\geq N}u_iz^i\in\mc V((z))$,
the RHS of \eqref{eq:adler2} is clearly a Laurent series in $z$
with powers of $z$ bounded below by $N$,
while the RHS of \eqref{eq:adler}
is a Laurent series in $w$ with powers of $w$ bounded below by $N$.
Hence, again \eqref{eq:adler} is consistent in the degrees of $z$ and $w$,
and \eqref{eq:adler} is equivalent to the following $\lambda$-brackets relations
for all $i,j\geq N$.
\begin{equation}\label{eq:adler-coeff2}
\begin{array}{l}
\displaystyle{
\vphantom{\Big(}
\{{u_i}_\lambda{u_j}\}
=
\sum_{n=0}^{j-N}
\big(
u_{j-n}(\lambda S)^{j-i-n}u_{i+n}
-
u_{i+n}(\lambda S)^{n}u_{j-n}
\big)
} \\
\displaystyle{
\vphantom{\Big(}
\qquad\qquad
-
\frac12
u_j((\lambda S)^j+1)((\lambda S)^{-i}-1)u_i
\,.}
\end{array}
\end{equation}

\begin{proposition}\label{20180312:prop}
Suppose the pseudodifference operator $L(z)\in\mc V((z^{\pm1}))$
satisfies the Adler identity \eqref{eq:adler}.
Then,
we have the multiplicative skew-symmetry relation
\begin{equation}\label{20180312:eq6}
\{L(z)_\lambda L(w)\}
=
-\big(\big|_{x=S}\{L(w)_{\lambda^{-1}x^{-1}}L(z)\}\big)
\,,
\end{equation}
and the multiplicative Jacobi identity:
\begin{equation}\label{20180312:eq7}
\{L(z)_\lambda {\{L(w)_\mu L(v)\}}\}
-\{L(w)_\mu {\{L(z)_\lambda L(v)\}}\}
=
\{{\{L(z)_\lambda L(w)\}}_{\lambda\mu} L(v)\}
\,.
\end{equation}
\end{proposition}
\begin{proof}
From \eqref{eq:adler} we have
\begin{equation}\label{20180313:eq1}
\begin{array}{l}
\displaystyle{
\vphantom{\Big(}
\big(\big|_{x=S}\{L(w)_{\lambda^{-1}x^{-1}} L(z)\}\big)
} \\
\displaystyle{
\vphantom{\Big(}
=
\Big(\Big|_{x=S}
L\big(\frac{zS}{\lambda x}\big)
\delta_+\big(\frac{zS}{w\lambda x}\big)
L^*\big(\frac{1}{w\lambda x}\big)
-
L(w)
\delta_+\big(\frac{zS}{w\lambda x}\big)
L(z)
} \\
\displaystyle{
\vphantom{\Big(}
\qquad
-\frac12
\big(L\big(\frac{zS}{\lambda x})+ L(z)\big)
\big(L^*\big(\frac{1}{w\lambda x}\big)- L(w)\big)
\Big)
} \\
\displaystyle{
\vphantom{\Big(}
=
L(w\lambda S)
\delta_+\big(\frac{z}{w\lambda S}\big)
L^*\big(\frac{\lambda}{z}\big)
-
L(z)
\delta_+\big(\frac{z}{w\lambda S}\big)
L(w)
} \\
\displaystyle{
\vphantom{\Big(}
\qquad
-\frac12
\big(L(w\lambda S)- L(w)\big)
\big(L^*\big(\frac{\lambda}{z})+ L(z)\big)
} \\
\displaystyle{
\vphantom{\Big(}
=
L(w\lambda S)
\delta_-\big(\frac{w\lambda S}{z}\big)
L^*\big(\frac{\lambda}{z}\big)
-
L(z)
\delta_-\big(\frac{w\lambda S}{z}\big)
L(w)
} \\
\displaystyle{
\vphantom{\Big(}
\qquad
+
\frac12
\big(L(w\lambda S)+ L(w)\big)
\big(L^*\big(\frac{\lambda}{z})- L(z)\big)
=
-\{L(z)_\lambda L(w)\}
\,.}
\end{array}
\end{equation}
For the second equality of \eqref{20180313:eq1} we used the definition of formal adjoint,
for the third equality we used \eqref{20180312:eq2},
and for the fourth equality we used \eqref{eq:adler2}.
This proves equation \eqref{20180312:eq6}.
Next, let us prove equation \eqref{20180312:eq7}.
By a straightforward computation (see \cite{DSKV15} for the same computation in the additive case),
using the Adler type identity \eqref{eq:adler} and the Leibniz rules, the Jacobi identity \eqref{20180312:eq7} can be rewritten as the vanishing of the following expression
\begin{align}
&
L(v\lambda\mu S)\delta_+(\frac{v\lambda\mu S}{z})
\Big(
L^*(\frac{\lambda}{z})\delta_+(\frac{v\mu S}{w})L^*(\frac{\mu}{w})
-
L^*(\frac{\mu}{w})\delta_+(\frac{z}{w\lambda S})L^*(\frac{\lambda}{z})
\Big)
\label{1a}
\\
&-L(v\lambda\mu S)
\Big(
\delta_+(\frac{w\lambda S}{z})L^*(\frac{\lambda}{z})
\Big)
\Big(
\delta_+(\frac{v\mu S}{w})L^*(\frac{\mu}{w})
\Big)
\label{1b}
\\
&
+\frac12L(v\lambda\mu S)L^*(\frac{\mu}{w})
\Big(\delta_+(\frac{w\lambda S}{z})+\delta_+(\frac{z}{w\lambda S})-1
\Big)L^*(\frac{\lambda}{z})
\label{1c}
\\
&
+L(v\lambda\mu S)\delta_+(\frac{v\lambda\mu S}{w})
\Big(
L^*(\frac{\lambda}{z})\delta_+(\frac{w}{z\mu S})L^*(\frac{\mu}{w})
-
L^*(\frac{\mu}{w})\delta_+(\frac{v\lambda S}{z})L^*(\frac{\lambda}{z})
\Big)
\label{2a}
\\
&-L(v\lambda\mu S)
\Big(
\delta_+(\frac{v\lambda S}{z})L^*(\frac{\lambda}{z})
\Big)
\Big(
\delta_+(\frac{w}{z\mu S})L^*(\frac{\mu}{w})
\Big)
+L(v\lambda\mu S)L^*(\frac{\mu}{w})\delta_+(\frac{v\lambda S}{z})L^*(\frac{\lambda}{z})
\label{2b}
\\
&
+\frac{1}{2}L(v)L^*(\frac{\mu}{w})
\Big(\delta_+(\frac{w\lambda S}{z})+\delta_+(\frac{z}{w\lambda S})-1
\Big)
L^*(\frac{\lambda}{z})
\label{3}
\\
&
+L(w)\delta_+(\frac{v\lambda\mu S}{w})
\Big(
L(z)\delta_+(\frac{v\lambda S}{z})L(v)
-
L(v)\delta_+(\frac{w}{z\mu S})L(z)
\Big)
\label{4a}
\\
&-L(w)
\Big(
\delta_+(\frac{z\mu S}{w})L(z)
\Big)
\Big(
\delta_+(\frac{v\lambda S}{z})L(v)
\Big)
\label{4b}
\\
&
+\frac12L(w)L(v)
\Big(\delta_+(\frac{z\mu S}{w})+\delta_+(\frac{w}{z\mu S})-1
\Big)L(z)
\label{4c}
\\
&
-L(z)\delta_+(\frac{v\lambda\mu S}{z})
\Big(
L(w)\delta_+(\frac{v\mu S}{w})L(v)
+
L(v)\delta_+(\frac{w\lambda S}{z})L(w)
\Big)
\label{5a}
\\
&+L(z)
\Big(
\delta_+(\frac{w\lambda S}{z})L(w)
\Big)
\Big(
\delta_+(\frac{v\mu S}{w})L(v)
\Big)
+L(z)\delta_+(\frac{v\lambda\mu S}{z})L(w)L(v)
\label{5b}
\\
&+\frac12 L(v\lambda\mu S)L(w)
\Big(\delta_{+}(\frac{z\mu S}{w})+\delta_+(\frac{w}{z\mu S})-1
\Big)L(z)
\label{6}
\\
&
+L(v\lambda\mu S)
\delta_+(\frac{v\lambda\mu S}{z})
L^*(\frac{\lambda\mu S}{z})
\Big(\delta_+(\frac{w\lambda S}{z})+\delta_+(\frac{z}{w\lambda S})-1\Big)L(w)
\label{7a}
\\
&
-\frac12L(v\lambda\mu S)
L^*(\frac{\lambda\mu S}{z})
\Big(\delta_+(\frac{w\lambda S}{z})+\delta_+(\frac{z}{w\lambda S})-1\Big)L(w)
\label{7b}
\\
&+\Big(\delta_+(\frac{v\lambda S}{z})L(v)\Big)
L(z\mu S)\Big(
\delta_+(\frac{z\mu S}{w})+\delta_+(\frac{w}{z\mu S})-1
\Big)L^*(\frac{\mu}{w})
\label{8a}
\\
&-\frac12L(v)
L(z\mu S)\Big(
\delta_+(\frac{z\mu S}{w})+\delta_+(\frac{w}{z\mu S})-1
\Big)L^*(\frac{\mu}{w})
\label{8b}\\
&
-\frac{1}{2}L(v\lambda\mu S)
L(z\mu S)
\Big(
\delta_+(\frac{z\mu S}{w})+\delta_+(\frac{w}{z\mu S})-1
\Big)
L^*(\frac{\mu}{w})
\label{9}
\\
&
-\frac{1}{2}L(v)L^*(\frac{\lambda\mu S}{z})
\Big(\delta_+(\frac{w\lambda S}{z})+\delta_+(\frac{z}{w\lambda S})-1\Big)L(w)
\,.
\label{10}
\end{align}
Using notation \eqref{eq:notation} we can rewrite
\begin{align*}
\eqref{2a}+\eqref{2b}&=L(v\lambda\mu x y)\Big(
\delta_+(\frac{v\lambda\mu x y}{w})\delta_+(\frac{w}{z\mu y})
-\delta_+(\frac{v\lambda\mu x y}{w})\delta_+(\frac{v\lambda x}{z})
\\
&-\delta_+(\frac{v\lambda x}{z})\delta_+(\frac{w}{z\mu y})
+\delta_{+}(\frac{v\lambda x}{z})
\Big)\big(|_{x=S}L^*(\frac{\lambda}{z})\big)\big(|_{y=S}L^*(\frac{\mu}{w})\big)
=0
\,.
\end{align*}
In the last identity we used \eqref{20180316:eq1},
with $x_1=\frac{v\lambda \mu x y}{w}$ and $x_2=\frac{w}{z\mu y}$.
Similarly, we rewrite
\begin{align*}
\eqref{5a}+\eqref{5b}&=L(z)\Big(
\delta_+(\frac{w\lambda y}{z})\delta_+(\frac{v\mu t}{w})
-\delta_+(\frac{v\lambda\mu y t}{z})\delta_+(\frac{v\mu t}{w})
\\
&-\delta_+(\frac{v\lambda \mu y t}{z})\delta_+(\frac{w\lambda y}{z})
+\delta_{+}(\frac{v\lambda \mu y t}{z})
\Big)\big(|_{y=T}L(w)\big)\big(|_{t=S}L(z)\big)
=0
\,.
\end{align*}
In the last identity we applied again the relation \eqref{20180316:eq1} with $x_1=\frac{w\lambda y}{z}$ and
$x_2=\frac{v\mu t}{w}$.
Furthermore, note that
\begin{align}
\begin{split}\label{20180316:eq3}
\eqref{1a}+\eqref{1b}&=L(v\lambda\mu xy)\Big(
\delta_+(\frac{v\lambda\mu xy}{z})\delta_+(\frac{v\mu y}{w})
-\delta_+(\frac{v\lambda\mu xy}{z})\delta_+(\frac{z}{w\lambda y})
\\
&-\delta_{+}(\frac{w\lambda x}{z})\delta_+(\frac{v\mu y}{w})
\Big)\big(|_{x=S}L^*(\frac{\lambda}{z})\big)\big(|_{y=S}L^*(\frac{\mu}{w})\big)
\\
&
=-L(w\lambda S)\delta_{+}(\frac{w\lambda\mu S}{z})L^*(\frac{\mu}{w})\delta(\frac{w\lambda S}{z})L^*(\frac{\lambda}{z})
\,.
\end{split}
\end{align}
In the last equality we used \eqref{20180316:eq2}
with $x_1=\frac{w\lambda x}{z}$ and $x_2=\frac{v \mu y}{w}$. 
Using equations \eqref{20180312:eq4} and \eqref{20180312:eq2}, we can rewrite
\begin{equation}
\begin{split}
\label{20180316:eq4}
\eqref{7a}
&=L(v\lambda\mu S)
\delta_+(\frac{v\lambda\mu S}{z})
L^*(\frac{\lambda\mu S}{z})
\delta(\frac{w\lambda S}{z})L(w)
\\
&=L(v\lambda\mu S)
\delta_+(\frac{v\lambda\mu S}{z})
L^*(\frac{\mu }{w})
\delta(\frac{w\lambda S}{z})L^*(\frac{\lambda}{z})
\,,
\end{split}
\end{equation}
where in the last equality we used the property \eqref{eq:delta2} of the multiplicative delta function.
Combining equations \eqref{20180316:eq3} and \eqref{20180316:eq4} we get
$$
\eqref{1a}+\eqref{1b}+\eqref{7a}=0
\,.
$$
Similarly, one shows that
\begin{equation}\label{20180316:eq5}
\begin{split}
\eqref{4a}+\eqref{4b}
& =-L(w)\delta_+(\frac{v\lambda\mu S}{w})L(v)\delta(\frac{z\mu S}{w})L(z)
\\
& =-\big(\delta_+(\frac{v\lambda S}{z})L(v)\big)L(z\mu S)\delta(\frac{z\mu S}{w})L(z)
\,.
\end{split}
\end{equation}
In the last identity we used the property \eqref{eq:delta2} of the multiplicative delta function.
Using equations \eqref{20180312:eq4}, \eqref{20180312:eq2} and \eqref{20180316:eq5} we get
$$
\eqref{4a}+\eqref{4b}+\eqref{8a}=0
\,.
$$
Finally, using equations \eqref{20180312:eq4}, \eqref{20180312:eq2} and the property \eqref{eq:delta2} of the multiplicative delta-function, one
shows that
$$
\eqref{1c}+\eqref{7b}=0
\,,\,\,
\eqref{3}+\eqref{10}=0
\,,\,\,
\eqref{4c}+\eqref{8b}=0
\,,\,\,
\eqref{6}+\eqref{9}=0
\,.
$$
This completes the proof.
\end{proof}
As an immediate consequence of Proposition \ref{20180312:prop}, we have
\begin{corollary}\label{20180312:cor}
Assume that $\mc V$ is a unital commutative associative algebra with an automorphism $S$,
and assume that the pseudodifference operator $L(S)\in\mc V((S^{\pm1}))$
satisfies the Adler identity \eqref{eq:adler}
with respect to a multiplicative $\lambda$-bracket $\{\cdot\,_\lambda\,\cdot\}$ of $\mc V$.
Let $\mc U\subset\mc V$ be the smallest subalgebra of $\mc V$ 
containing all the coefficients of $L(z)$ and preserved by the automorphism $S$.
Then,
$\mc U$ is a subalgebra with respect to the multiplicative $\lambda$-bracket $\{\cdot\,_\lambda\,\cdot\}$,\
i.e. $\{\mc U_\lambda\,\mc U\}\subset\mc U[\lambda,\lambda^{-1}]$,
and, moreover, the restriction of $\{\cdot\,_\lambda\,\cdot\}$
to $\mc U$ defines a structure of an mPVA on $\mc U$.
If, in particular, $\mc V$ is generated by the coefficients of $L(z)$ and the action of $S$,
then $\mc V$ is an mPVA.
\end{corollary}
\begin{proof}
Due to Proposition \ref{prop:master},
if skewsymmetry and Jacobi identity for the multiplicative $\lambda$-bracket
hold on a set of difference generators of $\mc U$,
then they hold on the whole $\mc U$.
\end{proof}

\section{The multiplicative 3-Adler identity}\label{sec:3}

By analyzing the work of Oevel and Ragnisco \cite{OR89}
(see \cite{DSKV19}) we arrive at the following definition.
\begin{definition}\label{def:3adler}
The \emph{multiplicative 3-Adler identity} on a pseudodifference operator $L(S)\in\mc V((S^{\pm1}))$
reads
\begin{equation}\label{eq:Adler3}
\begin{array}{l}
\displaystyle{
\vphantom{\big(}
\{L(z)_\lambda L(w)\}
=
L(w \lambda S)L(z)
\delta_+\big(\frac{w \lambda S}{z}\big)
\big(L^*(\frac{\lambda}{z})-\frac{w \lambda S}{z}L(w)\big)
} \\
\displaystyle{
\vphantom{\big(}
-\big(L(z)-L(w \lambda S)\frac{w \lambda S}{z}\big)
\delta_+\big(\frac{w \lambda S}{z}\big)
L(w)L^*(\frac{\lambda}{z})
\,.}
\end{array}
\end{equation}
\end{definition}
\noindent
Using equations \eqref{eq:delta2} and \eqref{20180312:eq4} 
we can rewrite equation \eqref{eq:Adler3} as
\begin{equation}\label{eq:Adler3b}
\begin{array}{l}
\displaystyle{
\vphantom{\big(}
\{L(z)_\lambda L(w)\}
=
\big(L(z)-L(w \lambda S)\frac{w \lambda S}{z}\big) \delta_-(\frac{w \lambda S}{z})
L(w)L^*(\frac{\lambda}{z})
} \\
\displaystyle{
\vphantom{\big(}
-L(w \lambda S)L(z)\delta_-(\frac{w \lambda S}{z})\big(L^*(\frac{\lambda}{z})-\frac{w \lambda S}{z}L(w)\big)
\,.}
\end{array}
\end{equation}
\begin{proposition}\label{20180312:propB}
Suppose the pseudodifference operator $L(z)\in\mc V((z^{\pm1}))$
satisfies the 3-Adler identity \eqref{eq:Adler3}.
Then,
multiplicative skew-symmetry relation \eqref{20180312:eq6}
and the multiplicative Jacobi identity \eqref{20180312:eq7} hold.
\end{proposition}
\begin{proof}
By equation \eqref{eq:Adler3} we have
\begin{equation}\label{20180313:eq2}
\begin{array}{l}
\displaystyle{
\vphantom{\big(}
\big(\big|_{x=S}\{L(w)_{\lambda^{-1}x^{-1}} L(z)\}\big)
=
\big(L(w\lambda S)-L(z)\frac{z}{w \lambda S}\big)
\delta_+\big(\frac z {w \lambda S}\big)
L(z)L^*(\frac{\lambda}{z})
} \\
\displaystyle{
\vphantom{\big(}
-L(w \lambda S)L(z)
\delta_+\big(\frac z {w \lambda S}\big)
\big(L(w)-\frac z{w \lambda S}L^*(\frac{\lambda}{z})\big)
\,.}
\end{array}
\end{equation}
Using equations \eqref{eq:Adler3}, \eqref{20180313:eq2} 
and the definition of the multiplicative $\delta$-function,
we thus get
$$
\begin{array}{l}
\displaystyle{
\vphantom{\big(}
\{L(z)_\lambda L(w)\}
+
\big(\big|_{x=S}\{L(w)_{\lambda^{-1}x^{-1}} L(z)\}\big)
} \\
\displaystyle{
\vphantom{\big(}
=
L(w\lambda S)L(z)\delta(\frac{w\lambda S}{z})\big( L^*(\frac{\lambda}{z})-L(w)\big)
} \\
\displaystyle{
\vphantom{\big(}
\quad
+\big(L(w\lambda S)-L(z)\big)\delta(\frac{w\lambda S}{z})L^{*}(\frac{\lambda}{z})L(w)
=0
\,.}
\end{array}
$$
In the last identity we used equation \eqref{eq:delta2}.
This proves the skewsymmetry relation \eqref{20180312:eq6}.
The Jacobi identity \eqref{20180312:eq7}
follows from a straightforward but long computation,
similar but much longer than the analogous proof of \eqref{20180312:eq7} 
in Proposition \ref{20180312:prop}. 
We omit it.
\end{proof}
To distinguish between different Adler identities we add subscripts to the $\lambda$-brackets
as follows:
we shall denote by
$\{\cdot\,_\lambda\,\cdot\}^{(L)}_2$ or $\{\cdot\,_\lambda\,\cdot\}^{(L)}_3$,
a $\lambda$-bracket on $\mc V$ for which the pseudodifference operator $L(S)\in\mc V((S^{\pm1}))$
satisfies the 2-Adler identity \eqref{eq:adler} or the 3-Adler identity \eqref{eq:Adler3}
respectively.
One can easily check that, for $\epsilon\in\mb C$,
\begin{equation}\label{20180405:eq1}
\{\cdot\,_\lambda\,\cdot\}^{(L+\epsilon)}_{3}
=
\{\cdot\,_\lambda\,\cdot\}^{(L)}_{3}
+2\epsilon\{\cdot\,_\lambda\,\cdot\}^{(L)}_{2}
+\epsilon^2\{\cdot\,_\lambda\,\cdot\}^{(L)}_{1}
\,,
\end{equation}
and
\begin{equation}\label{20180405:eq2}
\{\cdot\,_\lambda\,\cdot\}^{(L+\epsilon)}_{2}
=
\{\cdot\,_\lambda\,\cdot\}^{(L)}_{2}
+2\epsilon\{\cdot\,_\lambda\,\cdot\}^{(L)}_{1}
\,,
\end{equation}
where the 1-Adler identity reads:
\begin{equation}\label{eq:Adler1}
\begin{array}{l}
\displaystyle{
\vphantom{\big(}
\{L(z)_\lambda L(w)\}^{(L)}_1
=
\delta_+\big(\frac{w\lambda S}{z}\big)
\big(\frac{w\lambda S}{z}L^*(\frac{\lambda}{z})-L(w)\big)
} \\
\displaystyle{
\vphantom{\big(}
-
\big(L(z)\frac{w\lambda}{z}-L(w\lambda)\big)
\delta_+\big(\frac{w \lambda}{z}\big)
\,.}
\end{array}
\end{equation}
In particular, Proposition \ref{20180312:prop}
can be obtained as a consequence of Proposition \ref{20180312:propB}
and equation \eqref{20180405:eq1}.
And an analogous Proposition can be stated for the 1-Adler identity \eqref{eq:Adler1}.
From Proposition \ref{20180312:propB} and the analogous result for the 1-Adler type identity, we get that Corollary
\ref{20180312:cor} holds also for the $\lambda$-brackets $\{\cdot\,_\lambda\,\cdot\}_3^{(L)}$
and  $\{\cdot\,_\lambda\,\cdot\}_1^{(L)}$.

Let $L(z)=\sum_{i\leq N}u_i z^i\in\mc V((z^{-1}))$
(respectively, $L(z)=\sum_{i\geq N}u_i z^i\in\mc V((z))$). By comparing the coefficients of $z^iw^j$, $i,j\leq N$ (respectively, $i,j\geq N$), in both sides
of the 1-Adler identity \eqref{eq:Adler1} we get
the following local $\lambda$-bracket relations among the coefficients of the pseudodifference operator $L(S)$:
\begin{equation}\label{eq:Adler1-coeff}
\{{u_i}_\lambda u_j\}_1^{(L)}=\epsilon_{ij}\big((\lambda S)^{-i}-\lambda^j\big)u_{i+j}
\,,
\end{equation}
where $\epsilon_{ij}=1$ if $i,j\geq1$, $\epsilon_{ij}=-1$ if $i,j\leq0$, and $\epsilon_{ij}=0$ otherwise.
These multiplicative $\lambda$-brackets 
should be compared with Example \ref{ex:mLCA2}.

\section{Integrable hierarchies associated to Adler type pseudodifference operators}\label{sec:4x}

\subsection{Integrable hierarchies associated to a 3-Adler type pseudodifference operator}

\begin{theorem}\label{thm:hn}
Let $L(S)\in\mc V((S^{\pm1}))$
be a pseudodifference operator over the multiplicative Poisson vertex algebra $\mc V$.
Assume that $L(S)$ satisfies the multiplicative $3$-Adler identity \eqref{eq:Adler3}.
Define the elements $h_{n}\in\mc V$, $n\in\mb Z_{\geq0}$, by 
\begin{equation}\label{eq:hn}
h_{n}=
-\frac{1}{n}
\mres_z L^n(z)
\text{ for } n\neq0
\,,\,\,
h_{0}=0\,.
\end{equation}
Then: 
\begin{enumerate}[(a)]
\item
All the elements $\tint h_{n}$ are Hamiltonian functionals in involution:
\begin{equation}\label{eq:invol}
\{\tint h_{m},\tint h_{n}\}_3^{(L)}=0
\,\text{ for all } m,n\in\mb Z_{\geq 0}
\,.
\end{equation}
\item
The corresponding hierarchy of compatible Hamiltonian equations satisfies
\begin{equation}\label{eq:hierarchy}
\frac{dL(w)}{dt_{n}}
=
\{\tint h_{n},L(w)\}_3^{(L)}
=
[(L^{n+1})_+,L](w)
\,,\,\,n\in\mb Z_{\geq 0}
\end{equation}
(in the RHS we are taking the symbol of the commutator of difference operators),
and the Hamiltonian functionals $\tint h_{n}$, $n\in\mb Z_{\geq 0}$,
are integrals of motion of all these equations.
\end{enumerate}
\end{theorem}

In the remainder of the section we will give a proof of Theorem \ref{thm:hn}. The proof is based on Lemma \ref{20180420:lem1}
and Lemma \ref{20180420:lem2} below, which are the multiplicative analogues of Lemmas 2.1 and 5.6 in \cite{DSKV16}.
The proof of these lemmas is similar. For example the proof of Lemma \ref{20180420:lem2}
uses Proposition \ref{prop:rules}.

\begin{lemma}\label{20180420:lem1}
Given two pseudodifference operators $A(S),B(S)\in\mc V((S^{\pm1}))$, we have
\begin{enumerate}[(a)] 
\item $\mres_z A(z)B^*(\frac \lambda z)=\mres_z A(z\lambda S)B(z)$;
\item $\tint\mres_z A(zS)B(z)=\tint\mres_zB(zS)A(z)$. 
\end{enumerate}
\end{lemma}

\begin{lemma}\label{20180420:lem2}
Let $\mc V$ be an mPVA with multiplicative $\lambda$-bracket $\{\cdot\,_\lambda\,\cdot\}$.
Let $L(S)\in\mc V((S^{\pm1}))$.
Let $h_{n}\in\mc V$ be given by \eqref{eq:hn}. 
Then, for $a\in\mc V$, $n\in\mb Z_{\geq 1}$, we have
\begin{equation}
\begin{split}\label{20180420:eq1}
&\{ {h_{n}}_{\lambda} a\}|_{\lambda=1}=-\mres_z\{L(zx)_x a\}\big(|_{x=S}L^{n-1}(z)\big)
\\
&\tint \{a_\lambda h_{n}\}|_{\lambda=1}
=-\int\mres_z\{a_\lambda L(wx)\}|_{\lambda=1}\big(|_{x=S}L^{n-1}(w)\big)
\,.
\end{split}
\end{equation}
\end{lemma}

\begin{proof}[Proof of Theorem \ref{thm:hn}]
Applying the second equation in \eqref{20180420:eq1} first,
and then the first equation in \eqref{20180420:eq1}, we get
\begin{equation}\label{eq:hn-pr1}
\{\tint h_{m},\tint h_{n}\}_3^{(L)}
= 
\int \mres_z \mres_w
\{L(zx)_x L(wy)\}
\big(\big|_{x=\partial}L^{m-1}(z)\big)
\big(\big|_{y=\partial}L^{n-1}(w)\big)
\,.
\end{equation}
We can now use the 3-Adler identity \eqref{eq:Adler3}, 
and the fact that $L(zS)L^{m-1}(z)=L^n(z)$ and $L(wS)L^{n-1}(w)=L^n(w)$,
to rewrite the RHS of \eqref{eq:hn-pr1} as
\begin{equation}\label{eq:hn-pr2}
\begin{array}{l}
\displaystyle{
\vphantom{\Big(}
\int \mres_z \mres_w
\Big( L(wS)L^{m}(z)\delta_+(\frac{wS}{z})L^*(\frac{1}{z})L^{n-1}(w)
} \\
\displaystyle{
\vphantom{\Big(}
-L(wS)L^m(z)\delta_{+}(\frac{wS}{z})\frac{wS}{z}L^{n}(w)
-L^m(z)\delta_+(\frac{wS}{z})L^*(\frac{1}{z})L^n(w)
} \\
\displaystyle{
\vphantom{\Big(}
+L(wS)L^{m-1}(z)\delta_{+}(\frac{wS}{z})\frac{wS}{z}L^*(\frac{1}{z})L^n(w)
\Big)
\,.
} 
\end{array}
\end{equation}
We use
Lemma \ref{20180420:lem1}(b) 
to rewrite the first term in the RHS of \eqref{eq:hn-pr2} as
\begin{equation}\label{eq:hn-pr3}
\int \mres_z \mres_w
L^{m}(z)\delta_+(\frac{wS}{z})L^*(\frac{1}{z})L^{n}(w)
\,.
\end{equation}
Hence, the first and third term in \eqref{eq:hn-pr2} sum to zero.
On the other hand, using Lemma \ref{20180420:lem1}(a) (with $\lambda=1$), 
we rewrite the last term in
\eqref{eq:hn-pr2} as
\begin{equation}\label{eq:hn-pr4}
\int \mres_z \mres_w
L(wS)L^{m}(z)\delta_+(\frac{wS}{z})\frac{wS}{z}L^{n}(w)
\,.
\end{equation}
Hence, the second and last term in \eqref{eq:hn-pr2} sum to zero, thus showing that the RHS of \eqref{eq:hn-pr1} vanishes
and proving part (a).

We are left to prove part (b).
We have
\begin{equation}\label{eq:proofb-1}
\begin{array}{l}
\displaystyle{
\vphantom{\Big(}
\{\tint h_{n},L(w)\}_3^{(L)}
=
\{{h_{n}}_\lambda L(w)\}_3^{(L)}\big|_{\lambda=1}
=
-\mres_z
\{L(zx)_x L(w)\}
\big(\big|_{x=S}L^{n-1}(z)\big)
} \\
\displaystyle{
\vphantom{\Big(}
=
-
\mres_z
\Big(L(wS)L^n(z)\delta_+(\frac{wS}{z})L^*(\frac{1}{z})
-L(wS)L^n(z)\delta_+(\frac{wS}{z})\frac{wS}{z}L(w)
} \\
\displaystyle{
\vphantom{\Big(}
-
L^{n}(z)\delta_+(\frac{wS}{z})L^*(\frac 1z)L(w)+L(wS)L^{n-1}(z)\delta_+(\frac{wS}{z})\frac{wS}{z}L^*(\frac{1}{z})L(w)
\Big)
} \\
\displaystyle{
\vphantom{\Big(}
=
-\mres_z
\Big(L(wS)L^{n+1}(z)\delta_+(\frac{w}{z})
-L(wS)L^n(z)\delta_+(\frac{wS}{z})\frac{wS}{z}L(w)
} \\
\displaystyle{
\vphantom{\Big(}
-
L^{n+1}(z)\delta_+(\frac{wS}{z})L(w)+L(wS)L^{n}(z)\delta_+(\frac{wS}{z})\frac{wS}{z}L(w)
\Big)
} \\
\displaystyle{
\vphantom{\Big(}
=L^{n+K}(wS)_+L(w)-L(wS)L^{n+K}(w)_+
\,.}
\end{array}
\end{equation}
In the second equality we used 
the first equation in \eqref{20180420:eq1},
in the third equality we used the 3-Adler identity \eqref{eq:Adler3}
and some algebraic manipulations,
in the third equality we used Lemma \ref{20180420:lem1}(a) (with $\lambda=1$),
in the fourth equality we used equation \eqref{20180203:eq2}.
This proves \eqref{eq:hierarchy} and completes the proof of the theorem.
\end{proof}

\subsection{Integrable hierarchies associated to 2-Adler and 1-Adler type pseudodifference operators}

The analogue of Theorem \ref{thm:hn} for 2-Adler and 1-Adler pseudodifference operators can be proved
by similar computations
(see also \cite{DSKV16} for the same computations
in the additive case): 

\begin{theorem}\label{thm:hn2}
Let $L(S)\in\mc V((S^{\pm1}))$
be a pseudodifference operator over an mPVA $\mc V$.
Assume that $L(S)$ satisfies the multiplicative $2$-Adler identity \eqref{eq:adler}
(respectively, the 1-Adler identity \eqref{eq:Adler1}).
Define the elements $h_{n}\in\mc V$, $n\in\mb Z_{\geq 0}$, by \eqref{eq:hn}.
Then: 
\begin{enumerate}[(a)]
\item
All the elements $\tint h_{n}$ are Hamiltonian functionals in involution:
\begin{equation}\label{eq:invol2}
\{\tint h_{m},\tint h_{n}\}_{1,2}^{(L)}=0
\,\text{ for all } m,n\in\mb Z_{\geq 0}
\,.
\end{equation}
\item
The corresponding hierarchy of compatible Hamiltonian equations satisfies
\begin{align}\label{eq:hierarchy2}
&\frac{dL(w)}{dt_{n}}
=
\{\tint h_{n},L(w)\}_2^{(L)}
=
[(L^{n})_+,L](w)
\,,\,\,n\in\mb Z_{\geq 0},\,
\\
&\label{eq:hierarchy3}
\Big(\text{respectively,}\,
\frac{dL(w)}{dt_{n}}
=
\{\tint h_{n},L(w)\}_1^{(L)}
=
[(L^{n-1})_+,L](w)
\,,
n\in\mb Z_{\geq 1},\,
\Big)
\end{align}
and the Hamiltonian functionals $\tint h_{n}$, $n\in\mb Z_{\geq 0}$, 
are integrals of motion of all these equations.
\end{enumerate}
\end{theorem}

\subsection{Tri-Adler pseudodifference operators and tri-Hamiltonian hierarchies}

Let $\mc V$ be a unital commutative associative algebra endowed with an automorphism $S$, and 
let $L(S)\in\mc V((S^{\pm1}))$ be a pseudodifference operator. We say that $L(S)$ is of tri-Adler type if
there exist mPVA $\lambda$-brackets  $\{\cdot\,_\lambda\,\cdot\}^{(L)}_i$, $i=1,2,3$, on $\mc V$ for which the pesudodifference operator
$L(S)\in\mc V((S^{\pm1}))$ is of $i$-Adler type, $i=1,2,3$. We say that $\mc V$ is a tri-mPVA if any linear combination
of the $\lambda$-brackets $\{\cdot\,_\lambda\,\cdot\}_i^{(L)}$, $i=1,2,3$, is an mPVA.
The following result follows from Theorems \ref{thm:hn} and \ref{thm:hn2}.

\begin{corollary}\label{cor:magri}
Let $\mc V$ be a unital commutative associative algebra endowed with an automorphism $S$.
Let $L(S)\in\mc V((S^{\pm1}))$ be an invertible
 tri-Adler type pseudodifference operator with respect to multiplicative $\lambda$-brackets
 $\{\cdot\,_\lambda\,\cdot\}_i^{(L)}$, $i=1,2,3$, on $\mc V$, and assume that $\mc V$ is a
 tri-mPVA.
The elements $h_{n}\in\mc V$, $n\in\mb Z_{\geq 1}$, given by \eqref{eq:hn}
satisfy the following generalized Lenard-Magri recurrence relation:
\begin{equation}\label{eq:LM-K}
\{\tint h_{n-1},L(z)\}_3^{(L)}
=
\{\tint h_{n},L(z)\}_2^{(L)}
=
\{\tint h_{n+1},L(z)\}_1^{(L)}
\,,
n\in\mb Z_{\geq 1}
\,.
\end{equation}
Hence, \eqref{eq:hierarchy} is a hierarchy of compatible tri-Hamiltonian equations
on the tri-mPVA $\mc V$.
Moreover, all the Hamiltonian functionals $\tint h_{n}$, $n\in\mb Z_{\geq 0}$,
are integrals of motion of all the equations of this hierarchy.
\end{corollary}

\section{Pseudodifference operators of generic type}\label{sec:5}

Let $N\geq1$. In this section we denote by 
$\bar{\mc V}$
the algebra of difference 
polynomials in infinitely 
many variables $u_i$, $i\in \mb Z_{\leq N}$.
Let also $\mc I$
be the difference ideal (i.e. the minimal $S$-invariant ideal) generated by the elements $u_N-1$,  
and let
$\mc V=\bar{\mc V}/\mc I$.
Note that $\mc V$
is isomorphic to the algebra of difference polynomials in the
$u_i$, $i\in \mb Z_{\leq N-1}$.
Furthermore, let
\begin{equation}\label{eq:Lbar}
\bar L(S):=\bar L_N(S)=\sum_{i\leq N}u_iS^{i}\in\bar{\mc V}((S^{-1}))
\,,
\end{equation}
and
\begin{equation}\label{eq:L}
L(S):= L_N(S)=S^N+\sum_{i\leq N-1}u_iS^{i}\in \mc V((S^{-1}))
\,,
\end{equation}
We call $L(S)$ the \emph{pseudodifference operator of generic type} of order $N$.

\subsection{Pseudodifference operators of generic type and 1-Adler type identity}\label{sec:5.1}
By the discussion in Section \ref{sec:3}, we have that (see equation \eqref{eq:Adler1-coeff})
the assignment ($i,j\leq N$)
$$
\{{u_i}_\lambda u_j\}_1^{(\bar L)}=\epsilon_{ij}\big((\lambda S)^{-i}-\lambda^j\big)u_{i+j}
\,,
$$
where $\epsilon_{ij}=1$ if $i,j\geq1$, $\epsilon_{ij}=-1$ if $i,j\leq0$, and $\epsilon_{ij}=0$ otherwise,
defines an mPVA structure on $\bar{\mc V}$.
Note that $\{{u_N}_\lambda u_j\}_1^{(\bar L)}=0$, for every $j\leq N$. Hence, $\mc I$ is a central ideal
and we have an induced mPVA structure on $\mc V$ given by ($i,j\leq N-1$)
\begin{equation}\label{eq:AdlerDirac-coeff}
\{{u_i}_\lambda u_j\}_1^{(L)}=\epsilon_{ij}\big((\lambda S)^{-i}-\lambda^j\big)u_{i+j}
\,,
\end{equation}
where in the RHS $u_N=1$ and $u_k=0$ for $k>N$.

The linear independence of the integrals of motion $\int h_n$, $n\in \mb Z_{n\geq 1}$, is proved in the same way as in \cite{DSKV15}. Thus, from Theorem \ref{thm:hn2} and the above discussion we get the following result.
\begin{theorem}\label{liberazione:thm1}
Let $L(S)$ be the pseudodifference operator of generic type for the algebra of difference polynomials $\mc V$, and
endow $\mc V$ with the mPVA structure given by \eqref{eq:AdlerDirac-coeff}. Then we have 
an integrable hierarchy of Hamiltonian equations in $\mc V$ given by \eqref{eq:hierarchy3}.
\end{theorem}

\subsection{Pseudodifference operator of generic type and 2-Adler type identity}\label{sec:5.2}

By Corollary \ref{20180312:cor} we have an mPVA structure on $\bar{\mc V}$ whose $\lambda$-brackets
$\{{u_i}_\lambda u_j\}_2^{(\bar L)}$, $i,j\leq N$, are given by the RHS of equation \eqref{eq:adler-coeff1}. 
The next result can be proved easily using the Adler type identity \eqref{eq:adler}
\begin{lemma}\label{20180504:prop2}
The following identities hold:
\begin{enumerate}[a)]
\item $\{{u_N}_\lambda \bar L(w)\}_2^{(\bar L)}
=\frac{1}{2}\big(\bar L(w\lambda S)-\bar L(w)\big)\big(1+(\lambda S)^{-N}\big)u_N$.
\item $\{\bar L(z)_\lambda u_N\}_2^{(\bar L)}
=\frac12u_N\big(1+(\lambda S)^N\big)\big(\bar L(z)-\bar L^*(\frac{\lambda}{z})\big)$.
\item $\{{u_N}_\lambda u_N\}_2^{(\bar L)}=\frac12u_N\big((\lambda S)^N-(\lambda S)^{-N}\big))u_N$.
\end{enumerate}
\end{lemma}
We can apply the Dirac reduction procedure of Section \ref{sec:1.7}, by the constraint $\theta=u_N$,  to get the following result.
\begin{proposition}\label{20170420:prop1}
Let $L(S)$ be the pseudodifference operator of generic type of order $N$ 
defined in equation \eqref{eq:L}.
Then, it defines an mPVA structure on $\mc V$ via the following
Dirac reduced 2-Adler type identity:
\begin{equation}
\begin{split}\label{eq:adler-dirac}
&
\{L(z)_\lambda L(w)\}_2^{(L)D}
=
L(w\lambda S)
\delta_+(\frac{w\lambda S}{z})L^*(z^{-1}\lambda)
-
L(z)
\delta_+(\frac{w\lambda S}{z})
L(w)
\\
&+
\big(L(w)(\lambda S)^N-L(w\lambda S)\big)\big((\lambda S)^N-1\big)^{-1}
\big(L(z)- L^*(z^{-1}\lambda)\big)
\,.
\end{split}
\end{equation}
\end{proposition}
\begin{proof}
Equation \eqref{eq:adler-dirac} follows from the 2-Adler identity \eqref{eq:adler}, equation \eqref{eq:dirac1} defining the
Dirac reduced $\lambda$-bracket and Lemma \ref{20180504:prop2}. By Theorem \ref{20180310:thm1}, equation
\eqref{eq:adler-dirac} defines an mPVA structure on $\mc V$.
\end{proof}
By comparing the coefficients of $z^iw^j$, $i,j\leq N-1$, in both sides of \eqref{eq:adler-dirac},
we get the following $\lambda$-brackets relations in $\mc V$
\begin{equation}
\begin{split}\label{20180503:eq5}
\{{u_i}_\lambda u_j\}_2^{(L)D}&=\sum_{n=0}^{N-i}
\big(u_{j-n}(\lambda S)^{j-i-n}u_{i+n}
-u_{i+n}(\lambda S)^{n}u_{j-n}\big) 
\\
&+u_j\left((\lambda S)^N-(\lambda S)^j\right)
\left((\lambda S)^N-1\right)^{-1}\left(1-(\lambda S)^{-i}\right)u_i
\,,
\end{split}
\end{equation}
where $u_N=1$ and $u_k=0$ for $k>N$.
The local Poisson brackets on $\mc V$ corresponding to \eqref{20180503:eq5} (up to a constant factor)
have already appeared in \cite{Carlet}.

Note that the RHS of \eqref{20180503:eq5} is local only for $N=1$. The following result follows by a straightforward computation.
\begin{lemma}\label{liberazione:lem1}
Let $N=1$ and $h_n$ be defined as in \eqref{eq:hn}. Then, we have
$$
\{{h_n}_\lambda \bar L(w)\}_2^{(\bar L)}|_{\lambda=1}
=\{{h_n}_\lambda \bar L(w)\}_2^{(\bar L)D}|_{\lambda=1}
\,.
$$
\end{lemma}

From Theorem \ref{thm:hn2} and Lemma \ref{liberazione:lem1} we get the following result.
\begin{theorem}\label{liberazione:thm2}
Let $N=1$. Let $L(S)$ be the pseudodifference operator of generic type for the algebra of difference polynomials $\mc V$, and
endow $\mc V$ with the mPVA structure given by \eqref{20180503:eq5}. Then we have 
an integrable hierarchy of Hamiltonian equations in $\mc V$ given by \eqref{eq:hierarchy2}.
\end{theorem}

\subsection{Pseudodifference operator of generic type and 3-Adler type identity}\label{sec:5.2b}

The 3-Adler type identity \eqref{eq:Adler3} is not consistent for the pseudodifference operator $\bar L(S)\in\bar{\mc V}((S^{-1}))$. 
Indeed, the LHS of \eqref{eq:Adler3} has powers of $z$ bounded above by $N$, while the RHS of \eqref{eq:Adler3} contains powers
of $z$ of order greater than $N$. Hence, we can not use the operator $\bar L(S)$ to define an mPVA structure on $\bar{\mc V}$ using 
the 3-Adler type identity. However, similarly to what was done in Section \ref{sec:5.2}, we can perform a Dirac modification to
get an mPVA
structure on $\mc V$ given by the pseudodifference operator of generic type $L$.

We illustrate this procedure in the case of $N=1$. Let $\mc V$
be the algebra of difference polynomials in the $u_i$, $i\in \mb Z_{\leq 0}$, and let
$L(S)=S+\sum_{i\in\mb Z_{\geq0}}u_{-i}S^{-i}\in\mc V((S^{-1}))$ be the generic pseudodifference operator of order $1$.
Denote by $H_3^{(L)}(\lambda)(w,z)$ the RHS of \eqref{eq:Adler3}. 
Note that, $H_3^{(L)}(\lambda)(w,z)=\sum_{i,j\leq 2}(H_{3}^{(L)})_{ji}(\lambda)z^iw^j$.
On the other hand, $\{L(z)_\lambda L(w)\}=\sum_{i,j\leq 0}\{{u_i}_\lambda u_j\}z^iw^j$. Hence, the 3-Adler identity
\eqref{eq:Adler3} is not consistent.

Let $H_3^{(L)}(S)=\big((H_{3}^{(L)})_{ij}(S)\big)_{i,j\leq 2}$ and write it as a matrix in blocks form as follows
$$
H_3^{(L)}(S)=\begin{pmatrix}
A(S)&B(S)
\\
C(S)&
D(S)
\end{pmatrix}
\,,
$$
where
\begin{align*}
&A(S)=\big((H_3^{(L)}(S))_{i,j}\big)_{i=1,2; j=1,2}\,,
&
&
B(S)=\big((H_3^{(L)}(S))_{i,j}\big)_{i=1,2; j\leq0}
\,,
\\
&C(S)=\big((H_3^{(L)}(S))_{i,j}\big)_{i\leq0; j=1,2}
\,,
&
&D(S)=\big((H_3^{(L)}(S))_{i,j}\big)_{i\leq 0; j\leq0}
\,.
\end{align*}
We are interested in computing
the generating series for
the entries of the matrix pseudodifference operator obtained by taking the quasideterminant 
of $H_3^{(L)}(S)$ with respect to the block $D(S)$ (see end of Section \ref{sec:1.7})
\begin{equation}\label{20180314:eq1}
(H_3^{(L)})^D(S)=
\begin{pmatrix}
0 & 0
\\
0&D(S)-C(S)\circ A(S)^{-1}\circ B(S)
\end{pmatrix}
\,.
\end{equation}
From equation \eqref{eq:Adler3} we get, by a straightforward computation,
\begin{equation}
\begin{split}\label{20180313:eq3} 
&\sum_{i\leq 2}(H_3^{(L)}(\lambda ))_{i2}z^i=\big(L(z\lambda)-L(z)\big)\lambda^{-1}\,,
\\
&\sum_{j\leq 2}(H_3^{(L)}(\lambda))_{2j}w^j=\lambda S\big(L(w)-L^*(w^{-1}\lambda)\big)
\,,
\\
&
\sum_{i\leq 2}(H_3^{(L)}(\lambda))_{i1}z^i
=L(z\lambda S)\big(z+u_0(1+\lambda^{-1})\big)-\big(zS+u_0(1+\lambda^{-1})\big)L(z)\, ,
\\
&
\sum_{i\leq 2}(H_3^{(L)}(\lambda))_{1j}w^j
=\big(wS+(1+\lambda S)u_0)L(w)-\big(w+(1+\lambda S)u_0\big)L^*(w^{-1}\lambda)
\,.
\end{split}
\end{equation}
From equations \eqref{20180313:eq3} we immediately get
$$
A(S)=\begin{pmatrix}
0& S-1\\
1-S^{-1}& S\circ u_0-u_0 S^{-1}
\end{pmatrix}
\,,
$$
whose inverse is
\begin{equation}\label{20180313:eq4}
A(S)^{-1}=
\begin{pmatrix}
-(1-S^{-1})^{-1}\circ(S\circ u_0-u_0S^{-1})\circ(S-1)^{-1}
&
(1-S^{-1})^{-1}
\\
(S-1)^{-1}&0
\end{pmatrix}
\,.
\end{equation}
Let $(H_3^{(L)})^D(\lambda)(w,z)=\sum_{i,j\leq 2}\big((H_{3}^{(L)})^D\big)_{ji}(\lambda)z^iw^j$ be the generating series for the symbol of entries
of the matrix pseudodifference operator $(H_3^{(L)})^D$. 
Since, by equation \eqref{20180314:eq1}, $(H_3^{(L)})^D(\lambda)(w,z)=\sum_{i,j\leq 0}\big((H_{3}^{(L)})^D\big)_{ji}(\lambda)z^iw^j$,
we get a consistent Dirac modified $3$-Adler identity.
Explicitly,
using equations \eqref{eq:Adler3}, \eqref{20180314:eq1}, \eqref{20180313:eq3}
and \eqref{20180313:eq4}:
\begin{equation}
\begin{split}\label{eq:monster}
&\{L(z)_\lambda L(w)\}_3^{(L)D}=(H_3^{(L)})^D(\lambda)(w,z)
\\
&=L(w \lambda S)L(z)\big(1-\frac{w \lambda S}{z}\big)^{-1}\big(L^*(z^{-1} \lambda)-\frac{w \lambda S}{z}L(w)\big)
\\
&-\big(L(z)-L(w \lambda S)\frac{w \lambda S}{z}\big)\big(1-\frac{w \lambda S}{z}\big)^{-1}L(w)L^*(z^{-1} \lambda)
\\
&-\big(L(w\lambda S)w\lambda S-w\lambda SL(w)\big)\big(\lambda S-1\big)^{-1}\big(L(z)-L^*(z^{-1}\lambda)\big)
\\
&+\big(L(w\lambda S)-L(w)\big)\big((\lambda S)^{-1}-1\big)^{-1}\big(L(z)z\lambda^{-1}-z(\lambda S)^{-1}L^*(z^{-1}\lambda)\big)
\\
&-\big(L(w\lambda S)-L(w)\big)\big(\lambda S-1\big)^{-1}u_0\big(L(z)-L^*(z^{-1}\lambda)\big)
\\
&+\big(L(w\lambda S)-L(w)\big)u_0\big((\lambda S)^{-1}-1\big)^{-1}\big(L(z)-L^*(z^{-1}\lambda)\big)
\,.
\end{split}
\end{equation}
Skewsymmetry \eqref{20180312:eq6} and Jacobi identity \eqref{20180312:eq7}
for the multiplicative $\lambda$-bracket \eqref{eq:monster} 
follow by a straightforward (but long) computation that we omit
(note that $L(S)$ is not an operator of 3-Adler type, so we cannot apply
Theorem \ref{20180310:thm1}).
Hence, we have an mPVA structure on $\mc V$
given by the Dirac modified 3-Adler type identity \eqref{eq:monster}.
By comparing powers of $z^iw^j$, $i,j\leq 0$ in both sides of \eqref{eq:monster}
we get the following $\lambda$-brackets relations in $\mc V$
\begin{align}
\begin{split}\label{monster_explicit}
&\{{u_i}_\lambda u_j\}_3^{(L)D}
\\
& =\sum_{\substack{a\leq i\\b\leq 1}}
\big(u_a(\lambda S)^a u_b(\lambda S)^{b-j}u_{i+j-a-b}
-u_b(\lambda S)^{i-a} u_a(\lambda S)^{a+b-i-j}u_{i+j-a-b}
\big)
\\
&-\sum_{\substack{j+1\leq a\leq 1\\ b\leq 1}}
\big(
u_{i+j-a-b}(\lambda S)^{i+j-a-b} u_a(\lambda S)^{a-j}u_{b}
-u_{i+j-a-b}(\lambda S)^{i-b} u_b(\lambda S)^{-a}u_{a}
\big)
\,.
\\
&-\big(u_{j-1}(\lambda S)^{j-1}-\lambda S u_{j-1}\big)\big(\lambda S-1\big)^{-1}\big(1-(\lambda S)^{-i}\big)u_i
\\
&
+u_{j}\big((\lambda S)^{j}-1\big)\big((\lambda S)^{-1}-1\big)^{-1}\big(\lambda^{-1}-(\lambda S)^{-i}\big)u_{i-1}
\\
&
-u_{j}\big((\lambda S)^{j}-1\big)\Big(\big(\lambda S-1\big)^{-1}u_0-u_0\big((\lambda S)^{-1}-1\big)^{-1}\Big)\big(1-(\lambda S)^{-i}\big)u_{i}
\,,
\end{split}
\end{align}
where $u_1=1$ and $u_k=0$ for $k>1$. This agrees, up to a constant factor, with formulas in \cite{Carlet2}.

Similarly to the arguments provided in the previous section, from Theorem \ref{thm:hn} and Corollary \ref{cor:magri}
 we get the following result.
\begin{theorem}\label{liberazione:thm3}
Let $N=1$. Let $L(S)$ be the generic pseudodifference operator for the algebra of difference polynomials $\mc V$, and
endow $\mc V$ with mPVA structure given by \eqref{monster_explicit}. Then we have 
an integrable hierarchy of Hamiltonian equations in $\mc V$ given by \eqref{eq:hierarchy}.
Moreover, the three mPVA structures on $\mc V$ given by \eqref{eq:AdlerDirac-coeff}, \eqref{20180503:eq5} and \eqref{monster_explicit}
are compatible and the Lenard-Magri recursion relations \eqref{eq:LM-K} hold. 
Hence, \eqref{eq:hierarchy} is a compatible hierarchy of tri-Hamiltonian equations
on the tri-mPVA $\mc V$.
\end{theorem}

\begin{remark}\label{rem:20180420}
We can similarly define a generic pseudodifference operator in $\mc V((S))$. Similar results as the ones proved in this section still hold.
Moreover, for $M,N\geq1$, all the results of this section can be proved starting from a difference operator of the form
$$
\bar L(S)=\bar L_{M,N}(S)=\sum_{i=-M}^Nu_iS^i\in\bar{\mc V}[S,S^{-1}]
\,,
$$
where $\bar{\mc V}$ is the algebra of difference polynomials in the variables $u_i$, where
$-M\leq i\leq N$, see \cite{Carlet}. In particular equations \eqref{eq:adler-coeff1}
\eqref{eq:Adler1-coeff}, \eqref{20180503:eq5}, for $-M\leq i,j\leq N$, subject to the condition $u_k=0$ for $k<-M$,
still define mPVAs which were previously studied, in terms of Poisson algebras, in \cite{BM94,MS96,Carlet}. The same is true for $N=1$ for equation \eqref{monster_explicit}.
\end{remark}

\subsection{The multiplicative $W$-algebra $\mc W_N$ and its local subalgebra}
\label{sec:5.2c}

Let ${\mc V}_N$ be the algebra of difference polynomials in $u_0,u_1,\dots,u_{N-1}$,
where $N\geq2$.
Then formula \eqref{20180503:eq5} defines on it the following rational multiplicative Poisson
$\lambda$-bracket:
\begin{equation}\label{eq:mult}
\begin{split}
\{{u_i}_\lambda u_j\}
&=
u_0(\lambda S)^{-i}u_{i+j}-u_{i+j}(\lambda S)^j u_0
\\
&+
\Big(\sum_{n=1}^j-\sum_{n=i}^{i+j-1}\Big)
u_n(\lambda S)^{n-i}u_{i+j-n}
\\
&+u_j\left((\lambda S)^N-(\lambda S)^j\right)
\left((\lambda S)^N-1\right)^{-1}\left(1-(\lambda S)^{-i}\right)u_i
\,,
\end{split}
\end{equation}
where 
\begin{equation}\label{eq:mult2}
u_N=1\,\,\text{ and }\,\,
u_k=0
\,\,\text{ for }\,\,
k>N
\,.
\end{equation}

It is clear from this formula  that $u_0$ is a central element. Hence we can further reduce
by the difference ideal generated by $u_0-c$, where $c$ is a constant.
As a result we get the multiplicative $W$-algebra $\mc W_N$,
which is the algebra of difference polynomials in $u_1,\dots,u_{N-1}$,
with the following family of multiplicative rational Poisson $\lambda$-brackets:
$$
\{{u_i}_{\lambda}{u_j}\}
=
c\{{u_i}_{\lambda}{u_j}\}_1+\{{u_i}_{\lambda}{u_j}\}_2
\,,
$$
where
\begin{equation}\label{eq:wak1}
\{{u_i}_\lambda u_j\}_1
=
\big((\lambda S)^{-i}-\lambda^j\big) u_{i+j}
\,,
\end{equation}
and
\begin{equation}\label{eq:wak2}
\{{u_i}_\lambda u_j\}_2
=
\Big(\sum_{n=1}^j-\sum_{n=i}^{i+j-1}\Big)
u_n(\lambda S)^{n-i}u_{i+j-n}
+u_j
\frac{\big((\lambda S)^N-(\lambda S)^j\big)\big(1-(\lambda S)^{-i}\big)}{(\lambda S)^N-1}
u_i
\,,
\end{equation}
subject to \eqref{eq:mult2}.
The first Poisson structure has already appeared in Example \ref{ex:mLCA2},
while the second Poisson structure corresponds to the $q$-deformed $W$-algebras
of Frenkel and Reshetikin \cite{FR96}
(as discussed in Example \ref{ex:qdef} for $N=2$).

These two compatible multiplicative Poisson structures for $N=2$ and $3$
have been discussed in Examples \ref{ex:5.6} and \ref{ex:lattice_sl3} respectively.
In the first case we constructed a local mPVA subalgebra
corresponding to the local lattice Poisson algebra
of Faddeev-Takhtajan-Volkov.
%
The main result of the present section
is the generalization of this construction to arbitrary $N\geq 3$. It is proved by a direct computation.
\begin{theorem}\label{thm:wakimoto}
Let $N\geq3$.
Consider the difference subalgebra $\mc A_N$ of the algebra $\mc W_N$ localized by $u_{1}$,
generated by the following elements:
\begin{equation}
  \label{eq:wak0}
  \begin{split}
v_1
&=
\frac{1}{u_{1}(Su_{1})\dots(S^{N-1}u_{1})},
\\
v_i
&=
\frac{u_{i}}{u_{1}(Su_{1})\dots(S^{i-1}u_{1})}
\,\,,\,\,\,\,2\leq i\leq N-1
\,.
\end{split}
\end{equation}
%
Then both multiplicative Poisson $\lambda$-brackets \eqref{eq:wak1} and \eqref{eq:wak2},
restricted to the subalgebra $\mc A_N$, are local.
\end{theorem}
Formula \eqref{eq:wak0} is the generalized Miura transformation, introduced in \cite{MBW13}.
We can write explicit formulas
for both multiplicative $\lambda$-brackets $\{{v_i}_\lambda{v_j}\}_1$
and $\{{v_i}_\lambda{v_j}\}_2$ of any two generators $v_i$ and $v_j$
($1\leq i,j\leq N-1$).
In all these formulas we assume that $v_N=v_1$ and $v_i=0$ for $i>N$.
For the first $\lambda$-bracket, we have:
\begin{enumerate}[(i)]
\item
$$
\{{v_1}_\lambda v_1\}_1
=v_1\frac{1-(\lambda S)^N}{1-\lambda S}\big((\lambda S)^{-1}v_2-v_2\lambda S\big)\frac{1-(\lambda S)^{-N}}{1-(\lambda S)^{-1}}v_1
\,;
$$
\item
for $2\leq j\leq N-1$:
\begin{align*}
\{{v_j}_\lambda v_1\}_1
&=v_1\frac{1-(\lambda S)^N}{1-\lambda S}\big((\lambda S)^{-1}v_2-v_2\lambda S\big)\frac{1-(\lambda S)^{-j}}{1-(\lambda S)^{-1}}v_j
\\
&+v_1\frac{1-(\lambda S)^N}{1-\lambda S}\big(\lambda -(\lambda S)^{-j}\big)v_{j+1}
\,;
\end{align*}
\item
for $2\leq i,j\leq N-1$:
\begin{align*}
\{{v_j}_\lambda v_i\}_1
&=v_i\frac{1-(\lambda S)^i}{1-\lambda S}\big((\lambda S)^{-1}v_2-v_2\lambda S\big)\frac{1-(\lambda S)^{-j}}{1-(\lambda S)^{-1}}v_j
\\
&+v_i\frac{1-(\lambda S)^i}{1-\lambda S}\big(\lambda -(\lambda S)^{-j}\big)v_{j+1}
\\
&+\big(v_{i+1}(\lambda S)^i-(\lambda S)^{-1}v_{i+1}\big)\frac{1-(\lambda S)^{-j}}{1-(\lambda S)^{-1}}v_j
+\big((\lambda S)^{-j}-\lambda^i\big)v_{i+j}\,.
\end{align*}
\end{enumerate}
For the second $\lambda$-bracket, we have
\begin{enumerate}[(i)]
\item
$$
\{{v_1}_\lambda v_1\}_2
=v_1\frac{\big(1-(\lambda S)^{1-N}\big)\big(1-(\lambda S)^N\big)}{1-\lambda S}v_1
\,;
$$
\item
for $2\leq j\leq N-1$:
$$
\{{v_j}_\lambda v_1\}_2=v_1\frac{\big(1-(\lambda S)^{1-j}\big)\big(1-(\lambda S)^N\big)}{1-\lambda S}v_j
\,;
$$
\item
for $2\leq i,j\leq N-1$ such that $i+j\leq N+1$:
\begin{align*}
\{{v_j}_\lambda v_i\}_2
&=v_i\frac{\lambda S\big(1-(\lambda S)^{i-1}\big)\big(1-(\lambda S)^{-j}\big)}{1-\lambda S}v_j
\\
&-\big(\sum_{p=2}^j-\sum_{p=i}^{i+j-2}\big)v_{i+j-p}(\lambda S)^{i-p}v_p
+\big((\lambda S)^{1-j}-\lambda^{i-1}\big)v_{i+j-1}
\,;
\end{align*}
\item
for $2\leq i,j\leq N-1$ such that $i+j\geq N+2$:
\begin{align*}
\{{v_j}_\lambda v_i\}_2
&=v_i\frac{\lambda S\big(1-(\lambda S)^{i-1}\big)\big(1-(\lambda S)^{-j}\big)}{1-\lambda S}v_j
\\
&-\big(\sum_{p=i+j-N}^j-\sum_{p=i}^{N}\big)v_{i+j-p}(\lambda S)^{i-p}v_p
\,.
\end{align*}
\end{enumerate}

Let $H(S)$ and $K(S)$ be the multiplicative Poisson structures corresponding to the
$\lambda$-brackets from Theorem \ref{thm:wakimoto}:
$$
H(S)=\big({\{{v_j}_S v_i\}_1}_\rightarrow\big)_{i,j=1}^{N-1}
\quad
\text{and}
\quad
K(S)=\big({\{{v_j}_S v_i\}_2}_\rightarrow\big)_{i,j=1}^{N-1}
\,.
$$
Let
$$
\xi_0=(\frac{1}{Nv_1},0,\dots,0)^T
\quad
\text{and}
\quad
\xi_1=(0,-1,0,\dots,0)^T
\,.
$$
Then,
$$
K(S)\xi_0=0\,,
\qquad
K(S)\xi_1=H(S)\xi_0
\,,
$$
and we have the bi-Hamiltonian equation $\frac{dv}{dt_0}=H(S)\xi_0$. Explicitly,
\begin{equation}
\begin{split}\label{20180814:eq1}
\frac{dv_1}{dt_0}&=v_1\big( S^{-1}(v_2)-S^{N-1}(v_2)\big)\,,
\\
\frac{dv_i}{dt_0}&=v_i\big( S^{-1}(v_2)-S^{i-1}(v_2)\big)+v_{i+1}-S^{-1}(v_{i+1})\,,
\quad2\leq i\leq N-1\,.
\end{split}
\end{equation}
For $N>3$ this is the Mari Beffa-Wang lattice.
For $N=3$ this is
the Belov-Chaltikian lattice \cite{BC93}. It is unclear how to prove that the Lenard-Magri sequence extends to infinity, which
would prove integrability of this equation.

\section{Further examples of integrable Hamiltonian systems of differential-difference equations}
\label{sec:last}

\subsection{The Toda lattice}
Let $\mc V$
be the algebra of difference polynomials in two variables $u$ and $v$.
Let
\begin{equation}\label{L:Toda0}
L(S)=S+v+uS^{-1}\in\mc V[S,S^{-1}]
\,.
\end{equation}
By Remark \ref{rem:20180420}, equations \eqref{eq:monster}, \eqref{eq:adler-dirac}, \eqref{eq:Adler1}, Propositions \ref{20180312:prop} and 
\ref{20180312:propB}, and Theorem \ref{20180310:thm1},
we have the following three compatible mPVA structures on $\mc V$:
\begin{align}
\begin{split}\label{20180409:eq1}
&\{v_\lambda v\}_3=(\lambda^{-1}-\lambda S)uv-v\lambda Su+u(\lambda S)^{-1}v\,,
\\
&\{u_\lambda v\}_3=v^2(1-\lambda S)u+(1-\lambda S)u^2+u(\lambda S)^{-1}u-\lambda Su\lambda Su
\,,
\\
&\{u_\lambda u\}_3=2u(\lambda S)^{-1}uv-2uv\lambda S u
\,;
\end{split}
\\
\begin{split}\label{20180409:eq2}
&\{v_\lambda v\}_2=(\lambda^{-1}-\lambda S)u\,,
\qquad
\{u_\lambda v\}_2=v(1-\lambda S)u
\,,
\\
&\{u_\lambda u\}_2=u\big((\lambda S)^{-1}-\lambda S\big) u
\,;
\end{split}
\\
\begin{split}\label{20180409:eq3}
&\{v_\lambda v\}_1=0=\{u_\lambda u\}_1\,,
\qquad
\{u_\lambda v\}_1=(1-\lambda S)u
\,.
\end{split}
\end{align}
The mPVA structures defined by equations \eqref{20180409:eq1}, \eqref{20180409:eq2} and \eqref{20180409:eq3}
are known as the three compatible Poisson structures of the Toda lattice, see \cite{Kup85}.
By an explicit computation we have 
\begin{equation}\label{20180409:eq4}
L(S)^2=S^2+(v+v_1)S+u+u_1+v^2+(uv+uv_{-1})S^{-1}+uu_{-1}S^{-2}
\,.
\end{equation}
Then, by equations \eqref{eq:hn}, \eqref{L:Toda0} and \eqref{20180409:eq4} the first two non-trivial integrals of 
motion \eqref{eq:hn} are
$$
\tint h_1=-\tint v
\qquad
\text{and}
\qquad
\tint h_2=-\frac{1}{2}\tint  2u+v^{2}
\,.
$$
By a straightforward computation, using equations \eqref{20180409:eq3} and \eqref{20180409:eq2} we have
\begin{equation}\label{20180409:lenard1}
\{\tint h_1,u\}_2=u(v-v_{-1})=\{\tint h_2,u\}_1\,,
\quad
\{\tint h_1,v\}_2=u_1-u=\{\tint h_2,v\}_1
\,.
\end{equation}
Furthermore, using equations \eqref{20180409:eq2} and \eqref{20180409:eq1} we have
\begin{equation}
\begin{split}\label{20180409:lenard2}
&\{\tint h_1,u\}_3=u(u_1-u_{-1}+v^2-v_{-1}^2)=\{\tint h_2,u\}_2\,,
\\
&\{\tint h_1,v\}_3=u_1(v+v_1)-u(v+v_{-1})=\{\tint h_2,v\}_2
\,.
\end{split}
\end{equation}
From equations \eqref{20180409:lenard1} and \eqref{20180409:lenard2} the first two non-trivial equations of the hierarchy \eqref{eq:hierarchy} are
$$
\frac{du}{dt_1}=u(v-v_{-1})\,,
\qquad
\frac{dv}{dt_1}=u_1-u
$$
and
$$
\frac{du}{dt_2}=u(u_1-u_{-1}-v_{-1}^2+v^2)\,,
\qquad
\frac{dv}{dt_2}=u_1(v+v_1)-u(v+v_{-1})
\,.
$$
Note that, letting $v=\frac{d q}{dt_1}$ and $u=e^{q-q_{-1}}$, the first equation becomes the Toda lattice (see \cite{CDZ04,Mik})
$$
\frac{d^2 q_n}{d t_1^2}=e^{q_{n+1}-q_n}-e^{q_n-q_{n-1}}
\,,
\quad
n\in\mb Z
\,.
$$

\subsection{The Volterra lattice}
Let $\mc V$ be the algebra of of difference polynomials in two variables $u$ and $v$
with the second mPVA structure $\{\cdot\,_\lambda\,\cdot\}:=\{\cdot\,_\lambda\,\cdot\}_2$ given by equations \eqref{20180409:eq2}.
Let $\mc I \subset \mc V$ be the difference algebra ideal generated by the variable $v$.
Let $\tilde{\mc V}=\mc V/\mc I$ and let $\pi:\mc V\twoheadrightarrow\tilde{\mc V}$ be the
quotient map.

Let $\theta=v$, and let us apply the Dirac reduction procedure explained in Section \ref{sec:1.7}. Using equations
\eqref{eq:dirac1} and \eqref{20180409:eq2} we get that the Dirac reduced mPVA structure on $\tilde{\mc V}$ is given by 
\begin{equation}\label{20180410:eq1}
\{u_\lambda u\}^D=u((\lambda S)^{-1}-\lambda S)u
\,.
\end{equation}
Note that $\{u_\lambda u\}^D=\pi\{u_\lambda u\}$.
Hence, we have $\{\pi(a)_\lambda\pi(b)\}^D=\pi\{a_\lambda b\}$, for every $a,b\in\mc V$. In particular,
$\{\tint \pi(h_n),\tint\pi(h_m)\}^D=0$, for every $n,m\geq0$, where the integrals of motion $\tint h_n$ are defined by \eqref{eq:hn}.
Since $\tilde L(z):=\pi(\tilde L(z))=z+uz^{-1}$ is an odd Laurent polynomial, we have that $\tint\pi(h_n)=0$ for every odd integer $n\geq1$. The corresponding Dirac reduced integrable hierarchy for $\tilde{\mc V}$ has the form
$$
\frac{d \tilde L(z)}{dt_{n}}=[\tilde L^{2n}_+,\tilde L](z)\,,
\qquad
n\geq 1
\,,
$$
the first equation being the Volterra lattice:
$$
\frac{du}{dt_1}=u(u_1-u_{-1})
\,.
$$ 
In the case of the first mPVA structure, given by
equation \eqref{20180409:eq3}, it is impossible to define the Dirac modified
$\lambda$-bracket with $\theta=v$. In the case of the third mPVA structure, given by \eqref{20180409:eq1},
it is possible, but the Dirac reduced $\lambda$-bracket is not defined.
However the Volterra lattice is a
bi-Hamiltonian equation (see, e.g., \cite{Mik} and \cite{DSKVW18}), but we do not know
how to obtain the other Poisson structure along the above lines, and how
to prove the corresponding Conjecture 7.4 from \cite{DSKVW18}.

\subsection{The Bogoyavlensky lattice}

Let $p\geq 1$, and consider the algebra of differece polynomials 
$\mc V$
in the $p+1$ variables $u_i$, $i=0,\dots,p$.
Let
\begin{equation}\label{L:Toda}
L(S)=S+\sum_{i=0}^pu_iS^{-i}
\in\mc V[S,S^{-1}]
\,.
\end{equation}
and endow $\mc V$ with the mPVA structure $\{\cdot\,_\lambda\,\cdot\}:=\{\cdot\,_\lambda\,\cdot\}_2$ given by equation
\eqref{eq:adler-dirac}.
Let $\mc I
\subset \mc V$ be the difference algebra ideal generated by the variables
$u_{i}$, $i=0,\dots,p-1$.
Let $\tilde{\mc V}=\mc V/\mc I$, let $\pi:\mc V\twoheadrightarrow\tilde{\mc V}$ be the projection map, and let $u=\pi (u_p)$. Then $\tilde{\mc V}$ is the algebra of difference polynomials in $u$.
It is clear from equations \eqref{eq:dirac1} and \eqref{20180503:eq5} that the induced Dirac reduced mPVA structure on $\tilde{\mc V}$ is given by
\begin{equation}\label{20180410:eq2}
\{u_\lambda u\}^D=\pi\{u_\lambda u\}
=u\Big(\frac{\big((\lambda S)^{p+1}-1\big)\big(1-(\lambda S)^p\big)}{(\lambda S)^p\big(\lambda S-1\big)}\Big)u
\,.
\end{equation}
Hence, we have $\{\pi(a)_\lambda\pi(b)\}^D=\pi\{a_\lambda b\}$, for every $a,b\in\mc V$. In particular,
we get $\{\tint \pi(h_n),\tint\pi(h_m)\}^D=0$, for every $n,m\geq0$, where the integrals of motion $\tint h_n$ are defined by \eqref{eq:hn}.
Since $\tilde L(z):=\pi(L(z))=z+uz^{-p}$, it is easy to check that $\tint\pi(h_n)=0$ if $n$ is not a multiple of $p+1$. The corresponding Dirac reduced integrable hierarchy for $\bar{\mc V}$ has the form
$$
\frac{d \tilde L(z)}{dt_{n}}=[\tilde L^{n(p+1)}_+,\tilde L](z)\,,
\qquad
n\geq 1
\,.
$$
The first equation is known as the Bogoyavlensky lattice:
$$
\frac{du}{dt_1}=\sum_{i=0}^{p}u(u_i-u_{-i})
\,.
$$ 
This is easily computed using \eqref{20180410:eq2} and the fact that $\tint \pi(h_{p+1})=-\tint u$.

\subsection{The discrete KP}

Let $\mc V$ be the algebra of difference polynomials in infinitely many variables
$u_i$, $i\in \mb Z_{\leq0}$.
Let 
\begin{equation}\label{L:KP}
L(S)=S+\sum_{i\leq0}u_{i}S^{i}\in\mc V((S^{-1}))
\end{equation}
be the generic pseudodifference operator of order $1$.
Equations \eqref{eq:monster}, \eqref{eq:adler-dirac} and \eqref{eq:Adler1}
(equivalently, equations \eqref{monster_explicit}, \eqref{20180503:eq5} and \eqref{eq:Adler1-coeff} for $N=1$)
define three compatible mPVA structures on $\mc V$. Explicit formulas for the
first and the second $\lambda$-brackets can be found in the next Section. 
The corresponding tri-integrable hierarchy (cf. \eqref{eq:hierarchy})
\begin{equation}\label{L:KP2}
\frac{d L(z)}{dt_n}=[(L^{n})_+,L](z)\,,
\qquad
n\geq1
\,,
\end{equation}
is the \emph{discrete KP hierarchy}, see \cite{AvM99}. Note that $L(z)_+=z+u_0$. Hence, the first equation of the hierarchy is
$$
\frac{d L(z)}{dt_1}=z(S-1)L(z)+(L(z)-L(zS))u_0
\,,
$$
namely
$$
\frac{d u_i}{dt_1}=(S-1)u_{i-1}+u_i(1-S^{i})u_0
\,,
\qquad
i\leq0
\,.
$$
Furthermore, $L^2(z)_+=z^2+(S+1)u_0z+(S+1)u_{-1}+u_0^2$. Hence, the second equation of the hierarchy is
\begin{align*}
\frac{d L(z)}{dt_2}&=z^2(S^2-1)L(z)+z(Su_0+u_0S)L(z)-zL(zS)(S+1)u_0
\\
&+
(L(z)-L(zS))((S+1)u_{-1}+u_0^2)
\,.
\end{align*}
Explicitly, the latter equation is ($i\leq0$)
$$
\frac{d u_i}{dt_2}=(S^2-1)u_{i-2}+(Su_0+u_0S)u_{i-1}-u_{i-1}(S+1)S^{i-1}u_0
+u_i(1-S^{i})((S+1)u_{-1}+u_0^2)
\,.
$$

\subsection{Multiplicative Poisson $\lambda$-bracket for the two-dimensional Toda hierarchy}

Recall that the two-dimensional Toda hierarchy \cite{UT84}
is the hierarchy of Lax equations on the pseudodifference operators
$$
L(S)=S+u_0+u_{-1}S^{-1}+\dots
\,\in\mc V((S^{-1}))
\,,
$$
and 
$$
\bar L(S)=\bar u_{1}S^{-1}+\bar u_0+\bar u_{-1}S+\dots
\,\in\mc V((S))
\,.
$$
The equations extend the discrete KP hierarchy \eqref{L:KP2} as follows:
\begin{equation}\label{eq:UT}
\begin{split}
&
\frac{d L(z)}{dt_n}=[(L^{n})_+,L](z)
\,\,,\,\,\,\,
\frac{d \bar L(z)}{dt_n}=[(L^{n})_+,\bar L](z)
\,,\\
&
\frac{d L(z)}{d\bar t_n}=[(\bar L^{n})_-,L](z)
\,\,,\,\,\,\,
\frac{d \bar L(z)}{d\bar t_n}=[(\bar L^{n})_-,\bar L](z)
\,.
\end{split}
\end{equation}
Carlet computed the three compatible Poisson brackets of \cite{OR89}
for this hierarchy \cite{Carlet2}.
Here we present them in the equivalent language of $\lambda$-brackets.

The first mPVA $\lambda$-bracket is
\begin{equation}\label{20180629:eq4}
\begin{split}
\{ {u_i}_\lambda u_j\}_1&
=(1-\epsilon_{i}-\epsilon_j)\big((\lambda S)^{-i}-\lambda^j)u_{i+j}
\,,
\\
\{ {u_i}_\lambda \bar u_j\}_1&
=\big(\lambda^{-j}-(\lambda S)^{-i}\big)\big(\epsilon_{j+1}u_{i-j}-\epsilon_{i}\bar u_{j-i}\big)
\,,
\\
\{ \bar u_i{}_\lambda \bar u_j\}_1&
=(1-\epsilon_{-i}-\epsilon_{-j})\big(\lambda^{-j}-(\lambda S)^{i})\bar u_{i+j}
\,,
\end{split}
\end{equation}
where $\epsilon_{i}=1$ if $i\leq0$ and $\epsilon_i=0$ if $i\geq1$, and
in the RHS we assume that
\begin{equation}\label{eq:10.16}
u_1=1 \,\,\,\hbox{and}\,\,\, u_k=\bar u_k=0\,\,\, \hbox{if}\,\,\, k>1.
\end{equation}.
The second $\lambda$-bracket is
\begin{equation}\label{20180629:eq5}
\begin{split}
\{{u_i}_\lambda{u_j}\}_2
&=
\sum_{n=0}^{1-i}
\big(
u_{j-n}(\lambda S)^{j-i-n}u_{i+n}
-
u_{i+n}(\lambda S)^{n}u_{j-n}
\big)
\\
&-
u_j\big((\lambda S)^1-(\lambda S)^j\big)\big(\lambda S-1\big)^{-1}\big((\lambda S)^{-i}-1\big)u_i
\,,
\\
\{{u_i}_\lambda{\bar u_j}\}_2=
&
\sum_{n=0}^{\min\{1-i,1-j\}}
\big(
\bar u_{j+n}(\lambda S)^{-j-n} u_{i+n}\lambda ^n
-(\lambda S)^n\bar u_{j+n}(\lambda S)^{-i-n} u_{i+n}
\big)
\\
&+
\bar u_j\big((\lambda S)^{-j+1}-1\big)\big(\lambda S-1\big)^{-1}\big((\lambda S)^{-i}-1\big) u_i
\,,
\\
\{\bar u_i{}_\lambda{\bar u_j}\}_2=
&
\sum_{n=0}^{1-j}
\big(
\bar u_{i-n}(\lambda S)^{n}\bar u_{j+n}
-\bar u_{j+n}(\lambda S)^{i-j-n}\bar u_{i-n}
\big)
\\
&+
\bar u_j\big((\lambda S)^{-j+1}-1\big)\big(\lambda S-1\big)^{-1}\big((\lambda S)^{i}-1\big) \bar u_i
\,,
\end{split}
\end{equation}
subject to \eqref{eq:10.16}.
%
We do not give here the third $\lambda$-bracket since it is non-local.
These $\lambda$-brackets can be derived by using the theory of Adler type operators,
discussed in the previous sections.

Define the Hamiltonian functionals 
$$
h_p=-\frac1{p+1}\tint\mres L^{p-1}
\,\,,\,\,
\bar h_p=-\frac1{p+1}\tint\mres \bar L^{p-1}
\,\,\text{ for }\,\,
p\geq 1
\,\,,\,\,
h_0=\bar h_0=0
\,.
$$
According to the Oevel-Ragnisco theory \cite{OR89},
these Hamiltonian functionals are integrals of motion in involution
with respect to both $\lambda$-brackets
and the hierarchy \eqref{eq:UT}
has the following bi-Hamiltonian representation
(cf. Theorem \ref{thm:hn}):
\begin{equation}\label{eq:UTz}
\begin{split}
&
\frac{d f}{dt_n}
=
\{{h_n}_\lambda f\}_1\big|_{\lambda=1}
=
\{{h_{n-1}}_\lambda f\}_2\big|_{\lambda=1}
\,, \\
&
\frac{d f}{d\bar t_n}
=
\{{{\bar h}_n}\,_\lambda f\}_1\big|_{\lambda=1}
=
\{{\bar h}_{n-1}\,_\lambda f\}_2\big|_{\lambda=1}
\,,\,\, n\geq1\,.
\end{split}
\end{equation}

\end{document}